\documentclass[11pt]{amsart}
\title{A convenient category of parametrized spectra}
\author{Cary Malkiewich}
\address{Department of Mathematics, Binghamton University}
\email{malkiewich@math.binghamton.edu}

\usepackage{amsmath,amssymb,amsthm,array,color,multirow,pdflscape,mathrsfs,subcaption}
\usepackage[all]{xypic}
\usepackage{tikz,tikz-cd}
\usepackage{imakeidx}
\makeindex[columns=2, title=Index] 
\newdir{ >}{{}*!/-7pt/@{>}}

\usepackage[margin=1.25in]{geometry}
\newcommand{\beforesubsection}{\vspace{1em}}
\newcommand{\aftersubsection}{}

\makeatletter
\def\l@section{\@tocline{1}{0pt}{1pc}{}{}}
\def\l@subsection{\@tocline{2}{0pt}{1pc}{4.6em}{}}
\def\l@subsubsection{\@tocline{3}{0pt}{1pc}{7.6em}{}}
\renewcommand{\tocsection}[3]{%
	\indentlabel{\@ifnotempty{#2}{\makebox[2.3em][l]{%
				\ignorespaces#1 #2.\hfill}}}#3}
\renewcommand{\tocsubsection}[3]{%
	\indentlabel{\@ifnotempty{#2}{\hspace*{2.3em}\makebox[2.3em][l]{%
				\ignorespaces#1 #2.\hfill}}}#3}
\renewcommand{\tocsubsubsection}[3]{%
	\indentlabel{\@ifnotempty{#2}{\hspace*{4.6em}\makebox[3em][l]{%
				\ignorespaces#1 #2.\hfill}}}#3}
\makeatother

\usepackage{xcolor}
\definecolor{darkgreen}{rgb}{0,0.30,0} 
\definecolor{darkred}{rgb}{0.75,0,0}
\definecolor{darkblue}{rgb}{0,0,0.6} 
\usepackage[pdfborder=0,pagebackref,colorlinks,citecolor=darkgreen,linkcolor=darkgreen,urlcolor=darkblue]{hyperref}
\renewcommand*{\backref}[1]{}
\renewcommand*{\backrefalt}[4]{({%
    \ifcase #1 Not cited.%
          \or On p.~#2%
          \else On pp.~#2%
    \fi%
    })}

\def\makeautorefname#1#2{\expandafter\def\csname#1autorefname\endcsname{#2}}
\makeautorefname{equation}{Equation}%
\makeautorefname{footnote}{footnote}%
\makeautorefname{item}{item}%
\makeautorefname{figure}{Figure}%
\makeautorefname{table}{Table}%
\makeautorefname{part}{Part}%
\makeautorefname{appendix}{Appendix}%
\makeautorefname{chapter}{Chapter}%
\makeautorefname{section}{Section}%
\makeautorefname{subsection}{Section}%
\makeautorefname{subsubsection}{Section}%
\makeautorefname{paragraph}{Paragraph}%
\makeautorefname{subparagraph}{Paragraph}%
\makeautorefname{theorem}{Theorem}%
\makeautorefname{thm}{Theorem}%
\makeautorefname{addm}{Addendum}%
\makeautorefname{mainthm}{Main theorem}%
\makeautorefname{corollary}{Corollary}%
\makeautorefname{cor}{Corollary}%
\makeautorefname{lemma}{Lemma}%
\makeautorefname{lem}{Lemma}%
\makeautorefname{sublemma}{Sublemma}%
\makeautorefname{sublem}{Sublemma}%
\makeautorefname{subl}{Sublemma}%
\makeautorefname{prop}{Proposition}%
\makeautorefname{property}{Property}
\makeautorefname{pro}{Property}
\makeautorefname{sch}{Scholium}%
\makeautorefname{step}{Step}%
\makeautorefname{conject}{Conjecture}%
\makeautorefname{conj}{Conjecture}%
\makeautorefname{questn}{Question}
\makeautorefname{quest}{Question}
\makeautorefname{qn}{Question}
\makeautorefname{definition}{Definition}%
\makeautorefname{defn}{Definition}%
\makeautorefname{defi}{Definition}%
\makeautorefname{def}{Definition}%
\makeautorefname{dfn}{Definition}%
\makeautorefname{df}{Definition}%
\makeautorefname{notation}{Notation}
\makeautorefname{notn}{Notation}
\makeautorefname{rem}{Remark}%
\makeautorefname{rems}{Remarks}%
\makeautorefname{rmk}{Remark}%
\makeautorefname{rk}{Remark}%
\makeautorefname{remarks}{Remarks}%
\makeautorefname{rems}{Remarks}%
\makeautorefname{rmks}{Remarks}%
\makeautorefname{rks}{Remarks}%
\makeautorefname{example}{Example}%
\makeautorefname{examp}{Example}%
\makeautorefname{exmp}{Example}%
\makeautorefname{exam}{Example}%
\makeautorefname{exa}{Example}%
\makeautorefname{ex}{Example}%
\makeautorefname{axiom}{Axiom}%
\makeautorefname{axi}{Axiom}%
\makeautorefname{ax}{Axiom}%
\makeautorefname{case}{Case}%
\makeautorefname{claim}{Claim}%
\makeautorefname{clm}{Claim}%
\makeautorefname{assumpt}{Assumption}%
\makeautorefname{asses}{Assumptions}%
\makeautorefname{conclusion}{Conclusion}%
\makeautorefname{concl}{Conclusion}%
\makeautorefname{conc}{Conclusion}%
\makeautorefname{cond}{Condition}%
\makeautorefname{const}{Construction}%
\makeautorefname{con}{Construction}%
\makeautorefname{criterion}{Criterion}%
\makeautorefname{criter}{Criterion}%
\makeautorefname{crit}{Criterion}%
\makeautorefname{exercise}{Exercise}%
\makeautorefname{exer}{Exercise}%
\makeautorefname{exe}{Exercise}%
\makeautorefname{warn}{Warning}%
\newtheorem{thm}{Theorem}[subsection]
\newtheorem{cor}{Corollary}[subsection]
\newtheorem{lem}{Lemma}[subsection]
\newtheorem{prop}{Proposition}[subsection]
\theoremstyle{definition}
\newtheorem{df}{Definition}[subsection]
\newtheorem{ex}{Example}[subsection]
\newtheorem{warn}{Warning}[subsection]
\newtheorem{rmk}{Remark}[subsection]
\newtheorem{notn}{Notation}[subsection]
\numberwithin{equation}{subsection}
\numberwithin{figure}{subsection}

\makeatletter
\let\c@cor=\c@thm
\let\c@prop=\c@thm
\let\c@lem=\c@thm
\let\c@df=\c@thm
\let\c@ex=\c@thm
\let\c@warn=\c@thm
\let\c@rmk=\c@thm
\let\c@notn=\c@thm
\let\c@equation\c@thm
\let\c@figure\c@thm
\let\c@table\c@thm
\makeatother

\hypersetup{linktoc=all}

\newcommand{\sC}{\mathscr{C}}

\newcommand{\bA}{\mathbf{A}}

\newcommand{\bC}{\mathbf{C}}
\newcommand{\bD}{\mathbf{D}}

\newcommand{\bH}{\mathbf{H}}

\newcommand{\bS}{\mathbf{S}}
\newcommand{\bT}{\mathbf{T}}

\newcommand{\D}{\mathbb D}
\renewcommand{\L}{\mathbb L}
\newcommand{\M}{\mathbb M}
\newcommand{\N}{\mathbb N}
\newcommand{\R}{\mathbb R}

\newcommand{\Z}{\mathbb Z}

\newcommand{\Sph}{\mathbb S}
\newcommand{\mc}{\mathcal}
\newcommand{\cat}[1]{\textup{\textbf{{#1}}}}

\newcommand{\ra}{\longrightarrow}

\newcommand{\simar}{\overset\sim\longrightarrow}

\newcommand{\Hom}{\textup{Hom}}
\newcommand{\Map}{\textup{Map}}

\newcommand{\id}{\textup{id}}
\newcommand{\colim}{\textup{colim}\,}
\newcommand{\hocolim}{\textup{hocolim}\,}

\newcommand{\Cyl}{\textup{Cyl}}
\newcommand{\sma}{\wedge}
\newcommand{\barsmash}{\,\overline\wedge\,}
\newcommand{\barsma}[1]{\,\overline\wedge_{#1}\,}
\DeclareMathOperator{\barmap}{\underline{\smash{\textup{Map}}}}
\DeclareMathOperator{\barF}{\underline{\smash{F}}}

\newcommand{\sk}{\textup{Sk}}

\newcommand{\sh}{\textup{sh}}

\DeclareMathOperator{\ho}{\textup{Ho}}

\newcommand{\Ex}{\mathcal{E}x}
\newcommand{\bcr}[3]{\left[{#1}\xrightarrow{#2}{#3}\right]}

\newcommand{\shad}[1]{{\ensuremath{\hspace{1mm}\makebox[-1mm]{$\langle$}\makebox[0mm]{$\langle$}\hspace{1mm}{#1}\makebox[1mm]{$\rangle$}\makebox[0mm]{$\rangle$}}}}
\DeclareMathOperator{\ob}{\textup{ob}}
\newcommand{\Osp}{\mathcal{OS}}
\newcommand{\Psp}{\mathcal{PS}}
\newcommand{\adj}{\dashv}
\newcommand{\Th}{\textup{Th}}
\newcommand{\non}{\textup{non}}
\newcommand{\po}[2]{\ar@{}@<{#2}>[rd]|({#1})*\txt{\Large $\ulcorner$}}
\newcommand{\pb}[2]{\ar@{}@<{#2}>[rd]|({#1})*\txt{\Large $\lrcorner$}}

\begin{document}

\maketitle

\begin{abstract}
	We describe a point-set category of parametrized orthogonal spectra, a model structure on this category, and a separate, more geometric class of cofibrant-and-fibrant objects. The structures we describe are ``convenient'' in that they are preserved by the most common operations. They allow us to reduce sophisticated statements about the homotopy category to straightforward claims at the point-set level.
	
	We use this framework to give a construction of the bicategory of parametrized spectra, one that is far more direct than earlier approaches. This gives a clean bridge between the concrete index theory pioneered by Dold, and the formal bicategorical theory developed by May and Sigurdsson, Ponto, and Shulman.

	This article is a condensation of the more expository preprint ``Parametrized spectra, a low-tech approach.''
\end{abstract}

\setcounter{tocdepth}{2}
\tableofcontents

\parskip 2ex

\section{Introduction}

Parametrized spectra are an essential part of the foundations of homotopy theory. They were first introduced in order to systematically study the Becker-Gottlieb transfer \cite{becker1975transfer,becker1976transfer,clapp1984homotopy,crabb_james}, but they also occur naturally in fixed-point theory and intersection theory \cite{ponto_asterisque,klein_williams}, twisted $K$-theory \cite{atiyah_segal,abg_ktheory,fht_1,hebestreit2019homotopical}, differential cohomology \cite{braunack2018rational}, computations with Thom spectra, e.g. \cite{mahowald1977new,units_rigid,hahn2018eilenberg,klang2018factorization}, the functor calculus of Goodwillie \cite{calc3}, the algebraic $K$-theory of spaces of Waldhausen \cite{waldhausen_1,waldhausen_2}, symplectic geometry and string topology \cite{kragh_parametrized_nearby}, and the index theory of Dwyer, Weiss, and Williams \cite{dww,klein_williams_bundle}.

Unfortunately, parametrized spectra have always presented challenges that go significantly beyond those encountered with ordinary spectra. This is especially true if one is interested in simultaneously developing a point-set category and a homotopy category (or $\infty$-category) of parametrized spectra, and keeping a close relationship between the two. For applications to manifolds and to fixed-point theory, the link between these two is essential.

Let us describe two of the issues in detail. First, parametrized spectra over varying base spaces form both a symmetric monoidal bifibration (SMBF), and a bicategory. So they have an external smash product $\barsmash$, pullbacks $f^*$, pushforwards $f_!$, composition products $\odot$, and so on. Each of these structures involves composing left adjoints and right adjoints with each other, sometimes three at a time (left, then right, then left). This presents difficulties for any theoretical framework attempting to capture this structure on both the point-set level and in the homotopy category.

Second, the most obvious model structure for parametrized orthogonal spectra, the ``$q$-model structure,'' has run into considerable difficulty in its development. Through the work of May and Sigurdsson, culminating in the volume \cite{ms}, a workaround is introduced that has become standard. Rather than prove that the expected ``$q$-model structure,'' May and Sigurdsson produce a ``$qf$-model structure'' with the expected homotopical properties, at the expense of added complexity. Other model structures have arisen since then, but none capture the simplicity and efficiency of the $q$-model structure -- it is the most obvious adaptation of the standard stable model structure from \cite{mmss} and \cite{mandell2002equivariant} to parametrized spectra.

The purpose of this paper is to address these two problems.

\begin{thm}\label{thm:intro_q}(\autoref{thm:stable_model_structure})
	The $q$-model structure on parametrized orthogonal spectra exists.
\end{thm}
Furthermore this model structure has all of the expected homotopical behavior -- in the language of \cite{ms}, it is ``well-grounded.'' Note that \cite{hebestreit_sagave_schlichtkrull} uses a more formal approach to establish a similar model structure for parametrized symmetric spectra. We expect that their approach could be adapted to give a different proof of \autoref{thm:intro_q}.

This model structure is essential to our work, but by itself, it is not convenient enough. When composing left and right adjoints, one often needs a more geometric class of ``cofibrant and fibrant objects,'' such that the adjoints preserve the cofibrant/fibrant objects and the equivalences between them.
\begin{thm}\label{thm:intro_cof_fib}
	There is a notion of ``cofibrant'' and ``fibrant'' for parametrized orthogonal spectra with the following properties.
	\begin{itemize}
		\item Every spectrum $X$ is equivalent by a zig-zag
		\[ X \overset\sim\leftarrow QX \overset\sim\to PQX \]
		to a spectrum $PQX$ that is cofibrant and fibrant.
		\item Every cofibrant spectrum $X$ has a natural cofibrant-and-fibrant replacement $PX$. Furthermore $P$ commutes with the external smash product, $PX \barsmash PY \cong P(X \barsmash Y)$.
		\item The external smash product $\barsmash$ preserves cofibrant spectra, equivalences of cofibrant spectra, and spectra that are both cofibrant and fibrant.
		\item For a map $f\colon A \to B$, the pullback $f^*$ preserves cofibrant spectra, fibrant spectra, and equivalences of fibrant spectra.
		\item For a map $f\colon A \to B$, the pushforward $f_!$ preserves cofibrant spectra and equivalences between them. It also preserves fibrant spectra if $f$ is a Hurewicz fibration.
	\end{itemize}
\end{thm}
The two notions are called ``freely $f$-cofibrant'' and ``level $h$-fibrant,'' see \autoref{sec:convenient_definitions}. This theorem is a collection of results from Sections \ref{sec:spectra} and \ref{sec:stable}, the two most significant being \autoref{prop:spectra_pushout_product} and \autoref{prop:stably_derived}. This is the most technically involved part of the paper, as it relies on both the $q$-model structure of \autoref{thm:stable_model_structure} and the work contained in \S\ref{sec:reedy}.

\autoref{thm:intro_cof_fib} is a significant improvement on the existing literature. The analogous class of ``excellent prespectra'' from \cite{ms} requires a longer zig-zag of equivalences with weaker monoidal properties, see \cite[13.5.2]{ms}. As a result, most work on homotopical operations on spectra (such as $\barsmash$ and $\odot$) has needed to rely on a formally black-boxed version of the category of parametrized spectra, and not the relatively simple point-set category. By contrast, \autoref{thm:intro_cof_fib} makes it possible to work with the homotopical versions of $\barsmash$ and $\odot$ directly from their point-set versions.

As an application, we use \autoref{thm:intro_cof_fib} to give a new construction of the homotopy bicategory of parametrized spectra $\Ex$ (\autoref{thm:four_bicategories_of_spectra}). Namely, we build a point-set version of $\Ex$, restrict to the cofibrant-fibrant spectra, and observe that the bicategory operations preserve both the cofibrant-fibrant spectra and the equivalences between them. Therefore they define a coherent set of operations on the homotopy category as well.

This definition of $\Ex$ is significantly shorter and simpler than existing definitions. It allows us to directly express abstract bicategorical traces as explicit point-set collapse maps, without having to pass through chain complexes and simplicial homology. This is a drastic simplification of the existing theory, and is especially important because of applications to fixed-point theory and algebraic $K$-theory, see for instance \cite{ponto_asterisque,ponto_shulman_indexed,ponto_shulman_general,ponto_shulman_mult,lind2019transfer,campbell_ponto,mp1}.

The same technique also gives a clean description of the homotopy category of all spectra $\ho\Osp$ as a symmetric monoidal bifibration (SMBF), see \autoref{thm:spectra_SMBF}. Equivalently, it gives a strong symmetric monoidal structure to the pullback functors $f^*$, in a way that is coherent along compositions of pullbacks. This construction of $\ho\Osp$ is in fact a broader result, because the bicategory $\Ex$ can be recovered from it by a formal argument, see \cite{shulman_framed_monoidal,mp1}.

The following rigidity result is helpful for the construction of $\Ex$ and $\ho\Osp$. The argument is a generalization of the one used in the author's thesis \cite[1.2, 3.17]{malkiewich_cyclotomic_dx}.
\begin{thm}\label{thm:intro_rigidity}(\autoref{thm:spectra_rigidity})
	Any functor of the form $f_!g^*(X_1 \barsmash \ldots \barsmash X_n)$ is rigid, i.e. the only automorphism is the identity, provided that the product map $(f,g)$ is injective.
\end{thm}

When a functor $F$ is rigid in this sense, any diagram of natural isomorphisms between functors isomorphic to $F$, automatically commutes. As a result, \autoref{thm:intro_rigidity} greatly simplifies the verification of the coherence axioms for $\Ex$ and $\ho\Osp$ on the point-set level.

Finally, the ``convenient'' notions of cofibrant and fibrant from \autoref{thm:intro_cof_fib} do not fit into a model structure. So we can use them to derive the operations $\barsmash$, $f^*$ and $f_!$, but the commutation of the derived functors cannot be proven by restricting to cofibrant-fibrant objects and using Whitehead's theorem. Instead, we use a new framework for passing an isomorphism of composites of point-set functors to an isomorphism of their derived functors, in a way that respects composition:
\begin{prop}\label{thm:intro_deformable}(\autoref{prop:passing_natural_trans_to_derived_functors})
	Any isomorphism
	\[ F_n \circ \ldots \circ F_1 \cong G_k \circ \ldots \circ G_1 \]
	between composites of left- and right-deformable functors induces in a canonical way an equivalence between the composites of their derived functors, provided the two lists satisfy a certain condition (see \autoref{sec:composing_comparing}).
\end{prop}

The proof of this is almost trivial -- it is more of a perspective than a result. Still, the perspective is a powerful and flexible one. It does not require any additional assumptions about which of the functors are left-deformable and which are right-deformable, as in \cite{shulman_comparing} or in the usual framework of left and right Quillen functors. This makes it ideal for the current work, in which we frequently encounter left-then-right-then-left deformable functors. The most technical result we prove with this perspective is that passing to homotopy categories preserves categorical bifibrations under certain assumptions (\autoref{prop:deform_bifibration}).

\textbf{Outline.}
Section \ref{sec:spaces} recalls preliminaries on spaces and sets up the symmetric monoidal bifibration (SMBF) $\mc R$ of all retractive spaces over all base spaces. Section \ref{sec:derived} introduces the framework for deriving composites of functors and proves \autoref{thm:intro_deformable}. Section \ref{sec:spectra} introduces parametrized spectra, level equivalences, and the convenient notions of cofibrant and fibrant, and performs much of the technical work needed for \autoref{thm:intro_cof_fib}. Section \ref{sec:stable} defines stable equivalences, proves that the $q$-model structure exists (\autoref{thm:intro_q}), then finishes the proof of  \autoref{thm:intro_cof_fib}. Section \ref{sec:structures} gives the new construction of the symmetric monoidal bifibration $\ho\Osp$, the bicategory $\Ex$, and the fiberwise versions of these over a base space $B$. We also indicate applications to the Reidemeister trace $R(f)$, but this topic is treated in much more depth in the companion expository paper \cite{malkiewich2019parametrized}. Finally, in \autoref{sec:G} we indicate the changes needed to make the definitions and proofs equivariant with respect to a compact Lie group $G$.

\textbf{Acknowledgements.}
The author would like to thank Fabian Hebestreit, John Lind, Peter May, Kate Ponto, Steffen Sagave, and Michael Shulman for many conversations about parametrized spectra over the years. He thanks the Hausdorff Institute in Bonn for their hospitality in July of 2017 when he started working on the $q$-model structure, and the Max Planck Institute in Bonn for their hospitality in the spring and fall of 2018, during which most of the material in this paper was first written. He would like to thank David Carchedi, Madgalena Kedziorek, Inbar Klang, Mona Merling, Sean Tilson, and Sarah Yeakel for personal support and encouragement while the first version of this paper was being written. Finally, he would like to thank Thomas Brazelton, Thomas Blom, and Steffen Sagave for their comments on earlier versions of this paper. The author is supported by NSF grants DMS-2005524 and DMS-2052923.

\section{Parametrized spaces}\label{sec:spaces}

In this first section, we recall the definitions and results on parametrized spaces that we will build on later when we define parametrized spectra. Most of this is known, though the rigidity theorem of \autoref{prop:spaces_rigidity} seems to be new.

\subsection{Basic definitions and lemmas}\aftersubsection

Let $B$ be a topological space. A \textbf{retractive space}\index{retractive space} or ``ex-space'' over $B$ consists of a space $X$, called the total space, along with an inclusion and a projection map
\[ B \overset{i}\ra X \overset{p}\ra B \]
that compose to the identity of $B$. It is helpful to think of $X$ as a family of spaces $X_b = p^{-1}(b)$, each of which has a basepoint $i(b) \in X_b$. For this reason we often call $i$ the \textbf{basepoint section}\index{basepoint section} of $X$.

\begin{notn}
If $X \to B$ has no basepoint section, we define $X_{+B}$ to be the disjoint union $X \amalg B$, with projection given by $p$ and $\id_B$, and with the canonical inclusion of $B$.
\end{notn}

As usual, $B$ and $X$ are not plain-vanilla topological spaces. Recall that $X$ is \textbf{compactly generated} when its closed sets are determined by their preimages in every compact Hausdorff space $K$ mapping into $X$. Furthermore $X$ is \textbf{weak Hausdorff} if the image of every such $K \to X$ is closed. See \cite[App A]{lewis_appendix}, \cite{strickland2009category}, and \cite[App A]{schwede_global} for basic properties of such spaces. In this paper we adopt the following two conventions.
\begin{itemize}
	\item \textbf{(CGWH)}\index{CGWH} Every base space $B$ and total space $X$ is compactly generated weak Hausdorff.\index{(CGWH)}
	\item \textbf{(CG)}\index{CG} Every base space $B$ is compactly generated weak Hausdorff, but the total space $X$ only has to be compactly generated.\index{(CG)}
\end{itemize}
Most of the theory is exactly the same in both cases. At the points where it is different, we will use the tags (CGWH) and (CG) to distinguish what happens in each case. The paper can therefore be read using either convention (they give Quillen equivalent model categories at the end). Note that \cite{ms} is written in (CG).

Retractive spaces over $B$ form a category $\mc R(B)$. The morphisms or ``ex-maps'' from $(B,X,i_X,p_X)$ to $(B,Y,i_Y,p_Y)$ are the continuous maps $f\colon X \to Y$ commuting with the inclusion maps $i$ and projection maps $p$. In other words, $f \circ i_X = i_Y$ and $p_Y \circ f = p_X$, or the following diagram commutes.\index{retractive space!morphisms of}
\[ \xymatrix @R=1.2em{
	& B \ar[ld]_-{i_X} \ar[rd]^-{i_Y} & \\
	X \ar[rr]^-f \ar[rd]_-{p_X} && Y \ar[ld]^-{p_Y} \\
	& B &
} \]

The category $\mc R(B)$ has a zero object, and all colimits and limits. Note that the colimits do not commute with the functor that forgets $B$. For instance, the coproduct of $X$ and $Y$ is the pushout $X \cup_B Y$, not the disjoint union. Similarly, the product is the fiber product $X \times_B Y$. Pushouts and pullbacks do commute with the forgetful functor, as do all colimits and limits over diagrams whose indexing categories are connected.

\begin{df}
A map $f\colon X \to Y$ of retractive spaces over $B$ is:
\begin{itemize}
	\item a \textbf{weak equivalence}\index{weak equivalence} (or $q$-equivalence) if the map of total spaces $X \to Y$ is a weak homotopy equivalence.
	\item a \textbf{Hurewicz fibration} or \textbf{$h$-fibration}\index{$h$-fibration} if $X \to Y$ has the homotopy lifting property, in other words the projection map $X^I \to X \times_Y Y^I$ has a section.
	\item a \textbf{Serre fibration} or \textbf{$q$-fibration}\index{$q$-fibration} if $X \to Y$ satisfies the usual lifting property with respect to cylinders on discs:
	\[ \xymatrix @R=1.7em{
		D^n \times \{0\} \ar[r] \ar[d] & X \ar[d] \\
		D^n \times I \ar[r] \ar@{-->}[ur] & Y
	}\]
	\item a \textbf{Quillen cofibration} or \textbf{$q$-cofibration}\index{$q$-cofibration} if $X \to Y$ is a retract of a relative cell complex.
	\item a \textbf{closed inclusion} or \textbf{$i$-cofibration}\index{$i$-cofibration} if $X \to Y$ is a homeomorphism onto its image, which is closed in $Y$. (In (CGWH), the basepoint section is always a closed inclusion, so every space is $i$-cofibrant.)
	\item an \textbf{$h$-cofibration}\index{$h$-cofibration} if it is a closed inclusion, and $X \to Y$ has the homotopy extension property. In other words, the inclusion
\[ (X \times I) \cup_{(X \times \{0\})} (Y \times \{0\}) \to (Y \times I) \]
has a retract. (In (CGWH) we are allowed to drop ``closed inclusion'' from the definition because it follows from the homotopy extension property.)
	\item an \textbf{$f$-cofibration}\index{$f$-cofibration} if it is a closed inclusion and has the fiberwise homotopy extension property.\footnote{In \cite{ms} these are called $\bar{f}$-cofibrations; in \cite{crabb_james} they are called closed fibrewise cofibrations.} This means the inclusion
\[ (X \times I) \cup_{(X \times \{0\})} (Y \times \{0\}) \to (Y \times I) \]
has a retract \emph{that respects} the projection to $B$.
\end{itemize}
We say that a retractive space is $h$-fibrant if the map to the zero object is an $h$-fibration, and similarly for the other notions of fibrant and cofibrant.
\end{df}

As usual, the map $X \to Y$ is an $h$-cofibration iff we can identify $X$ as a closed subspace of $Y$ and give $(Y,X)$ the structure of an NDR-pair (see e.g. \cite[\S 6.4]{concise}). This is a function $u: Y \to [0,1]$ with $X = u^{-1}(0)$, and a homotopy $h\colon Y \times I \to Y$, constant on $X$, from the identity map of $Y$ to a map that retracts the open neighborhood $U = u^{-1}([0,1))$ back onto $X$. Similarly, the map is an $f$-cofibration iff we can find an NDR-pair structure on $(Y,X)$ such that the homotopy $h$ respects the projection to $B$.

The following chart summarizes all of the classes of maps we have defined.
\[ \xymatrix @R=.5em @C=2em{
	h\textup{-fibrations} \ar@{=>}[r] & q\textup{-fibrations} & \textup{weak equivalences} \\
	f\textup{-cofibrations} \ar@{=>}[r] & h\textup{-cofibrations} \ar@{=>}[r] & \textup{closed inclusions} \\
	& q\textup{-cofibrations} \ar@{=>}[u] &
} \]

\begin{ex}\label{ex:fiberwise_thom_space}
	Let $V \to B$ be an $n$-dimensional real vector bundle with an inner product, and $S(V)$ and $D(V)$ its associated unit sphere and disc bundles. The \textbf{Thom space} $\Th(V)$ is the quotient $D(V)/S(V)$. The \textbf{fiberwise Thom space}\index{fiberwise!Thom space $\Th_B(V)$} $\Th_B(V)$ is the fiberwise quotient, i.e. the pushout
	\[ \xymatrix @R=2em{
		S(V) \ar[r] \ar[d] \po{.65}{-1ex} & D(V) \ar[d] \\
		B \ar[r] & \Th_B(V).
	} \]
	So $\Th_B(V)$ is a sphere bundle over $B$ with a section, and is always $f$-cofibrant and $h$-fibrant.
\end{ex}


We collect some standard results here for future reference.

\begin{prop}\label{prop:q}
	 The weak homotopy equivalences, Quillen cofibrations, and Serre fibrations make $\mc R(B)$ into a model category. It is proper and cofibrantly generated by the maps
	 \[ \begin{array}{ccrll}
	I &=& \{ \ S^{n-1}_{+B} \to D^n_{+B} & : n \geq 0, & D^n \to B \ \} \\
	J &=& \{ \ D^n_{+B} \to (D^n \times I)_{+B} &: n \geq 0, & (D^n \times I) \to B \ \}.
	\end{array} \]
\end{prop}

This is true in both (CGWH) and (CG), and the forgetful functor (CGWH) $\to$ (CG) is a right Quillen equivalence.

\begin{prop}\label{prop:technical_cofibrations}\hfill
	\begin{enumerate}
		\item Each notion of ``cofibration'' in the category $\mc R(B)$ is closed under pushout, transfinite compositions (therefore also coproducts), and retracts.
		\item (Gluing lemma) Pushouts in $\mc R(B)$ along an $h$-cofibration are homotopical. (Equivalences of pushout diagrams go to equivalences of pushouts.)
		\item (Colimit lemma) Sequential colimits in $\mc R(B)$ along $h$-cofibrations are homotopical.
	\end{enumerate}
\end{prop}

The dual of this proposition also holds with $h$-fibrations or with $q$-fibrations.

\begin{cor}\label{cofibration_of_pushouts}
	For each notion of ``cofibration,'' a map of pushout diagrams
	\[ \xymatrix @R=1.7em{
			Y \ar[d] & X \ar[l] \ar[r] \ar[d] & Z \ar[d] \\
			Y' & X' \ar[l] \ar[r] & Z',
		} \]
	induces a cofibration on the pushouts, provided that both $Z \to Z'$ and $Y \cup_X X' \to Y'$ are cofibrations.
\end{cor}

\begin{lem}(\cite[\S 6.4]{concise}, \cite[Thm 6]{strom_2})\label{lem:product_ndr}
	If $A \subseteq X$ and $A' \subseteq X'$ are closed NDR-pairs then the pushout-product map
	\[ \xymatrix{ (A \times X') \cup_{A \times A'} (X \times A) \ar[r] & X \times X' } \]
	also has the structure of a closed NDR-pair. The same is true for fiberwise closed NDR-pairs over $B$, or for just closed inclusions.
\end{lem}

\begin{lem}\label{lem:dold}
	If $X \to Y$ is an $f$-cofibration of $h$-fibrant spaces over $B$, then the path-lifting function for $Y$ can be chosen so that it restricts to a path-lifting function of $X$.
\end{lem}

\begin{prop}[Clapp]\cite[Prop 1.3]{clapp1981duality}\label{prop:clapp}
	Given a pushout square of spaces over $B$
	\[ \xymatrix @R=1.7em{
		X \ar[r]^-i \ar[d]^-f \po{.65}{-1ex} & Y \ar[d] \\
		Z \ar[r] & Y \cup_X Z,
	} \]
	if $X$, $Y$, and $Z$ are all $h$-fibrant and $i$ is an $f$-cofibration then the pushout $Y \cup_X Z$ is also $h$-fibrant.
\end{prop}

\begin{lem}[Heath and Kamps]\cite[1.3]{heath_kamps}\label{lem:heath_kamps}
	An $h$-cofibration $A \to X$ between $h$-fibrant spaces over $B$ is an $f$-cofibration.
\end{lem}

\begin{lem}\label{lem:closed_inclusion_equalizer}
	(CGWH) The map $f\colon X \to Y$ in $\mc R(B)$ is a closed inclusion iff it is the equalizer of some pair of maps $Y \rightrightarrows Z$.\footnote{In (CG), the equalizer of two maps is only an inclusion. Ordinary inclusions are not as well-behaved as closed inclusions, for instance inclusions into $X$ are not preserved by pushouts along their intersection.}
\end{lem}

\autoref{prop:clapp} is a technical cornerstone for us. It says that fibrations can be glued together to form new fibrations. And by \autoref{lem:heath_kamps}, it also applied with ``$f$-cofibration'' replaced by ``$h$-cofibration.'' Of course, the set of maps to which the result applies has not been enlarged, it is merely easier to check that a given map is in that class.

\beforesubsection
\subsection{Base-change functors $f_!$, $f^*$, and $f_*$}\label{sec:base_change}\aftersubsection

If $f\colon A \to B$ is any map of spaces, for each retractive space $X$ over $B$ we can form the pullback $f^*X = A \times_B X$ and fit it into a diagram as indicated:
\[ \xymatrix @R=1.7em @C=3em{
	A \ar@{-->}[d] \ar@/_2em/[dd]_-{\id} \ar[r]^-f & B \ar[d]^-{i_X} \\
	f^*X \ar[r] \ar[d] \pb{.2}{0ex} & X \ar[d]^-{p_X} \\
	A \ar[r]^-f & B
} \]
This makes $f^*X$ into a retractive space over $A$. It is not hard to check this defines a functor, the \textbf{pullback functor}\index{pullback $f^*$} $f^*\colon \mc R(B) \to \mc R(A)$.
\begin{lem}\label{lem:f_star_preserves}
	The pullback functor $f^*$ preserves \vspace{-.5em}
	\begin{itemize}
		\item $h$-fibrations and $q$-fibrations,
		\item weak equivalences between $q$-fibrant (or $h$-fibrant) spaces,
		\item $f$-cofibrations, and closed inclusions.
	\end{itemize}
	Moreover if $f$ itself is a Hurewicz fibration ($h$-fibration), then $f^*$ preserves $h$-cofibrations, and all weak equivalences.
\end{lem}

\begin{proof}
	Most of these are straightforward exercises. The part with closed inclusions uses \cite[A.10.1]{lewis_appendix} and the part with $h$-cofibrations uses \cite[Thm 12]{strom_2}.
\end{proof}

Continuing to let $f\colon A \to B$ be any map of spaces, for each retractive space $X$ over $A$ we can form the pushout $f_!X = X \cup_A B$ and fit it into a diagram as indicated:
\[ \xymatrix @R=1.7em @C=3em{
	A \ar[d]_-{i_X} \ar[r]^-f  \po{.8}{-0.2ex} & B \ar[d] \ar@/^2em/[dd]^-{\id} \\
	X \ar[r] \ar[d]_-{p_X} & f_! X \ar@{-->}[d] \\
	A \ar[r]^-f & B
} \]
This makes $f_!X$ into a retractive space over $B$, and defines the \textbf{pushforward functor}\index{pushforward $f_{^^21}$} $f_!\colon \mc R(A) \to \mc R(B)$.
\begin{rmk}
	These pushouts and pullbacks commute with the forgetful functor to sets, in both (CG) and (CGWH).
\end{rmk}

\begin{lem}\label{lem:f_shriek_preserves}
	The pushforward functor $f_!$ preserves\vspace{-.5em}
	\begin{itemize}
		\item $f$-cofibrations, $h$-cofibrations, closed inclusions, and
		\item weak equivalences between $h$-cofibrant (or $f$-cofibrant) spaces.
	\end{itemize}
	Moreover if $f$ itself is a Hurewicz fibration ($h$-fibration), then $f_!$ preserves spaces that are both $h$-fibrant and $f$-cofibrant.
\end{lem}

\begin{proof}
	Again this is straightforward. The last claim follows from \autoref{prop:clapp}.
\end{proof}

\begin{ex}\label{ex:thom_base_change}
	For any vector bundle $V \to B$ the ordinary Thom space $\Th(V)$ is the pushforward of the fiberwise Thom space $\Th_B(V)$\index{fiberwise!Thom space} (see \autoref{ex:fiberwise_thom_space}),
	\[ r_!\Th_B(V) \cong \Th(V). \]
\end{ex}

The functors $f_!$ and $f^*$ are adjoint, because maps $f_!X \to Y$ and $X \to f^*Y$ naturally correspond to commuting diagrams of the following form.
\begin{equation}\label{eq:adjoint}
	\xymatrix @R=1.7em @C=3em{
		A \ar[d]_-{i_X} \ar[r]^-f & B \ar[d]^-{i_Y} \\
		X \ar[d]_-{p_X} \ar[r] & Y \ar[d]^-{p_Y} \\
		A \ar[r]^-f & B
	}
\end{equation}
This leads us to define a larger category $\mc R$ of all retractive spaces over all base spaces, as in \cite{hebestreit_sagave_schlichtkrull}.

\begin{df}\label{def:all_ret_spaces}
	The category $\mc R$\index{$\mc R$} has as its objects the pairs $(X,A)$, where $A$ is a space and $X \in \mc R(A)$. The morphisms $(X,A) \to (Y,B)$ are diagrams of the form \eqref{eq:adjoint}.
\end{df}

The forgetful functor to (CGWH) topological spaces
\[ \xymatrix{ \mc R \ar[r] & \cat{Top} } \]
is an example of a bifibration of categories.

\begin{df}\label{def:bifibration} (e.g. \cite{shulman_framed_monoidal})
	A \textbf{bifibration}\index{bifibration} is a functor $\Phi\colon \mc C \to \bS$ such that
\begin{itemize}
	\item For every arrow $f\colon A \to B$ in $\bS$ and object $X$ in the fiber category $\mc C^B = \Phi^{-1}(B)$, there is a \textbf{cartesian arrow}\index{cartesian/co-cartesian arrow} $f^*X \to X$ in $\mc C$ with the universal property illustrated below.
	\[ \xymatrix @R=1em{
		Y \ar@{~>}[dd]^-\Phi \ar[rrd] \ar@{-->}[rd]_-{\exists !} & & \\
		& f^*X \ar[r] \ar@{~>}[d]^-\Phi & X \ar@{~>}[d]^-\Phi \\
		\Phi(Y) \ar[r] & A \ar[r]^-f & B
	} \]
	\item For every $f\colon A \to B$ and $X$ in $\mc C^A = \Phi^{-1}(A)$, there is a \textbf{co-cartesian arrow} $X \to f_!X$ with the universal property illustrated below.
	\[ \xymatrix @R=1em{
		& & Y \ar@{~>}[dd]^-\Phi \\
		X \ar[r] \ar[rru] \ar@{~>}[d]^-\Phi & f_!X \ar@{~>}[d]^-\Phi \ar@{-->}[ru]_-{\exists !} & \\
		A \ar[r]^-f & B \ar[r] & \Phi(Y)
	} \]
	\item Optionally, there is a class of commuting squares in $\bS$ satisfying the Beck-Chevalley condition (see \autoref{prop:beck_chevalley_spaces}).
\end{itemize}
\end{df}

In any bifibration, these universal properties allow us to define pushforward functors $f_!$ and pullback functors $f^*$, and these are always adjoint to each other, as we have seen for $\mc R$.

\begin{prop}\label{spaces_quillen_adj}
	The functors $(f_!,f^*)$ form a Quillen adjunction between $\mc R(A)$ and $\mc R(B)$. If $f$ is a weak homotopy equivalence then it is a Quillen equivalence.
\end{prop}

\begin{proof}
	The Quillen conditions are checked above. If $f$ is a weak equivalence, $X \in \mc R(A)$ is $q$-cofibrant and $Y \in \mc R(B)$ is $q$-fibrant, then the following square in $\mc R$ has both horizontal maps weak equivalences.
	\[ \xymatrix @R=1.7em{
		X \ar[d] \ar[r]^-\sim & f_!X \ar[d] \\
		f^*Y \ar[r]_-\sim & Y
	} \]
	Therefore the left-hand map is an equivalence iff the right-hand map is, so $(f_! \adj f^*)$ is a Quillen equivalence.
\end{proof}

\begin{rmk}\label{integral_spaces}
	There is an \textbf{integral model structure}\index{integral model structure!for spaces} on $\mc R$, obtained by combining together the Quillen model structures on the fiber categories $\mc R(B)$, see \cite{harpaz_prasma_grothendieck,cagne2017bifibrations,hebestreit_sagave_schlichtkrull}. There are actually two of these, one where we take the weak equivalences in the base category to be the usual equivalences, and another where we take the weak equivalences in the base to be only the isomorphisms. In this paper we will focus on the case of isomorphisms in the base, since it will permit us to take homotopy categories of each fiber category $\ho \mc R(B)$ separately and still get a bifibration (\autoref{spaces_are_a_bifibration}).
\end{rmk}
	
Since $(f_! \adj f^*)$ is an adjoint pair, $f_!$ preserves all colimits and $f^*$ preserves all limits. In addition, $f^*$ preserves many of the colimits we actually care about.
\begin{prop}\label{prop:f_star_colimits}\hfill
	\vspace{-.5em}
	
	\begin{itemize}
		\item (CGWH) $f^*$ preserves coproducts, pushouts along a closed inclusion, sequential colimits along closed inclusions, and quotients by fiberwise actions of compact Hausdorff topological groups.
		\item (CG) $f^*$ preserves all colimits.
	\end{itemize}
\end{prop}
\begin{proof}
	The (CG) statement follows by analyzing the underlying sets, see also \cite[10.3]{lewis_appendix}. In (CGWH) it suffices to prove this statement in the category of spaces over $B$, then pass to retractive spaces (where the colimits change, by a pushout along $\amalg B \to B$). For the first step, the colimits listed above as constructed in weak Hausdorff spaces agree with the colimit in ordinary (or $k$-)spaces \cite{strickland2009category}, and the result is still weak Hausdorff, so it follows that they commute with $f^*$. For the second step, the pushout has one leg a closed inclusion because every (CGWH) retractive space is $i$-cofibrant. Therefore this follows from the first step.
\end{proof}

\begin{prop}\label{prop:beck_chevalley_spaces}\cite[2.2.11]{ms}
	$f_!$ and $f^*$ satisfy the Beck-Chevalley condition along every pullback square. That is, for each pullback square as below, the canonical natural transformation $f_!g^* \Rightarrow q^*p_!$ is an isomorphism.\index{Beck-Chevalley isomorphism}
	\begin{equation*}\label{diagonal_square}
	\xymatrix@R=1em @C=1em{
		& A \ar[ld]_-g \ar[rd]^-f & \\
		B \ar[rd]_-p && C \ar[ld]^-q \\
		& D &
	}
	\end{equation*}
\end{prop}


There is a third way to change base spaces. If $X \in \mc R(A)$, its \textbf{space of sections}\index{sections $\Gamma_B(-)$} $\Gamma_A(X)$ consists of all the maps $s\colon A \to X$ such that the composite $p_X \circ s\colon A \to X \to A$ is the identity. This is a based space $\Gamma_A(X) \in \mc R(*)$, with basepoint given by $i_X$.

More generally, along any fiber bundle $f\colon A \to B$, we can send a parametrized space $X \in \mc R(A)$ to the space $f_*X \in \mc R(B)$ whose fiber over $b \in B$ is the space of sections of $X$ over the subspace $f^{-1}(b)$ of $A$. The details of this are in the next proposition.

\begin{prop}\label{prop:sheafy_pushforward}
	If $f\colon A \to B$ is a fiber bundle and $B$ is a cell complex then $f^*$ has a right adjoint $f_*$.\index{sheafy pushforward $f_*$}
\end{prop}

\begin{proof}
	Cover $B$ by neighborhoods $U_\alpha \subseteq B$ on which the bundle $A \to B$ is trivial, and pick fiber-preserving homeomorphisms $\phi_\alpha\colon U_\alpha \times F_\alpha \cong f^{-1}(U_\alpha)$. For each $U_\alpha$, define $(f_*X)_\alpha$ as the pullback
	\[ \xymatrix @R=1.7em{
		(f_*X)_\alpha \ar[d] \ar[r] \pb{.25}{-0.5ex} & \Map(F_\alpha,X) \ar[d]^-{\Map(\id,p_X)} \\
		U_\alpha \ar[r] & \Map(F_\alpha,A).
	} \]
	For any pair $\alpha, \beta$, let $U_{\alpha\beta} = U_\alpha \cap U_\beta$ and let $(f_*X)_{\alpha\beta}$ and $(f_*X)_{\beta\alpha}$ denote the pullbacks
	\[ \xymatrix @R=1.7em{
		(f_*X)_{\alpha\beta} \ar[d] \ar[r] \pb{.25}{-0.5ex} & (f_* X)_\alpha \ar[d] \\
		U_{\alpha\beta} \ar[r] & U_\alpha
	}
	\quad
	\xymatrix @R=1.7em{
		(f_*X)_{\beta\alpha} \ar[d] \ar[r] \pb{.25}{-0.5ex} & (f_* X)_\beta \ar[d] \\
		U_{\alpha\beta} \ar[r] & U_\beta.
	}
	\]
	We form a homeomorphism $h_{\alpha\beta}\colon (f_*X)_{\alpha\beta} \cong (f_*X)_{\beta\alpha}$ as follows. Take the adjoint of the composite
	\[ \xymatrix @R=1.7em{
		U_{\alpha\beta} \times F_\alpha \ar[r]^-{\phi_\beta^{-1} \circ \phi_\alpha}_-\cong & U_{\alpha\beta} \times F_\beta \ar[r] & F_\beta
	} \]
	to get a function $\tau\colon U_{\alpha\beta} \to \Map(F_\alpha,F_\beta)$. Then form the map of products
	\[ \xymatrix @R=1.7em{
		U_{\alpha\beta} \times \Map(F_\beta,X) \ar[r]^-{(1,\tau) \times 1} &
		U_{\alpha\beta} \times \Map(F_\alpha,F_\beta) \times \Map(F_\beta,X) \ar[r]^-{1 \times \circ} &
		U_{\alpha\beta} \times \Map(F_\alpha,X).
	} \]
	This respects the subspaces
	\[ \xymatrix @R=1.7em{
		U_{\alpha\beta} \times_{\Map(F_\beta,A)} \Map(F_\beta,X) \ar[r] &
		U_{\alpha\beta} \times_{\Map(F_\alpha,A)} \Map(F_\alpha,X)
	} \]
	that define the pullbacks, and hence give a map $(f_*X)_{\beta\alpha} \to (f_*X)_{\alpha\beta}$. The same recipe in reverse gives the inverse of this map, hence it is a homeomorphism.
	
	The cocycle condition for the trivializations $\phi_\alpha$ imply that for any triple $\alpha, \beta, \gamma$ the composite $h_{\gamma\alpha} \circ h_{\beta\gamma} \circ h_{\alpha\beta}$ is the identity wherever it is defined. Hence the spaces $(f_* X)_\alpha$ glue together to give a single space $f_* X$ over $B$ whose restriction to $U_\alpha$ is canonically homeomorphic to $(f_* X)_\alpha$; this is elementary to argue in (CG).

	To establish the same claim in (CGWH), we need to check that this construction produces a weak Hausdorff space. For this we use the assumption that $B$ is a cell complex and choose the original cover so that every cell $e_i$ is contained in a $U_{\alpha(i)}$. Then any map into $f_*X$ from a compact Hausdorff space $K$ has image in $B$ landing in a finite complex, hence it suffices to prove that the image of $K$ when restricted to each closed cell $e_i$ in the base gives a closed subset of $(f_*X)_{e_i}$. This follows because each of the spaces $(f_*X)_\alpha$ is weak Hausdorff, being a pullback of weak Hausdorff spaces.
	
	The adjunction $(f^* \adj f_*)$ is left as a straightforward exercise.
\end{proof}

As a result $f^*$ preserves all colimits so long as $f$ is a fiber bundle with base a cell complex. It is a classical fact that every map is weakly equivalent to such a fiber bundle. Therefore $f^*$ is always a left adjoint on homotopy categories. It also commutes with homotopy colimits up to equivalence, even though in (CGWH) it doesn't always strictly commute with strict colimits.

\begin{prop}\label{prop:sheafy_pushforward_preserves}
	The functor $f_*$ of \autoref{prop:sheafy_pushforward} preserves $h$-fibrations. If the fibers of $f$ are cell complexes, then $f_*$ preserves $q$-fibrations and weak equivalences between $q$-fibrant spaces.
\end{prop}

\begin{proof}
	This is by a standard argument that rearranges one lifting condition to another. For instance, to prove the first part we re-arrange a lift in the square below on the right to a lift on the left. Note that in both squares any lift will automatically be a map over $A$ or $B$.
	\[ \xymatrix @R=1.7em @C=3em{
		f^*K \times \{0\} \ar[r] \ar[d] & X \ar[d] \\
		f^*K \times I \ar@{-->}[ur] \ar[r] & Y
	}
	\qquad
	\xymatrix @R=1.7em @C=3em{
		K \times \{0\} \ar[r] \ar[d] & f_*X \ar[d] \\
		K \times I \ar@{-->}[ur] \ar[r] & f_*Y
	}
	 \]
	The second part uses a similar argument along with Ken Brown's lemma \cite[1.1.12]{hovey_model_cats}.
\end{proof}

\begin{ex}\label{ex:other_adjunction_spaces}
	As a result, when $f$ is a fiber bundle with fiber a cell complex, the adjunction $(f^* \adj f_*)$ is Quillen. If $f$ is a weak equivalence, then $(f^* \adj f_*)$ is a Quillen equivalence, because $\L f^* \simeq \R f^*$ gives an equivalence of homotopy categories by \autoref{spaces_quillen_adj}.
\end{ex}

\begin{rmk}
	Under the (CG) assumptions, $f^*$ always has a right adjoint $f_*$ \cite[2.1]{ms}. However, $f_*$ does not seem to be right-deformable in the sense of \autoref{sec:derived_functors_intro} or \cite{dhks}. This makes the more general $f_*$ difficult to use for homotopy theory.
\end{rmk}


\subsection{External smash products $\barsmash$}\label{sec:external_smash}\aftersubsection

\begin{df}\label{def_external_smash}
If $X$ is a retractive space over $A$ and $Y$ is a retractive space over $B$, the \textbf{external smash product}\index{external smash product $X \barsmash Y$!of retractive spaces} $X \barsmash Y$ is a retractive space over $A \times B$ whose fiber over $(a,b) \in A \times B$ is the smash product of the fibers, $X_a \sma Y_b$. It is defined by the following pushout square.
\begin{equation}\label{external_smash}
\xymatrix @R=2em{
	(X \times B) \cup_{A \times B} (A \times Y) \ar[r] \ar[d] \po{.8}{+.5ex} & X \times Y \ar[d] \\
	A \times B \ar[r] & X \barsmash Y }
\end{equation}
The left vertical map arises from $p_X$ and $p_Y$, and the top horizontal arises from $i_X$ and $i_Y$. This is a retractive space over $A \times B$ in the evident way. Once again, the above pushout is preserved by the forgetful functor to sets in both (CG) and (CGWH).
\end{df}

We usually consider this as a functor $\barsmash\colon \mc R(A) \times \mc R(B) \to \mc R(A \times B)$ for each pair of base spaces $A$ and $B$. However, it also defines a functor $\barsmash\colon \mc R \times \mc R \to \mc R$ on the entire category of retractive spaces $\mc R$ from \autoref{def:all_ret_spaces}.
\begin{lem}\label{prop:spaces_sm}
	The external smash product makes $\mc R$ into a symmetric monoidal category with unit the retractive space $S^0$ over $*$.
\end{lem}

In the special case where we take $X_{+A} = X \amalg A$, with $X$ an unbased space over $A$, the external smash product $X_{+A} \barsmash Y$ is called a ``half-smash product'' and its total space is given more simply as the pushout
\begin{equation}\label{external_half_smash}
\xymatrix @R=1.7em{
	X \times B \ar[r] \ar[d]  \po{.65}{-.5ex} & X \times Y \ar[d] \\
	A \times B \ar[r] & X_{+A} \barsmash Y. }
\end{equation}
In the further special case of $X_{+A} \barsmash Y_{+B}$, we just get the disjoint union $(X \times Y)_{+(A \times B)}$.

\begin{ex}\label{ex:fiberwise_suspension}
	Taking $A = *$ and $X = S^1$, we get the definition of the \textbf{fiberwise reduced suspension}\index{fiberwise!reduced suspension $\Sigma_B$} functor on spaces $Y$ over $B$:
	\[ \Sigma_B Y := S^1 \barsmash Y \]
\end{ex}

\begin{ex}\label{ex:thom_smash_product}
	If we take two vector bundles $V$ and $W$ over $A$ and $B$, respectively, the fiberwise Thom space $\Th_{A \times B}(V \times W)$ can be expressed by the formula
	\[ \Th_{A \times B}(V \times W) \cong \Th_A(V) \barsmash \Th_B(W). \]
	Combining with the previous example, we get
	\[ \Th_B(\R^n \times W) \cong S^n \barsmash \Th_B(W) = \Sigma^n_B \Th_B(W). \]
	This is the fiberwise variant of the isomorphism $\Th(\R^n \times W) \cong \Sigma^n \Th(W)$.
\end{ex}

\begin{ex}
	If $X$ is an unbased space over $B$ with projection $p\colon X \to B$, and $Y$ is a retractive space over $B$, we can take the half-smash product \eqref{external_half_smash} and pull back along $\Delta_B$, which preserves the pushout in both (CG) and (CGWH). This gives the square
	\begin{equation*}
	\xymatrix @R=1.7em{
		X \times_B B \ar[r] \ar[d] & X \times_B Y \ar[d] \\
		B \times_B B \ar[r] & X_{+B} \sma_B Y }
	\end{equation*}
	and hence an isomorphism
	\begin{equation}\label{half_smash_internal}
		X_{+B} \sma_B Y \cong p_!p^* Y.
	\end{equation}
\end{ex}

To describe how $\barsmash$ interacts with cofibrations, fibrations, and weak equivalences, we must recall some basic facts about pushout-products. Suppose $\mc C$, $\mc D$, and $\mc E$ are categories with all small colimits, and
$\otimes\colon \mc C \times \mc D \to \mc E$
is a colimit-preserving functor. The \textbf{pushout-product}\index{pushout-product $\square$} of a map $f: K \to X$ in $\mc C$ and $g: L \to Y$ in $\mc D$ is the map in $\mc E$ from the pushout of the first three terms in the following square to the final vertex.
\[ \xymatrix @R=2em @C=3em{
	K \otimes L \ar[r]^-{\id_K \otimes g} \ar[d]_-{f \otimes \id_L} & K \otimes Y \ar[d]^-{f \otimes \id_Y} \\
	X \otimes L \ar[r]_-{\id_X \otimes g} & X \otimes Y
}\]
\[ f \square g\colon X \otimes L \cup_{K \otimes L} K \otimes Y \to X \otimes Y \]

Let $a\mc C$ be the category of arrows, whose objects are morphisms of $\mc C$ and whose morphisms are commuting squares. We say a morphism in $a\mc C$ is a ``pushout morphism'' if this square is a pushout square in $\mc C$. The following is a straightforward exercise.
\begin{lem}\label{pushout_product_lemmas}\hfill
	\vspace{-1em}
	
	\begin{itemize}
		\item The operation $- \square -$ defines a functor $a\mc C \times a\mc D \to a\mc E$.
		\item For fixed $g \in a\mc D$, $- \square g$ sends pushout morphisms to pushout morphisms.
		\item The pushout-product $(f' \circ f) \square g$ is a composition of $f' \square g$ and a pushout of $f \square g$.
	\end{itemize}
\end{lem}

This last point also generalizes to countable and transfinite compositions, as in \cite[2.1.1]{hovey_model_cats}. For the countable case, $f: X^{(0)} \to X^{(\infty)}$ is the composition of the maps $f_{m-1}: X^{(m-1)} \to X^{(m)}$ in $\mc C$, meaning that $X^{(\infty)}$ is identified with the colimit of the $X^{(m)}$, and $g: L \to Y$ is an arrow of $\mc D$. Then the pushout-product
\[ f \square g : (X^{(\infty)} \otimes L) \cup_{X^{(0)} \otimes L} (X^{(0)} \otimes Y) \to X^{(\infty)} \otimes Y \]
is a countable composition of pushouts of the maps $f_{m-1} \square g$.

If $\mc C$, $\mc D$, and $\mc E$ are model categories, we say that $\otimes$ is a \textbf{left Quillen bifunctor}\index{left Quillen bifunctor} if
	\begin{itemize}
		\item it is a left adjoint in each slot,
		\item for each cofibration $f$ in $\mc C$ and cofibration $g$ in $\mc D$, $f \square g$ is a cofibration, and
		\item if in addition one of $f$ or $g$ is a weak equivalence then $f \square g$ is a weak equivalence.
	\end{itemize} 
As a consequence, $X \otimes -$ is left Quillen whenever $X$ is cofibrant in $\cat C$, and similarly $- \otimes Y$ is left Quillen if $Y$ is cofibrant in $\cat D$. More informally, $X \otimes Y$ preserves equivalences when $X$ and $Y$ are both cofibrant. A standard application of \autoref{pushout_product_lemmas} is the following.\index{model category!monoidal}
\begin{prop}\label{prop:quillen_bifunctor_defn}
	If $\mc C$ and $\mc D$ are cofibrantly generated, then for $\otimes$ to be a Quillen bifunctor it suffices that
	\begin{itemize}
		\item it is a left adjoint in each slot,
		\item if $f,g \in I$ then $f \square g$ is a cofibration, and
		\item if $f \in I, g \in J$ or $f \in J, g \in I$ then $f \square g$ is an acyclic cofibration.
	\end{itemize} 
\end{prop}

\begin{lem}\label{ex:smashing_spaces_is_left_Quillen}
	The external smash product $\barsmash\colon \mc R(A) \times \mc R(B) \to \mc R(A \times B)$ is a left Quillen bifunctor. It therefore preserves equivalences when both inputs are $q$-cofibrant.
\end{lem}

\begin{proof}
	Simply check the conditions of \autoref{prop:quillen_bifunctor_defn}. The right adjoint is constructed in the next section.
\end{proof}

To build our ``convenient'' notions of cofibrant and fibrant spectra, we will need to know that $\barsmash$ plays well with $h$- and $f$-cofibrations as well.

\begin{prop}\label{prop:h_cofibrations_pushout_product}\hfill
	\vspace{-1em}
	
	\begin{enumerate}
		\item Let $f: K \to X$ and $g: L \to Y$ be $h$-cofibrations of retractive spaces over $A$ and $B$, respectively. Then $f \square g$, constructed using $\barsmash$, is an $h$-cofibration.
		\item The same is true for $f$-cofibrations or closed inclusions.
		\item If $X$ and $Y$ are $h$-cofibrant and $h$-fibrant then $X \barsmash Y$ is $h$-fibrant.
		\item If $X$ is an $h$-cofibrant space over $A$ and $g: Y \to Y'$ is a weak equivalence of $h$-cofibrant spaces over $B$ then $\id_X \barsmash g$ is a weak equivalence.
	\end{enumerate}
\end{prop}

Note that the third of these also appears in \cite[8.2.4]{ms} and \cite[3.6]{may1975classifying}.

\begin{proof}
	For the first two claims, we compare universal properties and write the desired pushout-product as the pushout of the map of spans
	\[ \xymatrix @R=1.7em{
		A \times B \ar@{=}[d] & (X \times B) \cup_{A \times B} (A \times Y) \ar[l] \ar[r] \ar@{=}[d] & (X \times L) \cup_{K \times L} (K \times Y) \ar[d] \\
		A \times B & (X \times B) \cup_{A \times B} (A \times Y) \ar[l] \ar[r] & X \times Y
	} \]
	By \autoref{cofibration_of_pushouts}, the claim reduces to the fact that $f \square g$ computed with the Cartesian product is a cofibration, which follows from \autoref{lem:product_ndr}.
	
	For the third claim, in the square \eqref{external_smash}, the top-left term is a pushout along $h$-cofibrations of spaces that are fibrant over $A \times B$. By \autoref{prop:clapp} and its extension with $h$-cofibrations, this top-left term of \eqref{external_smash} is therefore fibrant. The top map of \eqref{external_smash} itself is also an $h$-cofibration, and the other terms are clearly fibrant over $A \times B$, so again their pushout $X \barsmash Y$ is also fibrant.
	For the final claim, the same analysis shows that for a weak equivalence $Y \to Y'$ we get a weak equivalence on the top-left term of the square \eqref{external_smash}, as well as on the other two terms, and therefore gives a weak equivalence on the pushouts $X \barsmash Y \to X \barsmash Y'$.
\end{proof}

\begin{rmk} As a corollary, $\Sigma_B(-)$ preserves $f$-cofibrations, $h$-cofibrations, and weak equivalences between $h$-cofibrant spaces.
\end{rmk}

\begin{df}\label{space_cofiber}
	For any map $f\colon X \to Y$ in $\mc R(B)$, the mapping cone or \textbf{uncorrected homotopy cofiber}\index{uncorrected homotopy cofiber}\index{mapping cone} $C_B f$ is the pushout
	\[ C_B f = X \barsmash I \cup_{X \barsmash S^0} Y \]
	where $S^0 \to I$ is regarded as a map of based spaces (retractive spaces over $*$).
\end{df}

Over each point $b \in B$, this is the usual reduced mapping cone. As in the non-parametrized case, $C_B f$ doesn't always have the homotopy type we want. It does if $X$ is $h$-cofibrant, using \autoref{prop:h_cofibrations_pushout_product}, or if $f$ is an $h$-cofibration. In these cases we call it the \textbf{homotopy cofiber}\index{homotopy!cofiber}.

\subsection{External mapping spaces $\barmap$}\label{sec:external_smash}\aftersubsection

The external smash product has a right adjoint in each variable, the \textbf{external mapping space}\index{external!mapping space $\barmap_B(Y,Z)$} $\barmap_B(Y,Z)$. For $Z \in \mc R(A \times B)$ and $Y \in \mc R(B)$, this is defined as the following pullback.
\begin{equation}\label{external_map}
	\xymatrix @R=2em{
		\barmap_B(Y,Z) \ar[r] \ar[d] \pb{.2}{-1ex} & \Map(Y,Z) \ar[d] \\
		A \times \{{*}\} \times \{{*}\} \ar[r] & \Map(Y,A) \times \Map(Y,B) \times \Map(B,Z)
	}
\end{equation}
Here $\Map$ refers to the usual space of unbased maps, the right-vertical composes with the projections $Z \to A$, $Z \to B$, and the section $B \to Y$, and the bottom-horizontal assigns to each $a \in A$ the constant map $Y \to \{a\}$, the projection $Y \to B$, and the section $B \cong \{a\} \times B \to Z$.

When $Y$ is nonempty the bottom horizontal map is split,\footnote{When $Y$ is empty, so is $B$, and $\barmap_B(Y,Z) = \barmap_{\emptyset}(\emptyset,Z) = A$.} hence in (CGWH) both horizontal maps are closed inclusions, while in (CG) they are just inclusions. So $\barmap_B(Y,Z)$ is precisely the subspace of all maps $Y \to Z$ that fit into three commuting diagrams
\[ \xymatrix @R=1.7em{
	Y \ar[d] \ar[r] & Z \ar[d] \\
	\{{*}\} \ar[r]^-a & A
}
\qquad
\xymatrix @R=1.7em{
	Y \ar[dr] \ar[r] & Z \ar[d] \\
	 & B
}
\qquad
\xymatrix @R=1.7em{
B \ar[d] \ar[r]^-{a \times \id} & A \times B \ar[d] \\
Y \ar[r] & Z,
}
\]
or equivalently into one commuting diagram as follows, for some $a \in A$.
\[ \xymatrix @C=3em @R=1.7em{
	B \ar[d] \ar[r]^-{a \times \id} & A \times B \ar[d] \\
	Y \ar[d] \ar[r] & Z \ar[d] \\
	B \ar[r]^-{a \times \id} & A \times B
}
\]

Therefore $\barmap_B(Y,Z)$ is a retractive space over $A$ whose fiber over $a \in A$ is the space of maps of retractive spaces from $Y$ into $Z_a$ over $B$. The usual rules define adjunctions
\[ \xymatrix @R=0.3em{ X \barsmash Y \to Z \quad \textup{ over }A \times B \quad \longleftrightarrow \quad X \to \barmap_B(Y,Z) \quad \textup{ over }A, \\
X \barsmash Y \to Z \quad \textup{ over }A \times B \quad \longleftrightarrow \quad Y \to \barmap_A(X,Z) \quad \textup{ over }B.
} \]

\begin{ex}\label{ex:fiberwise_loops}\hfill
	\vspace{-1em}
	
	\begin{itemize}
		\item Setting $A = *$, the based space $\barmap_B(Y,Z)$ is just the space of maps $Y \to Z$ that agree with the inclusions and projections to $B$, in other words the space of maps of retractive spaces from $Y$ to $Z$. This gives an enrichment of the category $\mc R(B)$ in based spaces, in other words it gives a natural topology on the sets $\mc R(B)(Y,Z)$.\index{retractive space!enrichment}
		\item Setting $A = *$ and $X = S^1$ defines the \textbf{fiberwise based loops}\index{fiberwise!based loops $\Omega_B$} of $Y \in \mc R(B)$,
		\[ \Omega_B Y := \barmap_*(S^1,Y). \]
		Every point of $\Omega_B Y$ is a map $S^1 \to Y$ that lands entirely in one fiber $Y_b$, sending the basepoint of $S^1$ to the basepoint of $Y_b$. We also get an adjunction between $\Sigma_B$ and $\Omega_B$ that agrees with the usual one over every point of $B$.
		\item Setting $A = *$ and $Y = B_{+B}$, we get the space of sections,
		\[ \barmap_B(B_{+B},Z) \cong \Gamma_B(Z). \]
		More generally, if $Y = E_{+B}$ for some space $E$ with projection $p\colon E \to B$, juggling the above adjunctions and identifications gives
		\begin{equation}\label{half_map_sections}
			\barmap_B(E_{+B},Z) \cong \Gamma_E(p^* Z).
		\end{equation}
	\end{itemize}
\end{ex}

To describe how $\barmap$ interacts with cofibrations, fibrations, and weak equivalences, we recall the definition of pullback-homs, the duals of pushout-products. As before, start with categories $\mc C$, $\mc D$, and $\mc E$, but now assume that $\mc C$ has all small limits. Consider two functors
\[ \otimes\colon \mc C \times \mc D \to \mc E, \qquad
\Hom_{\mc D}\colon \mc D^{op} \times \mc E \to \mc C \]
and an adjunction between $- \otimes d$ and $\Hom_{\mc D}(d,-)$ for every $d \in \mc D$. For every $g: L \to Y$ in $\mc D$ and $h\colon M \to Z$ in $\mc E$, the \textbf{pullback-hom}\index{pullback-hom} $\Hom_\square(g,h)$ is the map in $\mc C$ that takes the first vertex of the square below to the pullback of the other three vertices:
\[ \xymatrix @R=2em{
	\Hom_{\mc D}(Y,M) \ar[r] \ar[d] & \Hom_{\mc D}(L,M) \ar[d] \\
	\Hom_{\mc D}(Y,Z) \ar[r] & \Hom_{\mc D}(L,Z)
} \]

\[ \Hom_\square(g,h)\colon \Hom_{\mc D}(Y,M) \to \Hom_{\mc D}(L,M) \times_{\Hom_{\mc D}(L,Z)} \Hom_{\mc D}(Y,Z). \]

This can be visualized as the space of all choices of diagonal in the square below, mapping to the space of all choices of two horizontal maps making the square commute.
\[ \xymatrix @R=1.7em @C=3em{
	L \ar@{-->}[r] \ar[d]_-g & M \ar[d]^-h \\
	Y \ar@{-->}[r] \ar@{-->}[ur] & Z
} \]

\begin{lem}\label{lem:pullback_hom_adjunction}
	For a fixed arrow $g \in a\mc D$, the functors $- \square g$ and $\Hom_\square(g,-)$ form an adjunction between $a\mc C$ and $a\mc E$.
\end{lem}
\begin{proof}
	A choice of dotted maps making the square
	\[ \xymatrix @R=1.7em{
		K \ar@{-->}[r] \ar[d]^-f & \Hom_{\mc D}(Y,M) \ar[d]^-{\Hom_\square(g,h)} \\
		X \ar@{-->}[r] & \Hom_{\mc D}(L,M) \times_{\Hom_{\mc D}(L,Z)} \Hom_{\mc D}(Y,Z)
	} \]
	commute in $\mc C$ can be rearranged into a choice of three dotted maps making the following diagram commute in $\mc E$.
	\[ \xymatrix @R=1.7em{
		X \otimes L \cup_{K \otimes L} K \otimes Y \ar@{-->}[r] \ar[d]^-{f \square g} & M \ar[d]^-{h} \\
		X \otimes Y \ar@{-->}[r] & Z
	} \]
	It is straightforward to check these identifications are natural.
\end{proof}

The following is a standard consquence of \autoref{ex:smashing_spaces_is_left_Quillen}.
\begin{lem}\label{ex:homming_spaces_is_right_Quillen}
	The external mapping space $\barmap_B(Y,Z)$ is $q$-fibrant and preserves weak equivalences, whenever $Y$ is $q$-cofibrant and $Z$ is $q$-fibrant.
\end{lem}

We also need to understand how the external mapping space interacts with $h$- and $f$-cofibrations.

\begin{prop}\label{h_fibrations_pullback_hom}\hfill
	\vspace{-1em}
	
	\begin{itemize}
		\item Let $g: L \to Y$ be an $h$-cofibration over $B$ and $h: M \to Z$ be an $h$-fibration over $A \times B$. Then $\Hom_\square(g,h)$, constructed using $\barmap_B$, is an $h$-fibration.
		\item (CGWH) If $h\colon M \to Z$ is a closed inclusion then so is $\barmap_B(Y,h)$.
		\item (CGWH) If $Y$ is compact and $h\colon M \to Z$ is an $f$-cofibration then $\barmap_B(Y,h)$ is an $f$-cofibration.
	\end{itemize}
\end{prop}

As a consequence, $\Omega_B(-)$ preserves $h$-fibrations and weak equivalences between $h$-fibrant spaces. Curiously, it also preserves $f$-cofibrations, at least in (CGWH).

\begin{proof}
	For the first claim, examine the proof of \autoref{lem:pullback_hom_adjunction}. A dotted map lifting the bottom square corresponds to a dotted map lifting the top square. Taking $X = K \times I$, we observe that $K \to K \times I$ is an $h$-cofibration and the left-vertical in the last square is clearly a homotopy equivalence. Therefore we get a lift by \cite[Thm 3]{strom_1}, hence the original right-vertical $\Hom_\square(g,h)$ was an $h$-fibration.
	
	The second claim follows from \autoref{lem:closed_inclusion_equalizer} because $\barmap_B(Y,-)$ is a right adjoint and therefore preserves equalizers. For the third claim, given a retract
	\[ \xymatrix{
		M \times I \cup_{M \times 0} Z \times 0 \ar[r] & Z \times I \ar@/_1em/[l]
	} \]
	over $A \times B$, form the following square.
	\[ \xymatrix{
		\Map(Y,M \times I \cup_{M \times 0} Z \times 0) \ar[r] \ar@/_/[d] & \Map(Y,Z \times I) \ar@/^/[d] \ar@/_/[l] \\
		(\Map(Y,M) \times I) \cup_{\Map(Y,M) \times 0} (\Map(Y,Z) \times 0) \ar[r] \ar[u] & \Map(Y,Z) \times I \ar[u]
	} \]
	The left-pointing map is $\Map(Y,-)$ applied to our retract. The maps pointing to the right are closed inclusions, so we regard the left-hand column as subspaces of the right-hand column, and define all the vertical maps by focusing on the right-hand column and checking our rule preserves the given subspaces. We define the upward vertical map by assembly, sending $(f,t)$ to the function $(y \mapsto (f(y),t))$. We define the first factor of the downward vertical map by projecting $Z \times I \to Z$. The second factor takes the minimum over $Y$ of the $I$-coordinate of $f$. It is straightforward to check directly that this is continuous.
	
	We can now retract the lower-right term onto the lower-left term by mapping up, left, and then down.\footnote{The author learned this argument from Irakli Patchkoria, who in turn learned it from Thomas Nikolaus and Wolfgang L\" uck.} It is straightforward to check that this retraction respects the three conditions that define the closed subspaces $\barmap_B(Y,M) \subseteq \Map(Y,M)$ and $\barmap_B(Y,Z) \subseteq \Map(Y,Z)$, and also the projection map from these closed subspaces to $A$.
\end{proof}

\begin{df}\label{space_fiber}
	The \textbf{uncorrected homotopy fiber}\index{uncorrected homotopy fiber} $F_B f$ is the pullback
	\[ F_B f = X \times_{\barmap_*(S^0,Y)} \barmap_*(I,Y). \]
	where as in \autoref{space_cofiber}, $S^0 \to I$ is regarded as a map of based spaces (retractive spaces over $*$). Over each point of $B$, this construction is the usual homotopy fiber. By \autoref{ex:homming_spaces_is_right_Quillen}, this preserves equivalences as soon as $Y$ is $q$-fibrant or $X \to Y$ is an $q$-fibration; in these cases we call it the \textbf{homotopy fiber}\index{homotopy!fiber}.
\end{df}
	\textbf{Warning:} The uncorrected homotopy fiber does not always preserve equivalences. Take $X = B = I$ and $Y$ to be the cone on a circle, with projection to $B$ by the cone coordinate, joined to an extra copy of $B$ along a single point of the base circle. Since $Y$ is contractible, the (corrected) homotopy fiber of $X \to Y$ is contractible, but the uncorrected homotopy fiber has total space equivalent to $\Z$.

\beforesubsection
\subsection{The symmetric monoidal bifibration (SMBF) $\mc R$}\aftersubsection

Recall that a \textbf{cartesian arrow} in the category $\mc R$ of retractive spaces is a morphism of the form $f^*X \to X$ from a pullback to the original space. Dually, a \textbf{cocartesian arrow} is a morphism of the form $X \to f_! X$.
\begin{lem}\cite[2.5.8]{ms}\label{lem:external_smash_and_base_change}
	The external smash product $\barsmash$ preserves cartesian arrows, and also preserves cocartesian arrows.
\end{lem}

This implies that external smash products commute with pullbacks and with pushforwards. To be specific, given the data of
\[ \xymatrix @R=.2em{
	X \in \mc R(A), & f\colon A \to A', & W \in \mc R(A'), \\
	Y \in \mc R(B), & g\colon B \to B', & Z \in \mc R(B'),
} \]
there are canonical isomorphisms
\begin{equation}\label{eq:external_smash_and_base_change}
\begin{array}{c}
f_!X \barsmash g_!Y \cong (f \times g)_!(X \barsmash Y) \\
f^*W \barsmash g^*Z \cong (f \times g)^*(W \barsmash Z).
\end{array}
\end{equation}

We have now given the category $\mc R$ from \autoref{def:all_ret_spaces} three operations, $\barsmash$, $f^*$, and $f_!$. These operations always commute with each other, because of \autoref{lem:external_smash_and_base_change} and \autoref{prop:beck_chevalley_spaces}. A category with this calculus of three operations is called a \textbf{symmetric monoidal bifibration (SMBF)}\index{symmetric monoidal!bifibration}\index{SMBF}.

\begin{df}\label{df:SMBF}(\cite{shulman_framed_monoidal})
An SMBF consists of
\begin{itemize}
	\item a bifibration $\Phi\colon \mc C \to \bS$ (\autoref{def:bifibration}),
	\item a symmetric monoidal structure $(\mc C,\boxtimes,I)$, and 
	\item a cartesian monoidal structure $(\bS,\times,*)$,
\end{itemize}
such that $\Phi$ is strict symmetric monoidal (preserving the products and coherent isomorphisms on the nose), and $\boxtimes$ preserves cartesian and also cocartesian arrows.
\end{df}

\begin{cor}\label{prop:spaces_SMBF}
	The category $\mc R$ of all retractive spaces, with the symmetric monoidal structure from \autoref{prop:spaces_sm}, is a symmetric monoidal bifibration.
\end{cor}

The isomorphisms \eqref{eq:external_smash_and_base_change}, and the Beck-Chevalley isomorphism of \autoref{prop:beck_chevalley_spaces}, allow us to commute pullbacks, pushforwards, and smash products with each other up to isomorphism. In order to calculate traces, we need to have control over the isomorphism thus produced. The simplest way to do this accounting is to recognize that the isomorphism thus produced is not just canonical, but unique.

\begin{prop}[Rigidity]\label{prop:spaces_rigidity}
	Suppose $n \geq 0$ and we have maps of spaces
	\[ \xymatrix{ B & \ar[l]_-f A \ar[r]^-g & C_1 \times \ldots \times C_n } \]
	such that $(f,g)\colon A \to B \times C_1 \times \ldots \times C_n$ is injective.
	
	Then any functor $\mc R(C_1) \times \ldots \times \mc R(C_n) \to \mc R(B)$ isomorphic to
	\[ \Phi\colon (X_1,\ldots,X_n) \leadsto f_!g^*(X_1 \barsmash \ldots \barsmash X_n) \]
	is in fact uniquely isomorphic to $\Phi$. In other words, $\Phi$ is \textbf{rigid}.\index{rigidity!for spaces}
\end{prop}

\begin{proof}
	It suffices to prove that $\Phi$ has a unique natural automorphism. Given any natural automorphism $\eta\colon \Phi \Rightarrow \Phi$, examine $\eta$ on an $n$-tuple of spaces of the form $(*_{+C_1},\ldots,*_{+C_n})$ where the free points map to $(c_1,\ldots,c_n) \in C_1 \times \ldots \times C_n$. The external smash product of these spaces is $*_{+(C_1 \times ... \times C_n)}$, with the free point mapping to $(c_1,\ldots,c_n)$. Pulling back gives $g^{-1}(c_1,\ldots,c_n)_{+A}$, which is a single fiber of $g$ regarded as a space over $A$, with a disjoint copy of $A$. Pushing forward gives a space over $B$. By assumption, every fiber of this space is either $*_+$ or $\emptyset_+$. Therefore the automorphism of this space provided by $\eta$ must be the identity map.
	
	Now suppose $(X_1,\ldots,X_n)$ is any other collection of retractive spaces in $\mc R(C_1) \times \ldots \times \mc R(C_n)$, and $x$ is some point in the total space of $f_!g^*(X_1 \barsmash \ldots \barsmash X_n)$. We will calculate the action of $\eta$ on $x$. Of course if $x$ is in the basepoint section then this action is automatically trivial, so we'll assume we aren't in that case. Then $x$ can be lifted along the following zig-zag of maps of total spaces, because the first backwards (vertical) map is surjective away from the basepoint section, and the second one is surjective everywhere.
	\[ \xymatrix @R=1.2em{
		& X_1 \times \ldots \times X_n \ar[d] \\
		g^*(X_1 \barsmash \ldots \barsmash X_n) \ar[d] \ar[r] & X_1 \barsmash \ldots \barsmash X_n \\
		f_!g^*(X_1 \barsmash \ldots \barsmash X_n) &
	} \]
	The resulting lift $(x_1,\ldots,x_n)$ corresponds to an $n$-tuple of maps $m_i\colon *_{+C_i} \to X_i$ over $C_i$. The natural transformation $\eta$ assigns to this $n$-tuple the commuting square
	\[ \xymatrix @C=6em{
		f_!g^*(*_{+C_1} \barsmash \ldots \barsmash *_{+C_n}) \ar[r]^-{\eta = \id} \ar[d]^-{f_!g^*(m_1 \barsmash \ldots \barsmash m_n)} & f_!g^*(*_{+C_1} \barsmash \ldots \barsmash *_{+C_n}) \ar[d]^-{f_!g^*(m_1 \barsmash \ldots \barsmash m_n)} \\
		f_!g^*(X_1 \barsmash \ldots \barsmash X_n) \ar[r]^-\eta & f_!g^*(X_1 \barsmash \ldots \barsmash X_n).
	} \]
	By construction, the vertical map $f_!g^*(m_1 \barsmash \ldots \barsmash m_n)$ contains the chosen point $x$ in its image.\footnote{The non-obvious part of that claim is that some point $z \in g^*(*_{+C_1} \ldots)$ is sent to our chosen lift $y$ of $x$ living in $g^*(X_1 \ldots)$. To find $z$, you have project $y$ to $A$, giving $a \in g^{-1}(c_1,\ldots)$. Then take $z$ to be the unique nontrivial point of $g^*(*_{+C_1} \ldots)$ lying over $a$. The universal property of the pullback then tells you $z$ is mapped to $y$, so the image of $z$ in $f_!g^*(*_{+C_1} \ldots)$ is sent to $x$.} Since the top horizontal map is the identity, we conclude the bottom horizontal map must send $x$ to $x$. But $x$ was arbitrary, therefore $\eta$ acts as the identity.
\end{proof}

\begin{rmk}
	The same conclusion applies if we restrict the domain of $\Phi$ to any full subcategory that contains all the $n$-tuples of spaces of the form $(*_{+C_1},\ldots,*_{+C_n})$.
\end{rmk}

\begin{cor}
	\begin{itemize}
		\item The functor
		\[ \Phi\colon (X_1,\ldots,X_n) \leadsto h^*k_!(X_1 \barsmash \ldots \barsmash X_n) \]
		is also rigid, with no restrictions on $h$ and $k$.
		\item The adjunction between the base-change functors $(f_! \adj f^*)$ is unique. Any model for $f_!$ and $f^*$ comes with one and only one adjunction between them.
		\item The symmetric monoidal structure on $\mc R$ constructed in \autoref{prop:spaces_sm} is unique. Once we fix $\barsmash$, there is only one possible choice for the associator, unitor, and symmetry isomorphisms.
	\end{itemize}
\end{cor}

\begin{proof}
	\begin{itemize}
		\item Define $f$ and $g$ by pulling back $h$ and $k$. The resulting pullback square of spaces is also a pullback square on the underlying sets, hence $f$ is injective on every fiber of $g$. Along the Beck-Chevalley isomorphism, $\Phi$ is isomorphic to the functor in \autoref{prop:spaces_rigidity}, and is therefore also rigid.
		\item It is standard that any two adjunctions are isomorphic along the identity of $f_!$ and some automorphism of $f^*$. But $f^*$ has no nontrivial automorphisms, hence any two adjunctions must be equal.
		\item All composites of the smash product are rigid, and we know that they are isomorphic, so the isomorphisms must be unique.
	\end{itemize}
\end{proof}

\beforesubsection
\subsection{The symmetric monoidal category $\mc R(B)$}\label{sec:internal}\aftersubsection

Specializing from all retractive spaces $\mc R$ back to the category $\mc R(B)$ of retractive spaces over a fixed base space $B$, we make $\mc R(B)$ into a symmetric monoidal category using the \textbf{internal smash product}\index{internal smash product $X \sma_B Y$}. This is defined by
\[ \sma_B\colon \mc R(B) \times \mc R(B) \to \mc R(B),
\qquad X \sma_B Y = \Delta_B^*(X \barsmash Y), \]
where $\Delta_B\colon B \to B \times B$ is the diagonal map. Applying $\Delta_B^*$ to the definition of the external smash product gives the pushout diagram
\begin{equation}\label{eq:internal_smash}
\xymatrix @R=1.7em{
	X \cup_B Y \ar[d] \ar[r] & X \times_B Y \ar[d] \\
	B \ar[r] & X \sma_B Y.
}
\end{equation}
Informally, $X \sma_B Y$ is a space over $B$ whose fiber over $b$ is the smash product $X_b \sma Y_b$.

By \autoref{prop:spaces_rigidity}, the internal smash product $\sma_B$ is a rigid functor. The commutation of pullbacks $f^*$ with $\barsmash$ allows us to verify that it is associative, unital, and commutative up to some natural isomorphisms. It follows immediately that those natural isomorphisms are unique and coherent, giving:
\begin{lem}\label{prop:spaces_sm}
	The internal smash product makes $\mc R(B)$ into a symmetric monoidal category with unit the retractive space
\[ S^0_B := S^0 \times B \cong r_B^* S^0. \]
\end{lem}

The internal smash product has right adjoints only if we adopt (CG) assumptions \cite{ms}.

The categories $\mc R(B)$, together with the products $\sma_B$ and pullbacks $f^*$, form another structure called an \textbf{indexed symmetric monoidal category with coproducts}\index{indexed category}. This means that pullbacks $f^*$ are strong symmetric monoidal and satisfy a projection formula, see \cite{ms,shulman_framed_monoidal}. This structure is formally equivalent to the statement that $\mc R$ is an SMBF, so we do not lose any generality if we phrase the results in terms of SMBFs.

The internal smash product $\sma_B$ is easier to contemplate than $\barsmash$, but it does not interact as well with cofibrations, fibrations, and weak equivalences.
\begin{cor}\label{cor:internal_smash_properties}\hfill
	\vspace{-1em}
	
	\begin{itemize}
		\item Pushout-products of $f$-cofibrations defined using $\sma_B$ are $f$-cofibrations.
		\item $\sma_B$ preserves spaces that are both $f$-cofibrant and $h$-fibrant, and also weak equivalences between such spaces.
		\item More generally, $\sma_B$ preserves equivalences when both spaces are $f$-cofibrant and at least one of them is $q$-fibrant.
	\end{itemize}
\end{cor}

\begin{proof}
	The first two points follow from \autoref{lem:f_star_preserves} and \autoref{prop:h_cofibrations_pushout_product}. For the last point, assume $X$ and $Y$ are $f$-cofibrant, and $X$ is $q$-fibrant. Since pullbacks and pushout-products preserve $f$-cofibrations, $X \cup_B Y \to X \times_B Y$ is an $f$-cofibration of spaces over $B$, so \eqref{eq:internal_smash} is a homotopy pushout square. The term $X \cup_B Y$ preserves equivalences because $X$ and $Y$ are cofibrant, and $X \times_B Y$ preserves equivalences because it is a pullback along a $q$-fibration. Therefore $X \sma_B Y$ preserves equivalences under these assumptions.
\end{proof}

\begin{rmk}\label{fixed_may_sigurdsson}
	The earlier work \cite{ms} regards $\sma_B$ as the most fundamental smash product, and defines $\barsmash$ in terms of it. This follows the philosophy of \cite{clapp1981duality,crabb_james} that every construction and definition should be a fiberwise version of the classical construction. However, it is better to regard the external smash product $\barsmash$ as the more fundamental notion, because \autoref{prop:h_cofibrations_pushout_product} has both weaker assumptions and stronger conclusions than \autoref{cor:internal_smash_properties}.
\end{rmk}

We briefly mention the \textbf{external smash product rel $B$}\index{external!smash rel $B$ $X \barsma{B} Y$}. When $X \in \mc R(D)$ and $Y \in \mc R(E)$, and $D$ and $E$ are equipped with maps to $B$, this relative external smash product is defined as
\[ \xymatrix @R=0.3em{
	X \barsma{B} Y := \Delta_{D,E}^*(X \barsmash Y), \\
	\Delta_{D,E}\colon D \times_B E \to D \times E.
} \]
It is essentially the external smash product, but carried out separately over every point of $B$. This makes the category of retractive spaces $X$ over spaces $E$ over $B$ into a symmetric monoidal category. The symmetric monoidal category $\mc R(B)$ that we constructed above is contained inside, as the subcategory on which $E \to B$ is the identity map of $B$. See also \autoref{sec:Osp_B}.

\beforesubsection
\subsection{Monoidal fibrant replacement $P$}\aftersubsection

We highlight a construction that is key to our approach to both \autoref{thm:intro_q} and \autoref{thm:intro_cof_fib}. It is a fibrant replacement functor on $\mc R$ that is strong symmetric monoidal with respect to $\barsmash$.

For each real number $t \in [0,1]$, let $p_t\colon B^I \to B$ be the map that evaluates a path in $B$ at the point $t$. We define the \textbf{monoidal fibrant replacement functor}\index{fibrant replacement!$P$ for spaces}
\[ P\colon \mc R(B) \to \mc R(B) \]
by the formula $PX = (p_1)_!(p_0)^*X$. More explicitly, when $X$ is a retractive space over $B$, $PX$ is the retractive space whose total space is the pushout
\[ \xymatrix @R=1.7em @C=6em{
	B^I \ar[r]^-{i_X \times_B \ \id_{B^I}} \ar[d]_-{p_1} \po{.8}{.5ex} & X \times_B B^I \ar[d] \\
	B \ar[r] & PX.
} \]
We also get a natural map $X \to PX$ of retractive spaces over $B$, by including $X$ into $X \times_B B^I$ along the inclusion of constant paths into $B^I$.

\begin{prop}\label{prop:px_properties}\hfill
	\vspace{-1em}
	
	\begin{itemize}
		\item $P$ preserves all $f$-cofibrations and $h$-cofibrations.
		\item If $X$ is $h$-cofibrant then $PX$ is $h$-fibrant and $X \to PX$ is a weak equivalence.
	\end{itemize}
\end{prop}

As a result of this and \autoref{lem:heath_kamps}, every $h$-cofibration $X \to Y$ of $h$-cofibrant spaces gets ``upgraded'' by $P$ to an $f$-cofibration $PX \to PY$ of $f$-cofibrant spaces.

\begin{proof}
\autoref{lem:f_star_preserves} and \autoref{lem:f_shriek_preserves} tell us that the $f$-cofibrations and $h$-cofibrations both survive the journey through the definition of $P$. If $X$ is $h$-cofibrant, \autoref{lem:f_star_preserves} tells us that the top horizontal in the above square is an $h$-cofibration. Since the left vertical is a weak equivalence, therefore so is the right vertical. The first three vertices are all $h$-fibrations over $B$, so by \autoref{prop:clapp} the fourth vertex is also a fibration over $B$.
\end{proof}

In total, each space $X \in \mc R(B)$ can be replaced by an equivalent $q$-cofibrant space $QX$ using the Quillen model structure from \autoref{prop:q}, and this in turn can be replaced by an equivalent space $PQX$ that is both $f$-cofibrant and $h$-fibrant.

\begin{prop}\label{prop:P_strong_monoidal}\cite[4.15]{malkiewich2017coassembly}
	There is a canonical homeomorphism $PX \barsmash PY \cong P(X \barsmash Y)$ of functors $\mc R(A) \times \mc R(B) \to \mc R(A \times B)$.
\end{prop}
\begin{proof}
	By ``canonical'' we mean that both functors come with an identification $F(A_{+A},B_{+B}) \cong (A^I \times B^I)_{+(A \times B)}$, and the map respects that identification. We define the canonical isomorphism by composing the isomorphisms of \autoref{lem:external_smash_and_base_change}, and tracing through the diagrams to check it respects the above identifications.
	
	Alternatively, we prove these maps are also natural with respect to replacing $A^I$ and $B^I$ with other spaces over $A$ and $B$, in which case the resulting self-isomorphism $(A^I \times B^I)_{+(A \times B)}$ is natural in $A^I$ and $B^I$ and so must be the identity by a rigidity argument.
\end{proof}

\begin{ex}
	The canonical homeomorphism from \autoref{prop:P_strong_monoidal} and the unique homeomorphism $P(S^0) \cong S^0$ make $P$ into a strong symmetric monoidal functor.
\end{ex}

Recall the homotopy cofiber construction from \autoref{space_cofiber}.

\begin{lem}\label{commute_cone_with_P}
	There is a canonical isomorphism $PC_B f \cong C_B Pf$.
\end{lem}

\begin{proof}
	This follows from \autoref{prop:P_strong_monoidal} and commuting the pullback $(p_0)^*$ with the pushout in the definition of $C_B$. To see that these commute, recall that in (CG) the pullback functor commutes with all colimits, and in (CGWH) the same is true for the colimit forming $C_B f$ because $B \to X$ and $S^0 \to I$ are closed inclusions.
\end{proof}

\begin{rmk}
	Our definition of $P$ comes from \cite{malkiewich2017coassembly}, and has different properties than the fibrant replacement $L$ defined in \cite[8.3.1]{ms}, which resembles the classical fibrant replacement $RX = X \times_B B^I$ but using Moore paths. The disadvantage of our choice is that we only have a fibrant replacement when the input space is $h$-cofibrant. The advantage of our choice is that $P$ commutes with smash products and therefore with homotopy cofibers. This leads to major simplifications when we set up the bicategory of parametrized spectra.
\end{rmk}

\section{Derived functors}\label{sec:derived}

In this section we recall how derived functors are unique, not just in model categories, but in homotopical categories. We give a new framework for relating composites of derived functors (\autoref{prop:passing_natural_trans_to_derived_functors}), and use it pass from bifibrations with weak equivalences to bifibrations of homotopy categories (\autoref{prop:deform_bifibration}). Our main example is $\ho\mc R$, the bifibration of retractive spaces whose fiber over $B$ is the \emph{homotopy category} $\ho\mc R(B)$ of retractive spaces over $B$. Finally, we prove in \autoref{spaces_are_a_bifibration} that $\ho\mc R$ is a symmetric monoidal bifibration (SMBF). In other words, we show how to derive the operations $\barsmash$, $f^*$, and $f_!$, while preserving all of the relationships between them.

\beforesubsection
\subsection{Uniqueness of derived functors}\label{sec:derived_functors_intro}\aftersubsection

Let $\bC$ be a category with a class of weak equivalences $W$. We assume at a minimum that $W$ is closed under 2 out of 3: if $f$ and $g$ are composable morphisms in $\bC$, and two of the three morphisms $f$, $g$, and $g \circ f$ are in $W$, then so is the third. Let $\ho\bC = \bC[W^{-1}]$ be the category in which the weak equivalences are inverted. This is guaranteed to exist modulo set-theory issues by \cite[1.1]{gabriel_zisman}. We may take the objects of $\ho \cat C$ to be the objects of $\cat C$, and we do so throughout this paper.

Let $\bC$ and $\bD$ be two such categories with weak equivalences. A functor $F\colon \bC \to \bD$ is \textbf{homotopical}\index{homotopical functor} $F\colon \cat C \to \cat D$ if it sends every weak equivalence to a weak equivalence. By the universal property of $\ho \cat C$, any homotopical functor $F$ will induce in a canonical way a functor $\ho\bC \to \ho\bD$. By abuse of notation we also refer to this as $F$, or ``$F$ on the homotopy category.''

When $F$ fails to be homotopical, we seek to replace it by a derived functor. Recall from \cite[40.2]{dhks} that $F\colon \cat C \to \cat D$ is \textbf{left-deformable}\index{left-deformable!functor} if we can find a full subcategory $\cat A \subset \cat C$ on which $F$ preserves weak equivalences, a ``cofibrant replacement'' functor $Q\colon \bC \to \bC$ whose image lands in $\bA$, and a natural weak equivalence $q\colon QX \overset\sim\to X$. In this case $F$ has a \textbf{left-derived functor}\index{derived functor}, defined as
\[ \L F = F \circ Q. \]
The left-derived functor satisfies two universal properties, which imply that it does not actually depend on the choice of retraction functor $Q$.

\begin{prop}\label{prop:derived_functors_universal_property}\cite[6.4]{riehl_basic}, \cite[41.1, 41.2]{dhks}
	If one such $\cat A$ and $Q$ exist, then $\ho\L F$ is terminal among all those functors $\ho\cat C \to \ho\cat D$ admitting a map to $F$ as functors $\cat C \to \ho\cat D$. Likewise $\L F$ is terminal in comma category of homotopy functors mapping to $F$, with the natural weak equivalences inverted.
\end{prop}

So, in either of these two settings, any two models of $\L F$ are canonically isomorphic. In fact, any map between them must be equal to the canonical isomorphism, as soon as we know that the map commutes with the map back to $F$.

Similarly, if there is a functor $R\colon \bC \to \bC$ whose image lands in $\bA$, and a natural weak equivalence \emph{from} the identity $r\colon X \overset\sim\to RX$, we say that $f$ is \textbf{right-deformable}\index{right-deformable!functor} and define the \textbf{right derived functor}\index{derived functor} to be
\[ \R F = F \circ R. \]
This functor, and its image in the homotopy category, satisfy the duals of the above universal properties, so they are also unique up to canonical isomorphism.

\begin{ex}
	Any left Quillen functor is left-deformable, and any right Quillen functor is right-deformable. However, because of \autoref{prop:derived_functors_universal_property}, we don't have to use the cofibrant or fibrant replacement coming from the model structure. This is particularly important when we start composing these derived functors with each other.
\end{ex}

\begin{rmk}
	If $F$ has neither a left nor a right derived functor, it may have a ``middle derived functor'' $\mathbb M F$\index{middle-deformable}\index{derived functor} obtained by applying a sequence of left and right deformation retracts. For instance, the internal smash product $\sma_B$ is middle-deformable. Unfortunately, middle-derived functors are not unique in general, they depend explicitly on which retractions we choose.
\end{rmk}

\begin{warn}\label{warn:changing_equivs}
	If we change the class of weak equivalences $W$, the derived functors $\L F$ or $\R F$ can change. For instance let $G$ be a finite group, $\cat C$ the category of $G$-spaces, and $F(X) = X^G$, the $G$-fixed points of $X$. If we derive $F$ using the maps inducing equivalences on all fixed-point subspaces $X^H \overset\sim\to Y^H$, we just get $X^G$ again. But if we derive $F$ using the maps inducing equivalences on the underlying space $X \overset\sim\to Y$, we get the homotopy fixed points $X^{hG}$.
\end{warn}

\beforesubsection
\subsection{Composing and comparing derived functors}\label{sec:composing_comparing}\aftersubsection

We need to lift the isomorphisms that commute $f_!$, $f^*$, and $\barsmash$ to weak equivalences that commute the derived functors $\L f_!$, $\R f^*$, and $\barsmash^{\L}$, in a canonical way. We describe a new theoretical framework for doing this, since the approach of \cite{ms} relies on choosing a class of cofibrant-fibrant objects, but we don't wish to have the theory depend on such choices. We take motivation from \cite{shulman_comparing} and adapt the idea to accommodate longer strings of functors.

Suppose we have a composable list of functors of homotopical categories
\begin{equation}\label{composable_functors}
\xymatrix{
	\cat C \ar@{=}[r] & \cat C_0 \ar[r]^-{F_1} & \cat C_1 \ar[r]^-{F_2} & \ldots \ar[r]^-{F_n} & \cat C_n \ar@{=}[r] & \cat C'
}
\end{equation}
and each functor $F_i$ is left or right deformable.\footnote{If $F_i$ is both left and right deformable and the canonical map $\L F_i \to F_i \to \R F_i$ is \emph{not} an equivalence, fix a choice of whether to look at $\L F_i$ or $\R F_i$. If the canonical map is an equivalence (for instance if $F_i$ is already a homotopy functor) then the choice does not matter for the discussion that follows.} Let $\D F_i$ refer to the left or right derived functor of $F_i$, and $\D F_i \leftrightarrow F_i$ the canonical map in the homotopy category of functors.

Our essential task is to take an isomorphism between two such composites, and deduce that the composites of the derived functors $\D F_n \circ \ldots \circ \D F_2 \circ \D F_1$ are also equivalent. This is not possible in general, even if all the functors are left-deformable. For instance, the colimit functor is a composite of a coproduct and a coequalizer:
\[ \colim \cong \textup{coeq} \circ \amalg. \]
But this isomorphism of point-set functors does not become an equivalence on the composites of left-derived functors:
\[ \L \colim \not\simeq \L(\textup{coeq}) \circ \L(\amalg). \]
In other words, the homotopy colimit is not the homotopy coequalizer of the coproduct.

Therefore, we must impose a condition on the functors in \eqref{composable_functors} if we want some control over the composite of their derived functors.

\begin{df}\label{coherently_deformable}
We say the list of composable functors in \eqref{composable_functors} is \textbf{coherently deformable}\index{coherently deformable} if there exists a full subcategory $\cat A_{i-1} \subseteq \cat C_{i-1}$ for each $1 \leq i \leq n$ such that
\begin{itemize}
	\item $\cat A_{i-1}$ contains the image of some composite of cofibrant and fibrant replacement functors, i.e. functors with a weak equivalence to or from the identity, on $\cat C_{i-1}$,
	\item on $\cat A_{i-1}$, the canonical map $\D F_i \leftrightarrow F_i$ is an equivalence,
	\item and $F_i(\cat A_{i-1}) \subseteq \cat A_i$.
\end{itemize}
Given another such list
\begin{equation}\label{composable_functors_2}
\xymatrix{
	\cat C \ar@{=}[r] & \cat D_0 \ar[r]^-{G_1} & \cat D_1 \ar[r]^-{G_2} & \ldots \ar[r]^-{G_k} & \cat D_k \ar@{=}[r] & \cat C',
}
\end{equation}
we say that \textbf{the pair \eqref{composable_functors}, \eqref{composable_functors_2} is coherently deformable} if there exists some coherent deformation for each list, with the same choice of subcategory $\cat A_0 \subseteq \cat C$.
\end{df}

\begin{prop}[Canonical recipe for deriving a natural transformation]\label{prop:passing_natural_trans_to_derived_functors}\hfill
	Given a pair of coherently deformable composites as above, for each natural transformation
	\[ \eta\colon F_n \circ \ldots \circ F_1 \Rightarrow G_k \circ \ldots \circ G_1 \]
	there is a unique morphism in the homotopy category of homotopy functors, or of functors on the homotopy category,
	\[ \D \eta\colon \D F_n \circ \ldots \circ \D F_1 \Rightarrow \D G_k \circ \ldots \circ \D G_1 \]
	that when restricted to $\cat A_0$ agrees with
	\begin{align*}
	& \D F_n \circ \ldots \circ \D F_2 \circ \D F_1 \\
	\simeq\ & \D F_n \circ \ldots \circ \D F_2 \circ F_1 \\
	\simeq\ & \D F_n \circ \ldots \circ F_2 \circ F_1 \\
	\simeq\ & F_n \circ \ldots \circ F_2 \circ F_1 \\
	\overset\eta\to\ & G_k \circ \ldots \circ G_2 \circ G_1 \\
	\simeq\ & \D G_k \circ \ldots \circ G_2 \circ G_1 \\
	\simeq\ & \D G_k \circ \ldots \circ \D G_2 \circ G_1 \\
	\simeq\ & \D G_k \circ \ldots \circ \D G_2 \circ \D G_1.
	\end{align*}
	Furthermore $\D \eta$ does not depend on the choice of $\cat A_0$.
\end{prop}

\begin{proof}
	Restricting from $\cat C_0$ to $\cat A_0$ induces an equivalence on the homotopy category of homotopy functors to $\cat D$. This is true because the given cofibrant and fibrant replacements provide the inverse equivalence. Therefore the above recipe on $\cat A_0$ determines a unique morphism $\D \eta$ of functors on all of $\cat C_0$.
	
	If $\cat A_0$ and $\cat A_0'$ are two possible subcategories arising in the coherent deformability condition, then so is the union $\cat A_0 \cup \cat A_0'$. Furthermore, $\D \eta$ as calculated on $\cat A_0 \cup \cat A_0'$ clearly agrees with $\D\eta$ as calculated on $\cat A_0$ or $\cat A_0'$. So $\D\eta$ does not depend on the choice of $\cat A_0$.
\end{proof}

\begin{rmk}
	If $\eta$ is an equivalence on some (therefore on all) choices of $\cat A_0$, $\D\eta$ is an isomorphism. This recipe also respects horizontal and vertical compositions of natural tranformations, whenever all of the terms are defined. (So it is a partially-defined 2-functor of partially-defined 2-categories.)
\end{rmk}

\begin{ex}\label{ex:derived_external_smash_and_base_change}
	Given maps $f: A \to A'$ and $g: B \to B'$, the two composites
	\[ \xymatrix{
		\mc R(A) \times \mc R(B) \ar[r]^-\barsmash \ar[d]_-{f_! \times g_!} & \mc R(A \times B) \ar[d]^-{(f \times g)_!} \\
		\mc R(A') \times \mc R(B') \ar[r]^-\barsmash & \mc R(A' \times B')
	} \]
	are coherently deformable, using for example the $h$-cofibrant spaces for all four categories. Therefore the canonical interchange isomorphism from \autoref{lem:external_smash_and_base_change} induces an equivalence of derived functors that is canonical in the homotopy category,
	\[ \L f_!X \barsmash^{\L} \L g_!Y \simeq \L (f \times g)_!(X \barsmash^{\L} Y). \]
	Similarly the two composites
	\[ \xymatrix{
		\mc R(A') \times \mc R(B') \ar[r]^-\barsmash \ar[d]_-{f^* \times g^*} & \mc R(A' \times B') \ar[d]^-{(f \times g)^*} \\
		\mc R(A) \times \mc R(B) \ar[r]^-\barsmash & \mc R(A \times B)
	} \]
	are coherently deformable, using for example the $f$-cofibrant and $h$-fibrant spaces for all four categories. We therefore also get a canonical equivalence
	\[ \R f^*X \barsmash^{\L} \R g^*Y \simeq \R(f \times g)^*(X \barsmash^{\L} Y). \]
	Both of these equivalences pass to canonical natural isomorphisms of functors on the homotopy categories $\ho\mc R(-)$, where they agree with the maps defined in \cite[9.4.1]{ms}.
\end{ex}

\begin{ex}\label{ex:derived_beck_chevalley}
	Given a commuting square of unbased spaces and ``rotated'' square of functors
	\[ \xymatrix@R=1em @C=1em{
		& A \ar[ld]_-g \ar[rd]^-f & \\
		B \ar[rd]_-p && C \ar[ld]^-q \\
		& D &
	} \qquad
	\xymatrix@R=1em @C=-.5em{
		& \mc R(A) \ar@{<-}[ld]_-{g^*} \ar[rd]^-{f_!} & \\
		\mc R(B) \ar[rd]_-{p_!} && \mc R(C) \ar@{<-}[ld]^-{q^*} \\
		& \mc R(D) &
	} \]
	if either $p$ or $q$ is an $h$-fibration then the functors are coherently deformable. The Beck-Chevalley isomorphism of \autoref{prop:beck_chevalley_spaces} therefore gives an equivalence of derived functors (cf. \cite[9.4.6]{ms})
	\[ \L f_! \circ \R g^* \simeq \R q^* \circ \L p_!. \]
\end{ex}

\begin{ex}
	Consider the category whose objects are diagrams of the form
	\[ \xymatrix{ X_0 \ar[r] & X_1 \ar[r] & X_2 \ar[r] & \ldots } \]
	where each map is a closed inclusion, each space $X_i$ has a $\Z/2$-action, and the maps are $\Z/2$-equivariant. There are two operations we can perform, ``colimit'' and ``$\Z/2$-fixed points,'' that commute with each other. We capture this in the commuting square:
	\[ \xymatrix @C=4em{
		\cat{Top}^{\N \times \Z/2} \ar[r]^-{\underset{\N}\colim} \ar[d]_-{\underset{\Z/2}\lim} &
		\cat{Top}^{\Z/2} \ar[d]^-{\underset{\Z/2}\lim} \\
		\cat{Top}^{\N} \ar[r]_-{\underset{\N}\colim} &
		\cat{Top}
	} \]
	The colimits can be left-deformed and the limits can be right-deformed, but the square of functors is not coherently deformable.
	
	To prove this, we take $X_n = \mathbb{RP}^n$ to be the $n$-skeleton of $\mathbb{RP}^\infty$ with the trivial $\Z/2$-action. The two composites of derived functors give
	\[ \hocolim \left((\mathbb{RP}^n)^{h\Z/2}\right) \quad \textup{and} \quad \left( \hocolim \mathbb{RP}^n \right)^{h\Z/2}. \]
	These are not weakly equivalent,\footnote{The framework of \cite{shulman_comparing} gives us a map between them, but the map is not an equivalence.} because by the Sullivan conjecture \cite{miller_sullivan},
	\[ \hocolim \Map(\mathbb{RP}^\infty,\mathbb{RP}^n) \simeq \hocolim * \simeq * \not\simeq \Map(\mathbb{RP}^\infty,\mathbb{RP}^\infty). \]
	Therefore the point-set isomorphism commuting the colimit with the fixed points does not pass to the derived functors, which implies that the functors are not coherently deformable.
\end{ex}

The ``coherently deformable'' condition can be strengthened.
\begin{df}\label{coherently_left_deformable}
Given coherently deformable functors as in \autoref{coherently_deformable}, if each functor $F_i$ is left-deformable and each subcategory $\cat A_{i-1}$ contains the image of a single cofibrant replacement $Q_{i-1}$, we say the functors are \textbf{coherently left deformable}\index{coherently left/right deformable}. A coherently right deformable list is defined similarly.
\end{df}

This stronger condition is called ``left deformability of a pair'' in \cite[42.3]{dhks}. It implies the following almost immediately.
\begin{lem}\cite[42.4]{dhks}
	If the list of functors \eqref{composable_functors} is coherently left deformable, then the composite functor $F_n \circ \ldots \circ F_1$ is also left-deformable and the canonical map
	\[ \xymatrix{ \L F_n \circ \ldots \circ \L F_1 \ar[r]^-\sim & \L(F_n \circ \ldots \circ F_1) } \]
	is an equivalence. The dual statement holds for coherently right-deformable functors: the composite is also right-deformable and we get a canonical equivalence
	\[ \xymatrix{ \R(F_n \circ \ldots \circ F_1) \ar[r]^-\sim & \R F_n \circ \ldots \circ \R F_1. }  \]
\end{lem}

\begin{ex}
	Any string of left Quillen functors is coherently left-deformable, and any string of right Quillen functors is coherently right-deformable. This is why composition respects passage to derived functors ``most of the time.''
\end{ex}

\begin{rmk}
	In \autoref{coherently_deformable}, if the two lists of functors \eqref{composable_functors}, \eqref{composable_functors_2} are both coherently left-deformable, it is no longer necessary to ask for a common subcategory $\cat A_0 \subseteq \cat C$. We get the natural transformation in \autoref{prop:passing_natural_trans_to_derived_functors} from the simpler fact that any natural transformation $\eta\colon F = F_n \circ \ldots \circ F_1 \Rightarrow G_k \circ \ldots \circ G_1 = G$ induces a canonical map of left-derived functors $\L F \Rightarrow \L G$.
\end{rmk}

\beforesubsection
\subsection{Deriving bifibrations; the SMBF $\ho\mc R$}\label{sec:deriving_SMBFs}\aftersubsection

We are now ready to pass the symmetric monoidal bifibration structure on $\mc R$ to a similar structure on $\ho\mc R$. In other words, to derive the operations $\barsmash$, $f^*$, and $f_!$ while preserving all of the relationships between them. We begin by handling the symmetric monoidal structure on its own.

Suppose $(\mc C,\otimes,I)$ is a symmetric monoidal category with a class of weak equivalences. We say that the symmetric monoidal structure is \textbf{left-deformable}\index{left-deformable!symmetric monoidal category} if there is a full subcategory $\cat A$ of ``cofibrant'' objects, containing the image of a cofibrant replacement functor $Q$ in $\mc C$, such that
\begin{enumerate}
	\item[(SM1)] $\otimes$ preserves cofibrant objects and weak equivalences between them,
	\item[(SM2)] and $I$ is cofibrant.\footnote{This is equivalent to the weaker condition that $QI \otimes QX \to I \otimes QX$ is an equivalence for every $X$. (Given this weaker condition, just enlarge $\cat A$ to include $I$. Note that condition (SM1) is preserved.)}
\end{enumerate}

\begin{ex}
	The symmetric monoidal category $(\mc R,\barsmash,S^0)$ is left-deformable, using for instance the $h$-cofibrant spaces and the whiskering functor $W$.
\end{ex}

\begin{prop}\label{prop:deform_sym_mon_cat}
	If $(\mc C,\otimes,I)$ is left-deformable then there is a canonical symmetric monoidal structure on $\ho\mc C$ whose product is $\otimes^{\L}$ and whose unit is $I$.
\end{prop}

\begin{proof}
	This is well-known but we can give a proof by repeated application of \autoref{prop:passing_natural_trans_to_derived_functors}. The isomorphisms $\alpha,\lambda,\rho,\gamma$ ascend to the homotopy category because in each case the lists of functors are coherently left-deformable, using the subcategory $\cat A$ in every case. They are coherent because they were coherent in $\mc C$, and the recipe in \autoref{prop:passing_natural_trans_to_derived_functors} respects composition of natural transformations and identity natural transformations.
\end{proof}

\begin{cor}
	Suppose $\mc C$ is both a model category and a symmetric monoidal category, the tensor product $\otimes$ is a left Quillen bifunctor, and the unit $I$ is cofibrant. Then $\ho\mc C$ has a canonical left-derived symmetric monoidal structure.
\end{cor}

Similarly, if $\mc C$ is a bifibration over $\bS$ and each fiber $\mc C^A$ has a class of weak equivalences, then we can form $\ho\mc C$ by inverting the weak equivalences in every fiber category $\mc C^A$. The universal property gives a functor $\ho\mc C \to \bS$, and we may ask whether this is a bifibration as well.

We say that $\mc C$ is \textbf{bi-deformable}\index{deformable!bifibration} if the following three conditions hold.
\begin{enumerate}
\item[(BF3)] The pushforwards are coherently left-deformable: there exists a simultaneous choice of cofibrant replacement $Q_A$ on each fiber category $\mc C^A$ under which the pushforwards $f_!$ preserve cofibrant objects and weak equivalences between them.
\item[(BF4)] The pullbacks are coherently right-deformable: there exists a simultaneous choice of fibrant replacement $R_A$ on each fiber category $\mc C^A$ under which the pullbacks $f^*$ preserve fibrant objects and weak equivalences between them.
\item[(BF5)] We choose a class of \textbf{homotopy Beck-Chevalley}\index{homotopy!Beck-Chevalley square} squares. Each one must satisfy the Beck-Chevalley condition in $\mc C$, and in addition, the two lists of functors in the Beck-Chevalley isomorphism must be coherently deformable.
\end{enumerate}

\begin{thm}\label{prop:deform_bifibration}
	If the bifibration $\mc C \to \bS$ is bi-deformable, then $\ho\mc C \to \bS$ is a bifibration as well. The fiber categories are canonically isomorphic to the homotopy categories of the fibers,
	\[ (\ho\mc C)^A \cong \ho(\mc C^A). \]
	Finally, each of the chosen homotopy Beck-Chevalley squares has the Beck-Chevalley condition in $\ho\mc C$. 
\end{thm}

\begin{proof}
	The first step is to show that $\ho\sC$ is a fibration with the desired fiber categories. For this, it suffices to define a second fibration $\bH \to \bS$ and an isomorphism of categories $\ho\sC\to \bH$ over $\bS$,
	such that the induced dotted map in the diagram
	\[\xymatrix @R=1.7em{\sC^A \ar[r]\ar[d]& (\ho\sC)^A \ar[r]&\bH^A\\\ho(\sC^A) \ar@{-->}[urr]}\]
	is an isomorphism for all $A$ in $\bS$.

	We build the fibration $\bH$ by a Grothendieck construction. For each $A \in \bS$ we take the homotopy category $\ho(\mc C^A)$. For each map $f\colon A \to B$ we take the derived pullback functor $f^*R_B\colon \ho(\mc C^B) \to \ho(\mc C^A)$. We extend this to an indexed category\index{indexed category} by picking composition isomorphisms and unit isomorphisms between these derived functors (see e.g. \cite[B1.2]{johnstone2002sketches}). We get them by right-deriving the same isomorphisms for the strict pullback functors in $\sC$, in other words using the maps
	\[ f^*R_Bg^*R_C \overset\sim\leftarrow f^*g^*R_C \cong (g \circ f)^*R_C \]
	\[ \id_A^* R_A \overset\sim\leftarrow \id_A^* \cong \id_{\sC^A}. \]
	These isomorphisms have to satisfy certain coherence conditions, but they are the same conditions the strict pullbacks already satisfy, so the coherence passes automatically to the derived pullbacks $\R f^*$ by the universal property of right-derived functors. By the general theory of Grothendieck constructions, $\bH \to \bS$ is a fibration.
	
	Concretely, the objects of $\bH$ are the objects of $\sC$, and a morphism from $X$ to $Y$ over $f\colon A \to B$ is a zig-zag in $\sC^A$ from $X$ to $f^*R_B Y$. The composition of morphisms in $\bH$ is given by
	\[ \xymatrix @R=1em{
		X \ar@{<->}[r] & f^*R_BY \ar@{<->}[r] & f^*R_Bg^*R_CZ \ar@{<-}[d]^-\sim \\
		&& f^*g^*R_CZ  \ar@{<-}[d]^-\cong \\
		&& (g \circ f)^*R_CZ.
	} \]
	The identity morphism at $X$ is $X \cong \id^* X \overset\sim\to \id^* R_A X$.
	
	Define $\sC \to \bH$ by taking each map $X \to Y$ over $f$, represented by a vertical map $X \to f^*Y$, to the composite $X \to f^*Y \to f^*R_BY$. Chasing diagrams, we check
	\begin{itemize}
		\item this respects composition and units, hence is a functor,
		\item it sends weak equivalences to isomorphisms, hence defines a functor $\ho\sC \to \bH$,
		\item the resulting map $\ho(\sC^A) \to (\ho\sC)^A \to \bH^A = \ho(\sC^A)$ is the identity,
		\item $\bH$ is generated by morphisms of the form $X \to X' \overset\sim\to \id^*R_A X'$, $X \overset\sim\leftarrow X' \overset\sim\to \id^*R_A X'$, and $f^*R_B Y = f^*R_B Y$, and 
		\item these are each in the image of $\ho\sC$, hence $\ho\sC \to \bH$ is surjective.
	\end{itemize}
	
	Then we map $\bH \to \ho\sC$ by sending a zig-zag $X \leftrightarrow f^*R_B Y$ to the zig-zag of maps of $\sC$, $X \leftrightarrow f^*R_B Y \to R_B Y \leftarrow Y$. We check
	\begin{itemize}
		\item this respects the equivalence relation that defines the fibers of $\bH$, hence is well-defined,
		\item this defines a functor $\bH \to \ho\sC$, and 
		\item the composite $\sC \to \bH \to \ho\sC$ is equal to the canonical inclusion.
	\end{itemize}
	Hence $\ho \sC \to \bH$ is also injective, and is therefore an isomorphism of categories.
	
	We deduce that there is a cartesian arrow in $\ho\mc C$ of the form $f^*RX \to RX \overset\sim\leftarrow X$ for each $X$ and $f$, identifying $f^*R$ with the pullback functor for $\ho\mc C$. Dualizing everything, $\ho\mc C$ has co-cartesian arrows of the form $X \overset\sim\leftarrow QX \to f_!QX$ that identify $f_!Q$ with the pushforward functor for $\ho\mc C$.
		
	Therefore each derived pushforward $f_!Q$ forms an adjunction with the derived pullback $f^*R$. The unit is the unique map $X \to f^*Rf_!QX$ lying over the map from each of these two terms to $f_!QX$ (one of which is co-cartesian and the other of which is cartesian):
	\[ \xymatrix @R=1.7em{
		X & \ar[l]_-\sim QX \ar@{-->}[d] \ar[rd] & \\
		& f^*f_!QX \ar@{-->}[d] \ar[r] & f_!QX \ar[d]^-\sim \\
		& f^*Rf_!QX \ar[r] & Rf_!QX
	} \]
	The dotted lines are filled in by the universal property of cartesian arrows in $\mc C$. Therefore the unit is the zig-zag $X \overset\sim\leftarrow QX \to f^*f_!QX \to f^*Rf_!QX$, and further the $Q$s can be dropped if $X$ is cofibrant. Similarly the counit is $f_!Qf^*RY \to f_!f^*RY \to RY \overset\sim\leftarrow Y$, and the $R$s can be dropped if $Y$ is fibrant.
	
	For each homotopy Beck-Chevalley square, the Beck-Chevalley map in $\mc C$ is an isomorphism by assumption. Because the functors are coherently deformable, as in \autoref{ex:derived_beck_chevalley} this gives an isomorphism of their derived functors on the homotopy category. It remains to show this isomorphism agrees with the Beck-Chevalley map that is constructed out of the adjunctions on the homotopy category. We check this by restricting to inputs in the common category on which $p_!$ and $g^*$ are derived, and inspecting \autoref{fig:derive_BC}.
\begin{figure}
\centering
\[ \resizebox{\textwidth}{!}{\xymatrix @C=-1em{
		&& q^*p_! \ar[rrr]^-r &&& q^*Rp_! &&& q^*Rp_!Q \ar[lll]_-q^-\sim \\
		&&& f_!g^*q^*p_! \ar@{-}[dd]^-\cong \ar[ul]_-\epsilon \ar[rrr]^-r &&& f_!g^*q^*Rp_! \ar@{-}[dd]^-\cong \ar[ul]_-\epsilon &&& f_!g^*q^*Rp_!Q \ar[ul]_-\epsilon \ar[lll]_-q^-\sim \\
		&&&&&&& f_!Qg^*q^*Rp_! \ar@{-}[dd]^-\cong \ar[ul]_-q &&& f_!Qg^*q^*Rp_!Q \ar@{-}[dd]^-\cong \ar[ul]_-q \ar[rd]^-r_-\sim \ar[lll]_-q^-\sim \\
		f_!g^* \ar[rrr]^-\eta &&& f_!g^*q^*p_! \ar[rrr]^-r &&& f_!g^*q^*Rp_! &&& && f_!Qg^*Rq^*Rp_!Q \\
		& f_!Qg^* \ar[rrr]^-\eta \ar[ul]_-q \ar[rd]^-r_-\sim &&& f_!Qg^*q^*p_! \ar[ul]_-q \ar[rd]^-r \ar[rrr]^-r &&& f_!Qg^*q^*Rp_! \ar[ul]_-q \ar[rd]^-r_-\sim &&& f_!Qg^*q^*Rp_!Q  \ar[rd]^-r_-\sim \ar[lll]_-q^-\sim \\
		&& f_!Qg^*R \ar[rrr]^-\eta &&& f_!Qg^*Rq^*p_! \ar[rrr]^-r &&& f_!Qg^*Rq^*Rp_! &&& f_!Qg^*Rq^*Rp_!Q \ar[lll]_-q^-\sim \\
	}} \]
\caption{Identifying the derived Beck-Chevalley map.}\label{fig:derive_BC}
\end{figure}
\end{proof}

\begin{rmk}\label{expand_beck_chevalley}
	By standard properties of Beck-Chevalley isomorphisms, the Beck-Chevalley condition is preserved any time we replace a square in $\bS$ by a ``weakly equivalent'' square, so long as for each ``weak equivalence'' the functors $\L f_!$ and $\R f^*$ induce equivalences on the fiber categories $\ho\mc C^A$. This allows us to further expand the class of squares for which $\ho\mc C$ satisfies the Beck-Chevalley condition.
\end{rmk}

The following results fall out of the proof of the previous theorem.\index{homotopy!(co-)cartesian arrow}
\begin{prop}\label{prop:homotopy_cartesian_arrows}(cf. \cite{harpaz_prasma_grothendieck})
	An arrow is cartesian in $\ho\mc C$ iff it is isomorphic to a cartesian arrow in $\mc C$ whose target is fibrant. A cartesian arrow $f^*Y \to Y$ in $\mc C$ is cartesian in $\ho\mc C$ iff the map $f^*Y \to f^*RY$ is an isomorphism in $\ho(\mc C^A)$. The dual statements apply to the co-cartesian arrows.
\end{prop}

A particularly important special case is that of a bifibration $\mc C$ over the walking arrow $\bullet \to \bullet$. This is the same thing as an adjunction between two categories.

\begin{cor}(cf. \cite[43.2]{dhks}, \cite[6.4.13]{riehl_basic})
	If $(F \adj G)$ is an adjunction in which $F$ is left deformable and $G$ is right deformable, then the derived functors form an adjunction $(\L F \adj \R G)$ in a canonical way.\index{deformable!adjunction}
\end{cor}

In particular, this gives the usual construction that passes from a Quillen adjunction to the corresponding adjuction on the homotopy category. This sort of canonical approach also appears in \cite{hinich2018so}.

Now suppose that $\mc C$ is a symmetric monoidal bifibration over $\bS$, and that each fiber category $\mc C^A$ has a class of weak equivalences. Suppose also that $\mc C$ has a left deformation $Q$ for $\boxtimes$ that is fiberwise over $\bS$, meaning that each map $QX \simar X$ lies over an identity map, and that $\boxtimes$ preserves the equivalences in some full subcategory of each fiber $\mc A^A \subseteq \mc C^A$ containing the image of $Q$.

The five conditions (SM1), (SM2), (BF3), (BF4), (BF5) are enough to make $\ho\mc C$ a bifibration with a symmetric monoidal structure. It remains to make the derived tensor product $\boxtimes^{\L}$ preserve cartesian and cocartesian arrows in $\ho\mc C$. For this we need the following final two conditions.
\begin{enumerate}
	\item[(SMBF6)] Suppose that $\boxtimes$ and \emph{all} of the pushforwards $f_!$ are coherently deformable, meaning there is a subcategory $\mc A$ containing a finite composite of fiberwise cofibrant and fibrant replacements in $\mc C$ (``fiberwise'' means they lie over the identity functor of $\bS$), preserved by $\boxtimes$ and each $f_!$, on which both functors are equivalent to their derived functors,
	\item[(SMBF7)] and the same condition for $\boxtimes$ and all of the pullbacks $f^*$ (the category $\mc A$ can be different).
\end{enumerate}
These are not necessarily stronger than (BF4) and (BF5), because for instance there does not have to be a single fibrant replacement functor that deforms the pullbacks and $\boxtimes$. We are allowed to use a composite of cofibrant and fibrant replacements.\index{deformable!SMBF}
\begin{prop}\label{prop:deform_SMBF}
	Under these assumptions, $\boxtimes^{\L}$ preserves cartesian and cocartesian arrows in $\ho\mc C$.
\end{prop}
\begin{proof}
	We give two proofs. By \autoref{prop:homotopy_cartesian_arrows}, a co-cartesian arrow in $\ho\mc C$ is isomorphic to a co-cartesian arrow of the form $Q_A X \to f_!Q_A X$, where $Q_A$ is the left deformation used for $f_!$. Taking $Y \in \mc A^A$ isomorphic to $X$ in $\ho\mc C^A$, we get a zig-zag of weak equivalences $Y \overset\sim\leftarrow Q_A Y \ldots Q_A X$ in $\mc C^A$, all of which are preserved by $f_!$. Therefore without loss of generality each of our co-cartesian arrows has source in $\mc A$. Since $\boxtimes$ preserves equivalences on $\mc A$, we have $\boxtimes^{\L} \simeq \boxtimes$ on this subcategory, hence $\boxtimes^{\L}$ of our two co-cartesian arrows is equivalent to $\boxtimes$ of them, which is co-cartesian in $\mc C$. The tensor product of the sources is still in $\mc A$, hence $(f\times g)_!Q_{A \times B} \simeq (f\times g)_!$ on this source, so this co-cartesian arrow is also co-cartesian in $\ho\mc C$. The proof for cartesian arrows is dual.
	
	The second proof is more explicit, but establishes the helpful corollary that the canonical isomorphisms
\begin{align*}
\L f_!X \boxtimes^{\L} \L g_!Y &\simeq \L(f \times g)_!(X \boxtimes^{\L} Y), \\
\R f^*X \boxtimes^{\L} \R g^*Y &\simeq \R(f \times g)^*(X \boxtimes^{\L} Y)
\end{align*}
	arising from the universal property of co/cartesian arrows in $\ho\mc C$ agree with the isomorphisms produced by applying \autoref{prop:passing_natural_trans_to_derived_functors} to the corresponding isomorphisms in $\mc C$. For the co-cartesian arrows, we assume $X$ and $Y$ are in $\mc A$ and form the following diagram in which the top row is the canonical co-cartesian arrow in $\ho\mc C$ for $X \boxtimes^{\L} Y$, and the left-hand column is $\boxtimes^{\L}$ of the canonical co-cartesian arrows for $X$ and $Y$. The bottom-right route is the isomorphism arising from \autoref{prop:passing_natural_trans_to_derived_functors}. All maps marked $\sim$ or $\cong$ lie over identity maps in $\bS$, and all others lie over $f \times g$. The commutativity of the diagram establishes that the isomorphism from \autoref{prop:passing_natural_trans_to_derived_functors} extends to an isomorphism of co-cartesian arrows, hence agrees with the canonical isomorphism arising from $\ho\mc C$.
	\[ \xymatrix @R=1.7em{
		QX \boxtimes QY \ar[rd]^-\sim & Q_{A \times B}(Q X \boxtimes Q Y) \ar[l]_-\sim \ar[d]^-\sim \ar[r] & (f \times g)_! Q_{A \times B}(Q X \boxtimes Q Y) \ar[d]^-\sim \\
		QQ_A X \boxtimes QQ_B Y \ar[u]_-\sim \ar[d] \ar[r]^-\sim & X \boxtimes Y \ar[d] \ar[r] & (f \times g)_! (X \boxtimes Y) \ar@{<->}[ld]^-\cong \\
		Qf_!Q_A X \boxtimes Qg_!Q_B Y \ar[r]^-\sim & f_!X \boxtimes g_!Y
	} \]
	For the pullbacks the argument is exactly the same, but the commuting diagram is as follows, with the two cartesian arrows we wish to compare running along the bottom and right-hand sides.
	\[ \xymatrix @R=1.7em @C=3em{
		f^*X \boxtimes g^*Y \ar@{<->}[d]^-\cong \ar[rd] & Qf^*X \boxtimes Qg^*Y \ar[l]_-\sim \ar[r]^-\sim \ar[rddd] & Qf^*R_A X \boxtimes Qg^*R_B Y \ar[dd] \\
		(f \times g)^* (X \boxtimes Y) \ar[d]^-\sim \ar[r] & X \boxtimes Y \ar[d]^-\sim & \\
		(f \times g)^* R_{A \times B}(X \boxtimes Y) \ar[r] & R_{A \times B}(X \boxtimes Y) & QR_A X \boxtimes QR_B Y \\
		(f \times g)^* R_{A \times B}(QX \boxtimes QY) \ar[u]_-\sim \ar[r] & R_{A \times B}(QX \boxtimes QY) \ar[u]_-\sim & QX \boxtimes QY \ar[u]_-\sim \ar[l]_-\sim \ar[luu]_-\sim
	} \]
\end{proof}

\begin{cor}\label{spaces_are_a_bifibration}
	The homotopy category of all retractive spaces $\ho\mc R$ is a symmetric monoidal bifibration over $\cat{Top}$ with Beck-Chevalley on the homotopy pullback squares.
\end{cor}

\begin{proof}
	We verify the seven conditions (SM1), (SM2), (BF3), (BF4), (BF5), (SMBF6), and (SMBF7) using all of the earlier results concerning $\barsmash$, $f_!$, $f^*$, $f$-cofibrations, $h$-fibrations, and weak equivalences. This gives an SMBF with Beck-Chevalley for strict pullbacks along fibrations. By \autoref{expand_beck_chevalley}, and the fact that every homotopy pullback square is weakly equivalent to a strict pullback along a fibration, we therefore get Beck-Chevalley for all homotopy pullback squares.
\end{proof}

\subsection{Pulling back an SMBF}

In this short section we recall from \cite[12.8]{shulman_framed_monoidal} a method for pulling back symmetric monoidal structures along a map of base categories $F\colon \bT \to \bS$, and give conditions under which the pullback of an SMBF is again an SMBF.

Suppose that $\Phi\colon \sC \to \bS$ is any symmetric monoidal bifibration (SMBF), with product $\boxtimes$. Let $F\colon \bT \to \bS$ be any functor from another category $\bT$ that has all finite products.

\begin{df}
The pullback category $F^*\sC$ has as objects the pairs $(X,a)$ with $X \in \ob\sC$, $a \in \ob\bT$, and $F(a) = \Phi(X)$. Similarly, the morphisms are morphisms in $\sC$ and $\bT$ lying over the same morphism in $\bS$.
\end{df}

Given any two pairs $(X,a)$ and $(X',b)$ in the pullback category $F^*\sC$, we can take the product $X \boxtimes X'$ in $\sC$, and then choose a cartesian arrow in $\sC$ of the form
\[ \xymatrix @R=1.7em{ X \otimes X' \ar@{~>}[d] \ar[r] & X \boxtimes X' \ar@{~>}[d] \\
	F(a \times b) \ar[r] & F(a) \times F(b).
} \]
This defines a new object $X \otimes X' \in \sC$. Similarly, we can define an object $I_\otimes$ using a cartesian arrow
\[ \xymatrix @R=1.7em{ I_\otimes \ar@{~>}[d] \ar[r] & I_\boxtimes \ar@{~>}[d] \\
	F({*}_\bT) \ar[r] & {*}_\bS
} \]
where $*$ denotes a terminal object, or empty product. These constructions define a tensor product and unit on the pullback category $F^*\sC$, and it is straightforward to use the universal properties of cartesian arrows to give associator, unitor, and symmetry isomorphisms making $F^*\sC$ into a symmetric monoidal category.

We give conditions under which this pulled-back symmetric monoidal structure makes $F^*\sC$ into an SMBF, for a chosen class of Beck-Chevalley squares in $\bT$. Suppose that
\begin{itemize}
	\item $F$ preserves Beck-Chevalley squares, and
	\item for any two maps $A \to A'$, $B \to B'$ in $\bT$ the square
	\[ \xymatrix @R=1.7em{
		F(A \times B) \ar[d] \ar[r] & F(A) \times F(B) \ar[d] \\
		F(A' \times B') \ar[r] & F(A') \times F(B')
	} \]
	is Beck-Chevalley in $\bS$.
\end{itemize}

\begin{prop}\label{pullback_SMBF}
	Under these assumptions, for any SMBF $\sC$ over $\bS$, the pullback $F^*\sC$ is an SMBF over $\bT$.
\end{prop}

\begin{proof}	  
	We encounter no difficulty in checking that $F^*\sC$ is a bifibration with (co)cartesian arrows those arrows that are (co)cartesian in $\sC$, and that $F^*\sC$ has Beck-Chevalley for the chosen class of squares in $\bT$. It remains to check that the tensor product $\otimes$ preserves (co)cartesian arrows in $F^*\sC$.
	
	For the cocartesian arrows, it suffices to check that in the square of objects in $\sC$ depicted on the left, lying over the square in $\bS$ on the right, the left vertical arrow is cocartesian.
	\[ \xymatrix{
		X \otimes Y \ar@{-->}[d] \ar[r] & X \boxtimes Y \ar[d] \\
		\phi_!X \otimes \gamma_!Y \ar[r] & \phi_!X \boxtimes \gamma_!Y
	}
	\qquad
	\xymatrix{
		F(A \times B) \ar[d]_-g \ar[r]^-f & F(A) \times F(B) \ar[d]^-k \\
		F(A' \times B') \ar[r]_-h & F(A') \times F(B').
	} \]
	To do this, replace the left-hand square by the isomorphic square
	\[ \xymatrix @R=1.7em{
		f^*(X \boxtimes Y) \ar@{-->}[d] \ar[r] & X \boxtimes Y \ar[d] \\
		h^*k_!(X \boxtimes Y) \ar[r] & k_!(X \boxtimes Y).
	} \]
	and show that the induced map $g_!f^*(X \boxtimes Y) \to h^*k_!(X \boxtimes Y)$ is an isomorphism. As one might expect, this turns out to be precisely the Beck-Chevalley map. It is defined by the universal property of cartesian arrows, which re-arranges into the statement that it is the unique map making this square commute, where $Z = X \boxtimes Y$.
	\[ \xymatrix @R=1.7em{
		f^*Z \ar[d] \ar[r] & Z \ar[r] & k_!Z \\
		g_!f^*Z \ar@{-->}[rr] && h^*k_!Z \ar[u]
	} \]
	We prove that the Beck-Chevalley map has this defining property by subdividing this square as follows.
	\[ \xymatrix @R=1.7em{
		f^*Z \ar[r]^-{cart} \ar[rd]^-{cocart} \ar[d]^-{cocart} & Z \ar[r]^-{cocart} & k_!Z & \\
		g_!f^*Z \ar[dd]^-{coev} \ar[rd]^-{cocart} & f_!f^*Z \ar[u]^-{ev} \ar[rd]^-{cocart} & & h^*k_!Z \ar[ul]^-{cart} \\
		& h_!g_!f^*Z \ar[r]^-\cong & k_!f_!f^*Z \ar[uu]^-{ev} & \\
		h^*h_!g_!f^*Z \ar[ur]^-{cart} \ar[rrr]^-\cong &&& h^*k_!f_!f^*Z \ar[ul]^-{cart} \ar[uu]^-{ev}
	} \]
	This finishes the proof that the tensor preserves cocartesian arrows. The proof for cartesian arrows is much shorter.
\end{proof}

\section{Parametrized spectra}\label{sec:spectra}

In this section we recall the definition of parametrized spectra and the smash product, but we don't define the stable equivalences between them, only the level equivalences. We prove several technical preliminaries before moving on to the stable equivalences in the next section.

We also define the convenient notions of ``freely $f$-cofibrant'' and ``level $h$-fibrant,'' and prove that they interact well with the level equivalences. The most involved and technical part of this is \autoref{prop:spectra_pushout_product}, which requires all of \S\ref{sec:reedy} for its proof.

\beforesubsection
\subsection{Basic definitions and examples}\aftersubsection

A parametrized spectrum is nothing more than a diagram spectrum in the sense of \cite{mmss}, except that the based spaces in the diagram are replaced by retractive spaces in $\mc R(B)$, and we use the external smash product when defining the bonding maps.

Define $S^n$ to be the one-point compactification of $\R^n$, regarded as a retractive space over $*$. For $X \in \mc R(B)$ we regard $\Sigma_B X := S^1 \barsmash X$ as a space over $B$, not $* \times B$.

\begin{df}\hfill
	
	\vspace{-1em}
	\begin{itemize}
		\item A \textbf{(parametrized) sequential spectrum} or \textbf{prespectrum} over $B$ is a sequence of retractive spaces
		\[ X_n \in \mc R(B), \qquad n \geq 0, \]
		together with bonding maps $\sigma\colon \Sigma_B X_n \to X_{1+n}$.\index{sequential spectra $\Psp(B)$}
		\item A \textbf{(parametrized) orthogonal spectrum} over $B$ is a sequential spectrum with a continuous action of the orthogonal group $O(n)$ on $X_n$ through maps of retractive spaces, such that the composite
		\[ \sigma^p\colon S^p \barsmash X_q \to S^{p-1} \barsmash X_{1+q} \to \ldots \to S^{1} \barsmash X_{(p-1)+q} \to X_{p + q} \]
		is $O(p) \times O(q)$-equivariant. A symmetric spectrum is defined similarly but with the symmetric group $\Sigma_n$ instead of $O(n)$.\index{orthogonal spectra!category $\Osp(B)$}
	\end{itemize}
\end{df}

\begin{rmk}
	We use the term \textbf{spectra} any time our discussion applies equally well to sequential spectra and to orthogonal spectra. We avoid symmetric spectra in this document because they require an additional discussion of semistability, see \cite{schwede_symmetric_spectra,braunack2018rational,hebestreit_sagave_schlichtkrull}.
\end{rmk}

Because external smash commutes with pullback, for any parametrized spectrum $X$ over $B$ and $b \in B$, the fibers $X_b$ form a (non-parametrized) spectrum, which we call the \textbf{fiber spectrum}\index{fiber spectrum} at $b$.

The most basic parametrized spectra are \textbf{fiberwise suspension spectra}\index{fiberwise!suspension spectrum} \index{fiberwise!suspension spectrum} $\Sigma^\infty_B X$ for $X$ a retractive space over $B$. We define these by taking external smash product of $X$ with the $n$-sphere $S^n$ over $*$ for every value of $n$. If we instead have a space of the form $X_{+B} = X \amalg B$ we denote this suspension spectrum $\Sigma^\infty_{+B} X$.

If $X$ is a parametrized spectrum, the bonding maps have adjoints
\[ X_n \to \Omega_B X_{1+n} := \barmap_*(S^1,X_{1+n}). \]
Furthermore the loop space functor will be derived if $X_{1+n} \to B$ is a fibration. We say $X$ is a \textbf{(weak) $\Omega$-spectrum}\index{$\Omega$-spectrum} if the composite
\[ X_n \to \Omega_B X_{1+n} \to \R\Omega_B X_{1+n} \]
is a weak equivalence.

Recall from \cite{mmss} that the topological category $\mathscr J$\index{orthogonal spectra!indexing category $\mathscr J$} has objects the finite-dimensional subspaces $V$ of some fixed inner product space isomorphic to $\R^\infty$. For each pair $V$ and $W$, let $O(V,W)$ be the space of all linear isometric embeddings of $V$ into $W$. This has a vector bundle whose fiber consists of the orthogonal complement of $V$ in $W$, and we let $\mathscr J(V,W)$ be the Thom space of this vector bundle. More concretely, we have the formula
\[ \mathscr J(\R^n,\R^{m+n}) \cong O(m+n)_+ \sma_{O(m)} S^m. \]
The composition in $\mathscr J$ composes linear embeddings and adds points in their orthogonal complements.

\begin{df}
	A parametrized $\mathscr J$-space is an enriched $\mathscr J$-diagram in $\mc R(B)$. Concretely, this means for each object $V$ an object $X(V) \in \mc R(B)$, and for each pair $V,W$ a map in $\mc R(B)$
	\[ \mathscr J(V,W) \barsmash X(V) \to X(W) \]
	such that all of the following ``associativity'' and ``unit'' diagrams commute:
	\[ \xymatrix @R=1.7em{
		\mathscr J(V,W) \barsmash \mathscr J(U,V) \barsmash X(U) \ar[r] \ar[d] & \mathscr J(V,W) \barsmash X(V) \ar[d] \\
		\mathscr J(U,W) \barsmash X(U) \ar[r] & X(W)
	}
	\qquad
	\xymatrix @R=1.7em{
		X(V) \ar[r] \ar@{=}[dr] & \mathscr J(V,V) \barsmash X(V) \ar[d] \\
		& X(V).
	}
	\]
\end{df}

\begin{lem}\label{change_universe_equivalence}
	Restricting to $V = \R^n$ gives an equivalence of categories from parametrized $\mathscr J$-spaces to parametrized orthogonal spectra. 
\end{lem}

\begin{proof} Identical to the proof in the non-parametrized case \cite{mmss}. \end{proof}

Let $\Osp(B)$ denote the category of parametrized $\mathscr J$-spaces. In light of the previous lemma, it is harmless to call this the \textbf{category of orthogonal spectra over $B$}\index{orthogonal spectra!category $\Osp(B)$}. Similarly let $\Psp(B)$ denote the \textbf{category of sequential spectra (prespectra) over $B$}.\index{sequential spectra $\Psp(B)$}

A colimit of a diagram of spectra is computed by taking the colimit in $\mc R(B)$ on each spectrum level separately, and then using the canonical commutation of $\Sigma_B(-)$ with colimits to define the bonding maps. The same applies to limits and $\Omega_B(-)$.\index{colimit}\index{limit}

Since spectra are diagrams, we can construct free spectra as free diagrams. Given $V$ and a retractive space $A$, define the \textbf{free orthogonal spectrum}\index{free spectrum $F_V A$} $F_V A$ to be the parametrized $\mathscr J$-space given by the formula
\[ (F_V A)(W) := \mathscr J(V,W) \barsmash A. \]
This is the left adjoint of the functor that sends a $\mathscr J$-space $X$ to $X(V)$. When $V = \R^n$ we call this orthogonal spectrum $F_n A$.
The analogous construction in sequential spectra gives the \textbf{free sequential spectrum} $F_n A$, which at spectrum level 0 to $(n-1)$ is the zero object of $\mc R(B)$, and at level $(n+k)$ is the $k$-fold fiberwise suspension $\Sigma_B^k A = S^k \barsmash A$.

In both cases, $F_0 = \Sigma^\infty$ is the suspension spectrum functor
\[ (\Sigma^\infty A)(W) = \Sigma^W_B A = S^W \barsmash A. \]

\begin{ex}\label{ex:thom_spectra}
	We may represent a virtual bundle $\xi$ over $B$ by a vector bundle $V$ and integer $n$, and define the \textbf{fiberwise Thom spectrum} of $\xi$ to be
	\[ \Th_B(\xi) := F_n \Th_B(V). \]\index{fiberwise!Thom spectrum $\Th_B(\xi)$}
	This depends on the choice of $V$ and $n$, an also on whether we take $F_n$ in orthogonal spectra or sequential spectra. But \autoref{ex:equivalent_thom_spectra} and \autoref{prop:pre_equiv_to_orth} below imply that different choices give stably equivalent results.
\end{ex}

\begin{df}\label{all_spectra}
	Let $\Osp$\index{$\Osp$} denote the \textbf{category of orthogonal spectra} over all base spaces. The objects are pairs $(B,X)$ where $B$ is a (CGWH) space and $X \in \Osp(B)$. The maps are defined as in \eqref{eq:adjoint} at each spectrum level, commuting with the action of $\mathscr J$.\index{orthogonal spectra!category $\Osp$} This is a bifibration, see \autoref{base_change_spectra}.
\end{df}

\beforesubsection
\subsection{Level equivalences, level fibrations, free cofibrations}\label{sec:convenient_definitions}\aftersubsection

\begin{df}\label{level_fibration}
Let $X$ and $Y$ be spectra over $B$.
\begin{itemize}
	\item A \textbf{level equivalence}\index{level equivalence} is a map of spectra $X \to Y$ inducing a weak equivalence on each level $X(V) \to Y(V)$.
	\item A \textbf{level $h$-fibration}\index{$h$-fibration!level} is a map which is a $h$-fibration at every level, and similarly for $q$-fibrations\index{$q$-fibration!level}.
	\item A \textbf{level $f$-cofibration}\index{$f$-cofibration!level}\index{$h$-cofibration!level}\index{$i$-cofibration!level} is an $f$-cofibration at every level, and similarly for other kinds of cofibrations.
\end{itemize}
Since $X(V) \cong X(\R^n)$ when $n = \dim V$, it is enough to check these conditions for $V = \R^n$.
\end{df}

We will need a slightly stronger kind of cofibration.

\begin{df}\label{freely_f_cofibrant}
	Consider all maps of spectra over $B$ of the form $F_n(K \to X)$, where $n\geq 0$ and $K \to X$ is an $f$-cofibration of $f$-cofibrant retractive spaces over $B$. The class of \textbf{free $f$-cofibrations}\index{$f$-cofibration!free} is the smallest class of maps of spectra containing the above, and closed under pushouts, transfinite compositions, and retracts. In particular, any countable composition of pushouts of coproducts of maps of the form $F_n(K \to X)$ is a free cofibration.
\end{df}

The \textbf{free $q$-cofibrations}, \textbf{free $h$-cofibrations}, and \textbf{free $i$-cofibrations} are defined similarly\index{$q$-cofibration!free}\index{$h$-cofibration!free}\index{$i$-cofibration!free}. For $q$-cofibrations, this class is generated by free cells $F_k(S^{n-1}_{+B}) \to F_k(D^n_{+B})$. When a spectrum is cofibrant in this sense, we say it is \textbf{freely cofibrant}.\footnote{The more obvious term ``free cofibrant'' turns out to be confusing because spectra that are cofibrant in this sense do not have to be free, only built out of free spectra in a certain way.}

\begin{lem}\label{prop:free_implies_level}
	Every free cofibration is a level cofibration.
\end{lem}

\begin{proof}
	It suffices to prove that $F_n(K \to X)$ is a level cofibration. At level $k$ it is the identity of $\mathscr J(\R^n,\R^k)$ or $\mathscr N(\R^n,\R^k)$ smashed with $K \to X$, which is a cofibration by \autoref{prop:h_cofibrations_pushout_product} (for $f$, $h$, or $i$-cofibrations) or \autoref{ex:smashing_spaces_is_left_Quillen} (for $q$-cofibrations).
\end{proof}

\begin{lem}
	The free spectrum functor $F_n$ sends cofibrant spaces to freely cofibrant spectra. It sends $h$-cofibrant, $h$-fibrant spaces to level $h$-fibrant spectra.
\end{lem}

\begin{proof}
	Also follows directly from \autoref{prop:h_cofibrations_pushout_product}.
\end{proof}

\begin{prop}[Level model structure]\label{prop:level_model_structure}
	The level equivalences, free $q$-cofibrations, and level $q$-fibrations define a proper model structure on spectra (sequential spectra or orthogonal spectra). It is cofibrantly generated by
	\[ \begin{array}{ccrll}
	I &=& \{ \ F_k\left[ S^{n-1}_{+B} \to D^n_{+B} \right] & : n,k \geq 0, & D^n \to B \ \} \\
	J &=& \{ \ F_k\left[ D^n_{+B} \to (D^n \times I)_{+B} \right] &: n,k \geq 0, & (D^n \times I) \to B \ \}.
	\end{array} \]
\end{prop}

\begin{proof}
	A standard argument verifies the conditions laid out in \cite[2.1.19]{hovey_model_cats}, see also \cite[12.1.7]{ms} or the proof of \autoref{thm:stable_model_structure} below. By \autoref{prop:free_implies_level} for $h$-cofibrations, the cofibrations in this model structure are always level $h$-cofibrations, therefore it is left proper. Right properness is also immediate.
\end{proof}

\begin{rmk}
	Everything in this subsection works equally well in (CGWH) and (CG), including \autoref{prop:level_model_structure}. We therefore have two distinct model categories of sequential spectra and two distinct model categories of orthogonal spectra. In each case, the forgetful functor (CGWH) $\to$ (CG) is a right Quillen equivalence.
\end{rmk}

The above model structure gives us a cofibrant replacement functor $Q$\index{cofibrant replacement}, and a fibrant replacement functor $R^{lv}$\index{fibrant replacement!$R^{lv}$ for spectra}. In principle these are only defined on each fiber category $\Osp(B)$ or $\Psp(B)$ separately, but an examination of the small-object argument reveals that they actually define functors on the category of spectra over all bases $\Osp$ or $\Psp$ (\autoref{all_spectra}), that preserve each fiber category.

We can also generalize the monoidal fibrant replacement functor $P$ to spectra. If $X$ is a spectrum, we define $PX$ by applying $P$ to every level $X_n$\index{fibrant replacement!$P$ for spectra} and using the canonical homeomorphisms to define the bonding maps
\[ \mathscr J(V,W) \barsmash PX(V) \cong P(\mathscr J(V,W) \barsmash X(V)) \to PX(W) \]
or more simply in sequential spectra
\[ S^n \barsmash PX_m \cong P(S^n \barsmash X_m) \to PX_{n+m}. \]
Again, this defines a functor on each fiber category $\Osp(B)$ or $\Psp(B)$, or on the entire category $\Osp$ or $\Psp$ of spectra over all base spaces.

\begin{prop}\label{prop:spectrum_px} \hfill
	\vspace{-1em}
	
	\begin{itemize}
		\item	If $X$ is level $h$-cofibrant, then $PX$ is level $h$-fibrant and $X \to PX$ is a level equivalence.
		\item $P$ preserves free/level $f$-cofibrations and free/level $h$-cofibrations.
		\item If the spectra are all level $h$-cofibrant, $P$ turns free/level $h$-cofibrations into free/level $f$-cofibrations.
	\end{itemize}
\end{prop}

\begin{proof}
	\autoref{prop:px_properties} implies everything except for the statements about free cofibrations. For that part, we observe that $P$ preserves retracts because it is a functor, and preserves pushouts and transfinite compositions when they are taken along free cofibrations, because free cofibrations are level cofibrations and therefore level closed inclusions, and the pullback and pushforward that together form $P$ preserve those colimits by \autoref{lem:f_shriek_preserves} and \autoref{prop:f_star_colimits}. This reduces our work to checking that $P$ sends $F_n(K \to X)$ to a free $f$-cofibration if $K \to X$ is an $h$-cofibration between $h$-cofibrant retractive spaces. Since $P$ is monoidal we get an identification
	\[ P(F_n(K \to X)) \cong F_n(PK \to PX), \]
	and $PK \to PX$ is an $f$-cofibration of $f$-cofibrant spaces by \autoref{prop:px_properties} again.
\end{proof}

In summary, any spectrum $X$ can be replaced by a freely $f$-cofibrant level $h$-fibrant spectrum $PQX$, and there is a zig-zag of level equivalences $X \overset\sim\leftarrow QX \overset\sim\to PQX$.

\beforesubsection
\subsection{Base change functors $f_!$, $f^*$, and $f_*$}\label{base_change_spectra}\aftersubsection

We only define these for orthogonal spectra $\Osp(B)$, but the constructions and proofs work equally well for sequential spectra $\Psp(B)$.

Given $f\colon A \to B$, we define the functor $f^*\colon \Osp(B) \to \Osp(A)$ by pulling back every level of a spectrum $X$ over $B$ back to $A$, and defining the bonding maps using the canonical commutation of pullback and smash product from \autoref{lem:external_smash_and_base_change}:
\[ \mathscr J(V,W) \barsmash f^*X(V) \cong f^*(\mathscr J(V,W) \barsmash X(V)) \to f^*X(W) \]
The pushforward is defined in exactly the same way.\index{pushforward $f_{^^21}$} The adjunction $(f_! \adj f^*)$ on each level commutes with the action of $\mathscr J(V,W)$ and therefore defines an adjunction $(f_! \adj f^*)$ as functors between $\Osp(A)$ and $\Osp(B)$.\index{pullback $f^*$}

Alternatively, if we let $\Osp$\index{$\Osp$} be the category of spectra over all base spaces, from \autoref{all_spectra}, we observe that the forgetful functor $\Osp \to \cat{Top}$ is a bifibration. We then define $f^*$ and $f_!$ by taking the the cartesian, respectively cocartesian arrows in $\Osp$, as in \autoref{prop:spaces_SMBF}.

\autoref{lem:f_star_preserves} and \autoref{lem:f_shriek_preserves} give sufficient conditions under which $f^*$ and $f_!$ preserve level equivalences, level fibrations, and level cofibrations. Since $f_!(F_n K) \cong F_n f_!K$ and $f^*(F_n K) \cong F_n f^*K$ by \autoref{lem:external_smash_and_base_change} again, we also get that $f_!$ preserves all free cofibrations, and $f^*$ preserves the free $f$-cofibrations.\footnote{See the proof of \autoref{prop:spectrum_px} for how to show that $f^*$ preserves the compositions and pushouts that make up a general free cofibration.}  Collecting this together:
\begin{lem}\label{lem:spectrum_f_preserves}
	$f^*\colon \Osp(B) \to \Osp(A)$ preserves
	\begin{itemize}
		\item level $h$-fibrations and $q$-fibrations,
		\item level equivalences between level $q$-fibrant (or $h$-fibrant) spectra,
		\item free or level $f$-cofibrations, and free or level closed inclusions.
	\end{itemize}
	$f_!\colon \Osp(A) \to \Osp(B)$ preserves
	\begin{itemize}
		\item free or level $f$-cofibrations, $q$-cofibrations, $h$-cofibrations, or closed inclusions, and
		\item level equivalences between level $h$-cofibrant spectra.
	\end{itemize}
	Moreover if $f$ itself is a Hurewicz fibration ($h$-fibration), then $f^*$ preserves free or level $h$-cofibrations, and all level equivalences, while $f_!$ preserves spectra that are both level $h$-fibrant and level $h$-cofibrant.
\end{lem}

In particular, $(f_! \adj f^*)$ is a Quillen pair for the level model structure, which is a Quillen equivalence when $f$ is a weak equivalence.

Finally, when $f$ is a fiber bundle with fiber a cell complex, we construct $f_*$ in spectra by applying the construction from \autoref{prop:sheafy_pushforward} on each level, and using the canonical commutation with $\Omega_B(-)$. We get a Quillen adjunction $(f^* \adj f_*)$ in the level model structure, which is a Quillen equivalence when $f$ is a weak equivalence.

\beforesubsection
\subsection{External smash products and mapping spaces}\label{sec:smash_with_space}\aftersubsection

For each retractive space $K \in \mc R(A)$ and spectrum $X \in \Psp(B)$ or $\Osp(B)$, we define the spectrum $K \barsmash X$ to have $n$th level $K \barsmash X_n$. We use the symmetric monoidal structure of $\barsmash$ to define the bonding maps.\index{external smash product $X \barsmash Y$!of a spectrum and a space} This gives functors
\[ \barsmash\colon \mc R(A) \times \Psp(B) \to \Psp(A \times B), \qquad \barsmash\colon \mc R(A) \times \Osp(B) \to \Osp(A \times B) \]
that are unique up to unique isomorphism.

\begin{ex}
	Every free spectrum $F_n K$ over $B$ is canonically isomorphic to the smash product of the space $K \in \mc R(B)$ with the free spectrum $F_n S^0 \in \Psp(*)$ or $\Osp(*)$.
\end{ex}

\begin{ex}
	We define the reduced suspension functor $\Sigma_B$ on spectra over $B$ to be smash product with the space $S^1 \in \mc R(*)$.\index{fiberwise!reduced suspension $\Sigma_B$}
\end{ex}

Smashing with a space has a right adjoint in each variable. For formal reasons, these right adjoints are also be uniquely defined up to unique isomorphism, and the choice of adjunction is also unique.

If we fix the space $K \in \mc R(A)$, the right adjoint of $K \barsmash -$ is the \textbf{external mapping spectrum}\index{external!function spectrum $\barF_B(Y,Z)$} functor
\[ \barF_A(K,-)\colon \Osp(A\times B) \to \Osp(B) \quad \textup{or} \quad \barF_A(K,-)\colon \Psp(A\times B) \to \Psp(B) \]
that applies $\barmap_A(K,-)$ at each spectrum level. (We use the notation $\barF$ because the output is a spectrum, and reserve $\barmap$ for when the output is a space.) The adjunctions from \autoref{sec:external_smash} pass in a unique way to an adjunction 
\[ K \barsmash X \to Y \quad \textup{ over }A \times B \quad \longleftrightarrow \quad X \to \barF_A(K,Y) \quad \textup{ over }B. \]

\begin{ex}
	We define the based loops functor $\Omega_B$ on spectra over $B$ by taking external maps out of the space $S^1 \in \mc R(*)$.\index{fiberwise!based loops $\Omega_B$}
\end{ex}

If we fix the spectrum $X \in \Osp(B)$, the right adjoint of $- \barsmash X$ sends the spectrum $Y$ over $A \times B$ to the \textbf{external mapping space}\index{external!mapping space $\barmap_B(Y,Z)$} $\barmap_B(X,Y)$. This is defined from the space-level mapping space as the equalizer in $\mc R(A)$
\[ \barmap_B(X,Y) \to \prod_n \barmap_B(X_n,Y_n) \rightrightarrows \prod_{m,n} \barmap_B(X_m \barsmash \mathscr J(\R^m,\R^n),Y_n) \]
where the two parallel maps either
\begin{itemize}
	\item pre-compose the map $X_n \to Y_n$ with $X_m \barsmash \mathscr J(\R^m,\R^n) \to X_n$, or
	\item post-compose $X_m \to Y_m$ with $Y_m \barsmash \mathscr J(\R^m,\R^n) \to Y_n$.
\end{itemize}
In other words, for each $a \in A$ a collection of maps $X_n \to (Y_a)_n$ in $\mc R(B)$ lands in the equalizer subspace iff they give a map of orthogonal spectra $X \to Y_a$ over $B$.

In sequential spectra, we similarly get the equalizer
\[ \barmap_B(X,Y) \to \prod_n \barmap_B(X_n,Y_n) \rightrightarrows \prod_{m,n} \barmap_B(X_m \barsmash S^{n-m},Y_n) \]
and in both cases the above definition helps us define an adjunction
\[ K \barsmash X \to Y \quad \textup{ over }A \times B \quad \longleftrightarrow \quad K \to \barmap_B(X,Y) \quad \textup{ over }A. \]

\begin{rmk}
	Setting $A = *$, this gives an enrichment of the categories $\Osp(B)$ and $\Psp(B)$ in based spaces. In other words it endows each set of maps of spectra over $B$, $\Osp(B)(X,Y)$ or $\Psp(B)(X,Y)$, with a topology.
\end{rmk}


Next we describe the smash product of two orthogonal spectra $X \in \Osp(A)$ and $Y \in \Osp(B)$. For each pair of objects $V$ and $V'$ in $\mathscr J$, we take the external smash product space
\[ X(V) \barsmash Y(V') \in \mc R(A \times B). \]
These form a diagram over the category $\mathscr J \sma \mathscr J$ whose morphisms are smash products of the morphism spaces
\[ (\mathscr J \sma \mathscr J)((V,V'),(W,W')) = \mathscr J(V,W) \sma \mathscr J(V',W'). \]
In other words, they form a \textbf{bispectrum} over $A \times B$. The direct sum of representations, plus some choice of embedding $\mc U \oplus \mc U \to \mc U$, gives a direct sum functor 
\[ \oplus\colon \mathscr J \sma \mathscr J \to \mathscr J. \]
\begin{df}\label{spectra_external_smash}\cite[11.4.10]{ms}
The \textbf{external smash product spectrum}
\[ X \barsmash Y \in \Osp(A \times B) \]
\index{external smash product $X \barsmash Y$!of parametrized spectra} is defined to be the left Kan extension of the bispectrum $X(-) \barsmash Y(-)$ along $\oplus$.\footnote{Unfortunately, the term ``external smash product'' is also used in \cite{mmss} to differentiate the bispectrum $\{X(V) \sma Y(V')\}$ from the spectrum $X \sma Y$. These are two different uses of the word ``external.'' The product $X \barsmash Y$ we define in this section is internal over $\mathscr J$ but external over the base spaces. In other words it produces a spectrum (not a bispectrum) but the spectrum is over $A \times B$ (not just $A$). See also \cite[11.1.7]{ms}.}
\end{df}

The following lemma has the same proof as the non-parametrized case.
\begin{lem}\label{barsmash_free}
	There is a natural isomorphism for retractive spaces $X$ and $Y$
	\[ F_V X \barsmash F_W Y \cong F_{V \oplus W} (X \barsmash Y), \]
	and a natural isomorphism for retractive spaces $X$ and parametrized spectra $Z$
	\[ (\Sigma^\infty X) \barsmash Z \cong X \barsmash Z. \]
\end{lem}
We will soon see that these isomorphisms are not just canonical, but unique.

The external smash product of spectra also has a right adjoint in each variable. The proof of this follows the same formal recipe found in \cite[\S 21]{mmss}. To describe the resulting right adjoint let $Y \in \Osp(B)$ and $Z \in \Osp(A \times B)$. Recall that the \textbf{shift functor} $\sh^n$ simply re-indexes the levels of $Z$,
\[ (\sh^n Z)_m = Z_{m+n}. \]
Pulling $Z$ back along the direct sum map $\oplus$ gives a bispectrum whose $n$th level is the orthogonal spectrum $\sh^n Z$. Since this is an orthogonal spectrum of orthogonal spectra, we can apply $\barmap_B(Y,-)$ and get just an orthogonal spectrum, which we call $\barF_B(Y,Z)$:\index{external!function spectrum $\barF_B(Y,Z)$}
\[ \barF_B(Y,Z)_n = \barmap_B(Y,\sh^n Z). \]
The formal argument cited above gives an adjunction
\[ X \barsmash Y \to Z \quad \textup{ over }A \times B \quad \longleftrightarrow \quad X \to \barF_B(Y,Z) \quad \textup{ over }A. \]
As an important special case, the right adjoint of $- \barsmash F_n S^0$ is the shift functor $\sh^n$.

We finish this subsection by describing how external smashes of spectra interact with level equivalences, level fibrations, and free cofibrations.
\begin{thm}\label{prop:spectra_pushout_product}\hfill
	\vspace{-1em}
	
	\begin{enumerate}
		\item Let $f: K \to X$ and $g: L \to Y$ be free $h$-cofibrations of spectra over $A$ and $B$, respectively. Then $f \square g$, constructed using $\barsmash$, is a free $h$-cofibration. The same is true for free $f$-cofibrations, free $q$-cofibrations, and free closed inclusions.
		\item If $X$ and $Y$ are freely $h$-cofibrant and level $h$-fibrant then $X \barsmash Y$ is level $h$-fibrant.
		\item If $X$ is freely $h$-cofibrant and $g: Y \to Y'$ is a level equivalence of freely $h$-cofibrant spectra then $\id_X \barsmash g$ is a level equivalence.
	\end{enumerate}
\end{thm}

\begin{proof}
	At the moment we can only prove the first claim. Start by assuming that $f$ and $g$ are free spectra on $h$-cofibrations. Then by \autoref{barsmash_free}, their pushout-product is a free spectrum on a pushout-product of $h$-cofibrations, which is an $h$-cofibration by \autoref{prop:h_cofibrations_pushout_product}, so $f \square g$ is a free $h$-cofibration.
	
	Now examine the class of pairs of maps $f,g$ for which $f \square g$ is a free $h$-cofibration. It contains the free spectra on the $h$-cofibrations. By \autoref{pushout_product_lemmas}, it is closed under pushouts. By \autoref{pushout_product_lemmas}, it is closed under transfinite composition. It is clearly closed under retracts, just because $\square$ is a bifunctor. Therefore it contains all pairs of free $h$-cofibrations.
\end{proof}

The final two claims of \autoref{prop:spectra_pushout_product} are a trivial corollary of \autoref{prop:reedy_pushout_product} below, which we prove after introducing another kind of cofibration that is ``semifree'' instead of free. The payoff for this is the ability to commute $\barsmash$ and $f^*$ on the homotopy category. Note that a slightly weaker form of these two claims appears in \cite[Prop 4.27 and 4.31]{malkiewich2017coassembly}. Also, a different proof of (3) appears in \cite[Lem 9.7]{mp1}, and a more ad-hoc argument that $\barsmash$ and $f^*$ commute in the homotopy category can be found in \cite[Thm 9.9]{mp1}.

\beforesubsection
\subsection{The SMBF $\Osp$}\aftersubsection

Recall from \autoref{all_spectra} that $\Osp$ is the category of all orthogonal spectra over all base spaces. We have defined the base-change functors and the external smash product already, so to make this a symmetric monoidal bifibration (SMBF), it remains to show that the external smash product commutes with pullbacks and pushforwards.

The universal property of pushforward and pullback give canonical maps
\[ \xymatrix @R=0.3em{
	(f \times g)_!(X \barsmash Y) \ar[r] & f_!X \barsmash g_!Y \\
	f^*W \barsmash g^*Z \ar[r] & (f \times g)^*(W \barsmash Z).
} \]
\begin{prop}\label{prop:spectra_external_smash_and_base_change}(cf. \cite[11.4.1]{ms}, \cite{shulman_framed_monoidal})\hfill
	\vspace{-1em}
	
	\begin{itemize}
		\item (CG) These maps are isomorphisms.
		\item (CGWH) The first map is an isomorphism. The second is an isomorphism when $X'$ and $Y'$ are freely $i$-cofibrant.\footnote{It is apparently enough if only one of them is semifreely $i$-cofibrant in the sense of the next section, and the other is arbitrary, but the proof is longer.}
	\end{itemize}
\end{prop}

\begin{proof}
	For the first map, this reduces to the fact that $f_!$ commutes with left Kan extensions and the space-level smash products. For the second map, it is an isomorphism when $X' = F_n K$ and $Y' = F_m L$ because $f^*$ commutes with free spectra. In (CG), $f^*$ preserves all colimits so the conclusion follows. In (CGWH), $f^*$ preserves pushouts along level closed inclusions. So if we take a pushout square along a level closed inclusion
	\[ \xymatrix @R=1.7em{
		X_1 \ar[r] \ar[d] & X_2 \ar[d] \\
		X_3 \ar[r] & X_4
	} \]
	and the claim is true for the first three vertices and $Y'$, then it is true for $X_4$ and $Y'$. (The desired map is a pushout of three isomorphisms and is therefore an isomorphism.) By a similar argument, if $X'$ is a colimit of a sequence of closed inclusions $X_i \to X_{i+1}$ and the claim is true for each $X_i$ and $Y'$ then it is true for $X'$ and $Y'$. Therefore the class of spectra for which it is true must include all freely $i$-cofibrant $X'$ and $Y'$.
\end{proof}

The rigidity theorem \autoref{prop:spaces_rigidity} generalizes to spectra, but it's a good idea to make the statement a little stronger to accommodate the fact that $f^*$ and $\barsmash$ don't always commute in (CGWH).

\begin{thm}[Rigidity]\label{thm:spectra_rigidity}
	Suppose $n \geq 0$ and we have maps of spaces
	\[ B \overset{f}\leftarrow A \overset{g}\to C_1 \times \ldots \times C_n \]
	such that $(f,g)\colon A \to B \times C_1 \times \ldots \times C_n$ is injective.
	
	Then any functor $\Osp(C_1) \times \ldots \times \Osp(C_n) \to \Osp(B)$ isomorphic to
	\[ \Phi\colon (X_1,\ldots,X_n) \leadsto f_!g^*(X_1 \barsmash \ldots \barsmash X_n) \]
	is in fact uniquely isomorphic to $\Phi$.
	
	More generally, on any full subcategory of $\Osp(C_1) \times \ldots \times \Osp(C_n)$ containing all $n$-tuples of the form $(F_{V_1} *_{+C_1},\ldots, F_{V_n} *_{+C_n})$, any functor isomorphic to $\Phi$ is uniquely isomorphic to $\Phi$.\index{rigidity!for spectra}
\end{thm}

\begin{proof}
	Let $\eta$ be an automorphism of $\Phi$. Using the adjunction between free spectra and evaluation, the proof of \autoref{prop:spaces_rigidity} shows that $\eta$ is determined by its value on every tuple of free spectra of the form $(F_{V_1} *_{+C_1},\ldots, F_{V_n} *_{+C_n})$. It is therefore enough to show that $\eta$ is the identity on these spectra. By \autoref{prop:spaces_rigidity} again, $\eta$ must be trivial on the suspension spectra of any $n$-tuple of spaces of the form $(*_{+C_1},\ldots,*_{+C_n})$ where the free points map to $(c_1,\ldots,c_n) \in C_1 \times \ldots \times C_n$. We argue the same for the free spectra $(F_{V_1} *_{+C_1},\ldots, F_{V_n} *_{+C_n})$ using the argument from \cite[3.17]{malkiewich_cyclotomic_dx}: the action of $\eta$ over any point of $B$ is a self-map of a spectrum of the form $F_{V_1 + \ldots + V_n} S^0$ which agrees with the identity map of $F_0S^0$ along any choice of point in $S^{V_1 + \ldots + V_n}$. Since the orthogonal group acts faithfully on the sphere, this is enough to conclude that $\eta$ acts as the identity on this tuple of free spectra. This finishes the proof.
\end{proof}

\begin{rmk}
	The same rigidity theorem holds if some of the factors in the smash product are spaces. It also holds in sequential spectra, provided at most one of the factors in the smash product is a spectrum, and the rest are spaces.
\end{rmk}

The same corollaries follow as before, except that in (CGWH) we restrict the domain to freely $i$-cofibrant spectra any time we need to switch $f^*$ with $\barsmash$ to get an isomorphism. For instance:
\begin{itemize}
	\item The interchange transformations of \autoref{prop:spectra_external_smash_and_base_change} are unique on the subcategory of freely $i$-cofibrant spectra. This implies that, as natural transformations, they are unique on the entire category.
	\item The category $\Osp$ of all parametrized spectra over all parametrized spaces is a symmetric monoidal category.
	\item In (CG), $\Osp$ is a symmetric monoidal bifibration. In both (CG) and (CGWH), the full subcategory $\Osp'$ of freely $i$-cofibrant spectra is a symmetric monoidal bifibration.
	\item In (CG), $\Osp(B)$ is a symmetric monoidal category under the internal smash product $\Delta_B^*(- \barsmash -)$. In (CG) and (CGWH), the subcategory $\Osp(B)'$ is a symmetric monoidal category.
\end{itemize}

\begin{rmk}\label{rmk:general_rigidity}
	The proof of \autoref{thm:spectra_rigidity} actually establishes a stronger statement: if we modify $\Phi$ by taking the smash products after the pullbacks, then any natural transformation $\Phi \Rightarrow \Theta$ that is an isomorphism on free spectra, is the unique natural transformation with that property. See \cite[\S 3.3]{malkiewich_cyclotomic_dx} for similar statements to this one.
\end{rmk}

\begin{cor}\hfill
	\vspace{-1em}
	
	\begin{itemize}
		\item (CG) The functor $P\colon \Osp \to \Osp$ has a strong symmetric monoidal structure.
		\item (CGWH) The functor $P\colon \Osp \to \Osp$ has a lax symmetric monoidal structure that is an isomorphism on freely $i$-cofibrant spectra.
	\end{itemize}
\end{cor}

\begin{proof}
	The isomorphism $\Sph \cong P\Sph$ is unique. There is a canonical map $PX \barsmash PY \to P(X \barsmash Y)$ over the identity of $A \times B$, by the above and the fact that $P = (p_1)_!(p_0)^*$, which is an isomorphism (CG) always (CGWH) on freely $i$-cofibrant spectra. As in \autoref{prop:P_strong_monoidal}, by \autoref{rmk:general_rigidity} we can keep track of which natural transformation is the right one by restricting attention to $X = \Sigma^\infty_{+A} A$ and $Y = \Sigma^\infty_{+B} B$, where it gives a map $A^I \times B^I \to (A \times B)^I$. Using this principle it is straightforward to check the associator, unitor, and symmetry coherences.
\end{proof}

\beforesubsection
\subsection{Skeleta and semifree cofibrations}\label{sec:reedy}\aftersubsection

In this subsection we finish the proof of \autoref{prop:spectra_pushout_product}. We cannot proceed using only free cofibrations, so we pass to a more general notion of semifree cofibrations. We then utilize a close relationship between semifree cofibrations and the skeleta of an orthogonal spectrum.

The \textbf{semifree spectrum}\index{semifree spectrum} $\mc G_n K$ on an $O(n)$-equivariant retractive space $K$ is defined to be the parametrized $\mathscr J$-space
\[ \mathscr J(\R^n,-) \barsmash_{O(n)} K. \]
This gives the left adjoint to the operation that takes a parametrized orthogonal spectrum $X$ to $X_n$ as an $O(n)$-equivariant retractive space.

\begin{prop}\label{smash_two_semi_free}
	Let $K$ be a retractive $O(m)$-space over $A$, and $L$ a retractive $O(n)$-space over $B$. Then there is a natural isomorphism
	\[ \mc G_m K \barsmash \mc G_n L \cong \mc G_{m+n} (O(m+n)_+ \barsmash_{O(m) \times O(n)} K \barsmash L). \]
\end{prop}

\begin{proof}
	As in the non-fiberwise case (cf \cite[I.5.14]{schwede_symmetric_spectra}, \cite[2.3.4]{brun_dundas_stolz}), we look at the $\mathscr J \sma \mathscr J$-space defining the smash product on the left-hand side. At bilevel $(m+m',n+n')$ it is
	\[ (\mathscr J(\R^m,\R^{m+m'}) \sma \mathscr J(\R^n,\R^{n+n'})) \barsmash_{O(m) \times O(n)} (K \barsmash L). \]
	Since the left-hand side of the statement in the proposition is defined as a left Kan extension of the diagram just above, it is a left adjoint functor from $O(m) \times O(n)$-spaces over $A \times B$ to spectra over $A \times B$, applied to $K \barsmash L$. To make this agree with the right-hand side, we simply verify that both right adjoints send the spectrum $X$ over $A \times B$ to $X_{m+n}$ considered as a retractive $O(m) \times O(n)$-space.
\end{proof}

Given an orthogonal spectrum $X$ over $B$, its \textbf{$n$-skeleton}\index{skeleton $\sk^n X$} is the orthogonal spectrum over $B$
\[ \sk^n X := (i_n)_!(i_n)^* X, \]
where $i_n$ is the inclusion of the full subcategory of $\mathscr J$ on the objects $\R^0$ through $\R^n$, $(i_n)^*$ is the restriction to this subcategory, and $(i_n)_!$ is the enriched left Kan extension. Writing the definition out gives the presentation
\[ \bigvee_{i,j \leq n} \mathscr J(\R^j,-) \barsmash \mathscr J(\R^i,\R^j) \barsmash X_i \rightrightarrows
\bigvee_{i \leq n} \mathscr J(\R^i,-) \barsmash X_i \to \sk^n X. \]
The $n$-skeleton also has a ``semifree'' presentation
\[ \bigvee_{i,j \leq n} \mathscr J(\R^j,-) \barsmash_{O(j)} \mathscr J(\R^i,\R^j) \barsmash_{O(i)} X_i \rightrightarrows
\bigvee_{i \leq n} \mathscr J(\R^i,-) \barsmash_{O(i)} X_i \to \sk^n X. \]
By a standard categorical argument, $\sk^{n-1} X$ admits a map to $X$ which is an isomorphism levels 0 through $n-1$. At level $n$ this map becomes an $O(n)$-equivariant map $L_n X \to X_n$ over $B$; we refer to it as the $n$th \textbf{latching map}\index{latching map} of $X$. (We define the 0th latching map to be $B \to X_0$.)

\begin{ex}
	Let $X = \mc G_n K$. Then $\sk^m X$ is the zero object for $m < n$ and $\mc G_n K$ for $m \geq n$. By comparing universal properties, it has a single nontrivial latching map, which is at level $n$ and is just the inclusion of the zero object $B \to K$.
\end{ex}

More generally, any map of spectra $X \to Y$ over $B$ has a sequence of relative skeleta, defined as pushouts of the skeleta of $Y$ with $X$ along the skeleta of $X$:
\[ \xymatrix @R=1.7em{
	\sk^n X \ar[r] \ar[d] & \sk^n Y \ar[d] \\
	X \ar[r] & \sk^n (X \to Y) } \]
The map from $\sk^{n-1}(X \to Y)$ to $Y$ on level $n$ can then be written
\[ L_n Y \cup_{L_n X} X_n \to Y_n, \]
and we call this the $n$th relative latching map of $X \to Y$. The following two lemmas are straightforward comparisons of universal properties.

\begin{prop}\label{prop:skeleta_pushout}
	For each spectrum $X$ over $B$, the latching map $L_n X \to X_n$ fits into a natural pushout square
	\[ \xymatrix @R=1.7em{
		\mc G_n L_n X \ar[r] \ar[d] & \mc G_n X_n \ar[d] \\
		\sk^{n-1} X \ar[r] & \sk^n X. } \]
\end{prop}

\begin{prop}\label{prop:map_skeleta_pushout}
	For each map of spectra $X \to Y$ over $B$, the relative latching map $L_n Y \cup_{L_n X} X_n \to Y_n$ fits into a natural pushout square
	\[ \xymatrix @R=1.7em @C=3em{
		\mc G_n (L_n Y \cup_{L_n X} X_n) \ar[r] \ar[d] & \mc G_n Y_n \ar[d] \\
		\sk^{n-1} (X \to Y) \ar[r] & \sk^n (X \to Y). } \]
\end{prop}

\begin{df}
	A map of spectra $X \to Y$ over $B$ is a \textbf{semifree} or \textbf{Reedy $h$-cofibration}\index{$h$-cofibration!semifree} if each relative latching map $L_n Y \cup_{L_n X} X_n \to Y_n$ is an $h$-cofibration. The notions of semifree $f$-cofibration and semifree closed inclusion are defined similarly.\index{$f$-cofibration!semifree}\index{$i$-cofibration!semifree}\footnote{Note that ``flat cofibrations'' of orthogonal spectra (see e.g. \cite{schwede_global}) are the same thing as semifree $q$-cofibrations. To be precise, the latching map has to be a relative $O(n)$-cell complex at spectrum level $n$.}
\end{df}

We emphasize that in this and the previous definition, the latching map is $O(n)$-equivariant, but that action is ignored when determining if the map is a cofibration. This means our proofs below will be more general than the ones done in \cite[\S 4]{malkiewich2017coassembly}.\footnote{However when we develop the equivariant theory in \autoref{sec:G_reedy}, we have to stick to homotopy extension properties that respect the $G \times O(n)$-action.}

\begin{rmk}\label{rmk:reedy_prespectra}
	The notion of a semifree cofibration makes sense for sequential spectra, but it turns out to be the same thing as a free cofibration. The presence of the $O(n)$-action is what makes the two notions different in orthogonal spectra.
\end{rmk}

To simplify our language a bit, in the rest of this section we will use the word ``cofibration'' when making a statement that is equally true for $f$-cofibrations, $h$-cofibrations, or closed inclusions.

\begin{prop}\label{reedy_preserved}
	Semifree cofibrations are closed under pushouts, transfinite compositions (therefore also coproducts), and retracts.
\end{prop}

\begin{proof}
	The proofs are very similar to those in \autoref{pushout_product_lemmas}, but formally a little different.
	
	If $X \to X \cup_K M$ is a pushout of the semifree cofibration $K \to M$, then its latching map
	\[ L_n (X \cup_K M) \cup_{L_n X} X_n \to X_n \cup_{K_n} M_n \]
	rearranges as
	\[ L_n M \cup_{L_n K} X_n \to M_n \cup_{K_n} X_n \]
	This is a pushout of the original latching map $L_n M \cup_{L_n K} K_n \to M_n$ by the diagram in which every square is a pushout
	\[ \xymatrix @R=1.7em @!C{
		L_n K \ar[r] \ar[d] & L_n M \ar[d] \\
		K_n \ar[r] \ar[d] & L_n M \cup_{L_n K} K_n \ar[r] \ar[d] & M_n \ar[d] \\
		X_n \ar[r] & L_n M \cup_{L_n K} X_n \ar[r] & M_n \cup_{K_n} X_n
	} \]
	and is therefore a cofibration. So the pushout $X \to X \cup_K M$ is a semifree cofibration.
	
	For compositions, suppose $X \to Y \to Z$ are both semifree cofibrations. In the following diagram, every square is a pushout:
	\[ \xymatrix @R=1.7em @!C=6em{
		L_n X \ar[r] \ar[d] & X_n \ar[d] \ar[rd] & \\
		L_n Y \ar[r] \ar[d] & L_n Y \cup_{L_n X} X_n \ar[r]^-{c} \ar[d] & Y_n \ar[d] \ar[rd] \\
		L_n Z \ar[r] & L_n Z \cup_{L_n X} X_n \ar[r]^-{c} & L_n Z \cup_{L_n Y} Y_n \ar[r]^-{c} & Z_n
	}\]
	Therefore the maps labeled $c$ are cofibrations; two of them by assumption, and the last because it is a pushout of a cofibration. It follows that the latching map $L_n Z \cup_{L_n X} X_n \to Z_n$ is also a cofibration.
	
	If $X^{(0)} \to X^{(\alpha)}$ is the transfinite composition of the semifree cofibrations $X^{(\beta)} \to X^{(\beta + 1)}$, then the latching map for $X^{(0)} \to X^{(\alpha)}$ is the transfinite composition of the maps
	\[ L_n X^{(\alpha)} \cup_{L_n X^{(\beta)}} X^{(\beta)}_n \to L_n X^{(\alpha)} \cup_{L_n X^{(\beta + 1)}} X^{(\beta + 1)}_n. \]
	However, each of these is a cofibration by the argument we used above for compositions, with $X = X^{(\beta)}$, $Y = X^{(\beta + 1)}$, and $Z = X^{(\alpha)}$. Therefore their composition
	\[ L_n X^{(\alpha)} \cup_{L_n X^{(0)}} X^{(0)}_n \to X^{(\alpha)}_n \]
	is a cofibration as well.
	
	Finally, the latching map construction preserves retracts, and a retract of a cofibration is a cofibration, so a retract of a semifree cofibration is a semifree cofibration.
\end{proof}

\begin{cor}
	The class of semifree cofibrations is the smallest class containing the semifree spectra on cofibrations of spaces, and closed under pushouts, transfinite compositions, and retracts. In particular, every free cofibration is a semifree cofibration.
\end{cor}

For the next set of results, we need to establish conditions under which smashing and quotienting out an $O(n)$-action will preserve cofibrations, fibrant objects, and weak equivalences. Let $O$ be a compact Lie group, such as $O(n)$ for $n \geq 0$. A retractive $O$-equivariant space is a retractive space $Y$ over $B$ with a fiber-preserving $O$-action. Furthermore, when we have an $O$-equivariant map of retractive spaces $X \to Y$, we say it is an $f$-cofibration or $h$-cofibration if this is true after forgetting the $O$-action.

\begin{prop}\label{smashing_with_free_complex}
	Under the conventions above:
	\begin{enumerate}
		\item If $f: K \to L$ is a free $O$-cell complex of based spaces and $g: X \to Y$ is an $O$-equivariant map of retractive $B$-spaces that is a cofibration, then $f \square g$ constructed by $(- \barsmash -)_{O}$ is a cofibration.
		\item If $L$ is any finite free based $O$-cell complex, and $X$ is any $O$-equivariant retractive space that is both $h$-cofibrant and $h$-fibrant, then $(L \barsmash X)_O$ is $h$-fibrant.
		\item If $L$ is any finite free based $O$-cell complex, and $g: X \to Y$ is any $O$-equivariant map of retractive $h$-cofibrant spaces over $B$ that is nonequivariantly a weak equivalence, then $(\id_L \barsmash g)_{O}$ is a weak equivalence.
	\end{enumerate}
\end{prop}

\begin{proof}
	The key ingredient for the first claim is that there is an isomorphism
	\[ ((O \times D)_+ \barsmash Y)_{O} \cong (D_+ \barsmash Y) \]
	of retractive spaces over $B$ without $O$-actions (so we forget the action on the right), which is natural in the unbased space $D$ and the retractive space $Y \in \mc R(B)$. This follows from the pushout square \eqref{external_half_smash}.
	Therefore these claims hold in the special case where  $f$ is a single free based $O$-cell,
	\[ (O \times S^{d-1})_+ \to (O \times D^d)_+, \]
	by reduction to \autoref{prop:h_cofibrations_pushout_product}. But the case of general $f$ immediately follows because $- \square g$ preserves pushouts and compositions in the sense described in \autoref{pushout_product_lemmas}.
	
	For the third statement we induct on the skeleta of $L$. If $L$ is obtained from $K$ by pushout along a single cell
	\[ \xymatrix @R=1.7em{
		(O \times S^{d-1})_+ \ar[r] \ar[d] & (O \times D^d)_+ \ar[d] \\
		K \ar[r] & L
	} \]
	Then $g$ gives a map of two pushout squares of the form
	\[ \xymatrix @R=1.7em{
		S^{d-1}_+ \barsmash (-) \ar[r] \ar[d] & D^d_+ \barsmash (-) \ar[d] \\
		(K \barsmash (-))_O \ar[r] & (L \barsmash (-))_O.
	} \]
	The horizontal maps are cofibrations by the first part of this proposition, so each of these squares is a homotopy pushout square. The map $g$ induces equivalences on the top-left and top-right terms by \autoref{prop:h_cofibrations_pushout_product}, and on the bottom-left term by inductive hypothesis, and therefore induces an equivalence on the bottom-right term. By induction $(\id_{L'} \barsmash f)_O$ is an equivalence for all skeleta $L'$ of $L$, and therefore for $L$ itself.
	
	The second statement is proven by the same induction, and we only get one pushout square of the above form. The first three terms are $h$-fibrations and the top horizontal map is an $h$- (therefore $f$-)cofibration, so by \autoref{prop:clapp} the last term is an $h$-fibration.
\end{proof}

\begin{lem}
	$\mathscr J(\R^n,\R^{m+n})$ is a based free $O(n)$-cell complex.
\end{lem}
\begin{proof}
	This is a consequence of Illman's triangulation theorem \cite{illman}.
\end{proof}

Now we can show that semifree cofibrations are level cofibrations, generalizing \autoref{prop:free_implies_level}.
\begin{prop}\label{reedy_implies_skeleta_level}
	If the map of spectra $X \to Y$ is a semifree cofibration then each map of skeleta $\sk^n X \to \sk^n Y$ is a level cofibration.
\end{prop}

\begin{proof}
	We induct on $n$. The $0$-skeleta $\sk^0 X \to \sk^0 Y$ are just suspension spectra of the map of spaces $X_0 \to Y_0$, which is a cofibration by assumption (it is the 0th relative latching map). Since $S^m \barsmash -$ preserves cofibrations by \autoref{prop:h_cofibrations_pushout_product}, this gives a cofibration on each spectrum level $m$.
	
	For the inductive step, we focus on spectrum level $m$ for some $m \geq 0$. When $m < n$ we get the square in which the horizontal maps are homeomorphisms,
	\[ \xymatrix @R=1.7em{
		(\sk^{n-1} X)_m \ar[r] \ar[d] & (\sk^n X)_m \ar[d] \\
		(\sk^{n-1} Y)_m \ar[r] & (\sk^n Y)_m.
	} \]
	The left-hand vertical is a cofibration by inductive hypothesis. Therefore the right-hand vertical is a cofibration as well. When $m \geq n$, two applications of \autoref{prop:skeleta_pushout} gives us two pushout squares
	\[ \resizebox{\textwidth}{!}{$
		\xymatrix{
		\mathscr J(\R^n,\R^m) \barsmash_{O(n)} L_n X \ar[d] \ar[r] & \mathscr J(\R^n,\R^m) \barsmash_{O(n)} X_n \ar[d] \\
		(\sk^{n-1} X)_m \ar[r] & (\sk^n X)_m }
	\quad \ra \quad
	\xymatrix{
		\mathscr J(\R^n,\R^m) \barsmash_{O(n)} L_n Y \ar[d] \ar[r] &\mathscr J(\R^n,\R^m) \barsmash_{O(n)}  Y_n \ar[d] \\
		(\sk^{n-1} Y)_m \ar[r] & (\sk^n Y)_m }
	$}
	\]
	Applying \autoref{cofibration_of_pushouts} with the lower-left corners as $B$ and the upper-right corners as $C$, for the map of pushouts to be a cofibration, it suffices that the following two maps are cofibrations.
	\[ (\sk^{n-1} X)_m \ra (\sk^{n-1} Y)_m \]
	\[ \mathscr J(\R^n,\R^m) \barsmash_{O(n)} (L_n Y \cup_{L_n X} X_n \to Y_n) \]
	The first is by inductive hypothesis and the second is by \autoref{smashing_with_free_complex}. This completes the induction.
\end{proof}

\begin{cor}\label{reedy_implies_level}
	Every semifree cofibration is a level cofibration.
\end{cor}

The remaining results concern level equivalences and fibrations.

\begin{prop}
	A level equivalence of semifreely $h$-cofibrant spectra $f: X \simar X'$ gives a level equivalence of skeleta $\sk^n X \to \sk^n X'$ for all $n \geq 0$.
\end{prop}

\begin{proof}
	Induction on $n$. Again the $0$-skeleton $\sk^0 X$ is the suspension spectrum $\Sigma^\infty X_0$. Since both $X_0$ and $X'_0$ are $h$-cofibrant the map between their suspension spectra is an equivalence on every level by \autoref{prop:h_cofibrations_pushout_product}.
	
	For the inductive step we again use the pushout square of \autoref{prop:skeleta_pushout}, and prove the claim for each spectrum level $m \geq n$, assuming the claim is true for $(n-1)$ and for every value of $m$.
	
	In particular, we can assume that the map of spaces $L_n X \to L_n X'$ is a weak equivalence, since this is the statement of the result for $(n-1)$ at spectrum level $n$. We also know that the spaces $L_n X$, $L_n X'$, $X_n$, and $X_n'$ are $h$-cofibrant by \autoref{reedy_implies_skeleta_level}. Combining these observations with \autoref{smashing_with_free_complex}, we conclude that $f$ induces an equivalence on the top-left and top-right terms of the pushout square
	\[ \xymatrix @R=1.7em{
		\mathscr J(\R^n,\R^m) \barsmash_{O(n)} L_n X \ar[d] \ar[r] & \mathscr J(\R^n,\R^m) \barsmash_{O(n)} X_n \ar[d] \\
		(\sk^{n-1} X)_m \ar[r] & (\sk^n X)_m } \]
	Again by \autoref{smashing_with_free_complex} the top horizontal is an $h$-cofibration, and the inductive hypothesis tells us that $f$ induces an equivalence on the bottom-left term. We conclude that $f$ gives an equivalence on $(\sk^n X)_m$, completing the induction.
\end{proof}

\begin{cor}\label{cor:latching_objects_derived}
	A level equivalence of semifreely $h$-cofibrant spectra $X \simar X'$ gives an equivalence of the latching objects $L_n X \simar L_n X'$. Therefore the relative latching map $L_n X' \cup_{L_n X} X_n \to X_n'$ is a weak equivalence for every $n \geq 0$.
\end{cor}

\begin{prop}
	If $X$ is semifreely $h$-cofibrant and level $h$-fibrant then the each skeleton $\sk^n X$ is level $h$-fibrant.
\end{prop}

\begin{proof}
	The same induction on $n$ as above. The levels $S^n \barsmash X_0$ of the 0-skeleton are $h$-fibrant by \autoref{prop:h_cofibrations_pushout_product}. For the inductive step we get to assume that $L_n X$ is $h$-fibrant, and so both $L_n X$ and $X_n$ are both $h$-cofibrant and $h$-fibrant. Combining these observations with \autoref{smashing_with_free_complex}, we conclude that the first three vertices in the square
	\[ \xymatrix @R=1.7em{
		\mathscr J(\R^n,\R^m) \barsmash_{O(n)} L_n X \ar[d] \ar[r] & \mathscr J(\R^n,\R^m) \barsmash_{O(n)} X_n \ar[d] \\
		(\sk^{n-1} X)_m \ar[r] & (\sk^n X)_m } \]
	are $h$-fibrant, and the top horizontal is an $h$-cofibration. By \autoref{prop:clapp} again the last vertex is also $h$-fibrant.
\end{proof}

\begin{cor}\label{cor:latching_objects_fibrant}
	If $X$ is semifreely $h$-cofibrant and level $h$-fibrant then each latching space $L_n X$ is $h$-fibrant.
\end{cor}

The climax of this section is the following extension of \autoref{prop:spectra_pushout_product}.
\begin{thm}\label{prop:reedy_pushout_product}
	Let $f: X \to Y$ and $g: X' \to Y'$ be maps of orthogonal spectra over $A$ and $B$, respectively. Form $f \square g$ using the external smash product of spectra.
	\begin{enumerate}
		\item If $f$ and $g$ are semifree $h$-cofibrations, so is their pushout-product $f \square g$. The same applies to semifree $f$-cofibrations and semifree closed inclusions.
		\item If $Y$ and $Y'$ are both semifreely $h$-cofibrant and level $h$-fibrant then so is $Y \barsmash Y'$.
		\item If $f$ is a semifree $h$-cofibration, $g$ is a level equivalence, and all four spectra are semifreely $h$-cofibrant, then $f \square g$ is a level equivalence.
	\end{enumerate}
\end{thm}

\begin{proof}
	Using \autoref{prop:map_skeleta_pushout} and \autoref{pushout_product_lemmas}, (1) reduces to the case where the maps $f$ and $g$ are of the form
	\[ f: \mc G_m(K \to L), \qquad g: \mc G_n(K' \to L'), \]
	where $\phi: K \to L$ and $\gamma: K' \to L'$ are cofibrations of ex-spaces over $A$ and $B$, respectively. By \autoref{smash_two_semi_free} and the fact that the semifree spectrum functor preserves colimits we get
	\[ f \square g \cong \mc G_{m+n}(O(m+n)_+ \barsmash_{O(m) \times O(n)} (\phi \square \gamma)) \]
	Therefore it suffices to prove that $O(m+n)_+ \sma_{O(m) \times O(n)} (\phi \square \gamma)$ is a $h$-cofibration. This follows from one application of \autoref{smashing_with_free_complex} and one of \autoref{prop:h_cofibrations_pushout_product}.
	
	For (2), it suffices to prove this when $Y$ is equal to its $m$-skeleton and $Y'$ is equal to its $m'$-skeleton, for any value of $m$ and $m'$. When both are equal to 0 it is clearly true. Incrementing $m$ or $m'$ requires us to replace one of the two terms by a pushout along a map of the form $\mc G_{m+1}(K \to L)$ where $K$ and $L$ are both $h$-cofibrant and $h$-fibrant (by \autoref{cor:latching_objects_fibrant}), and $K \to L$ is an $h$-cofibration. By \autoref{smashing_with_free_complex} we know that each level of $\mc G_{m+1} K$ and $\mc G_{m+1} L$ is $h$-fibrant, therefore at each spectrum level we get a square of the form found in \autoref{prop:clapp}, therefore the final term is also $h$-fibrant.
	
	For (3) assume first that $f$ and $g$ are in the special form described in (1) above, with $\gamma\colon K' \to L'$ a weak equivalence. Then $f \square g$ is a semifree spectrum on a weak equivalence, and is therefore a level equivalence by \autoref{smashing_with_free_complex}. Next, if $g$ is a semifree cofibration and level equivalence, it factors as a sequence of pushouts of maps of the previous form, by \autoref{prop:map_skeleta_pushout} and \autoref{cor:latching_objects_derived}. Therefore $f \square g$ is a composition of pushouts of maps that are level $h$-cofibrations and weak equivalences, so $f \square g$ is also a level $h$-cofibration and weak equivalence. By the same argument, we may also allow $f$ to be a general semifree $h$-cofibration.
	
	Finally we adapt the proof of Ken Brown's lemma. For fixed $f$, the class of maps $g$ between semifreely cofibrant spectra such that $f \square g$ is a level equivalence includes the acyclic cofibrations (i.e. semifree cofibrations that are level equivalences) and is closed under 2-out-of-3. Given an arbitrary level equivalence $X \to Y$ of semifreely cofibrant spectra, we factor the obvious map $X \vee Y \to Y$ into a cofibration $X \vee Y \to Z$ followed by a level equivalence $Z \to Y$, using \autoref{prop:level_model_structure}. Then the inclusions of $X$ and $Y$ into $Z$ are semifree cofibrations and level equivalences, so they are in the class. The composite $Y \to Z \to Y$ is the identity, which is in the class, and so by 2 out of 3 the projection $Z \to Y$ is in the class. Then $X \to Z \to Y$ is a composite of two maps in the class and is therefore in the class as well. This finishes the proof.

\end{proof}

\begin{rmk}
	(3) can be strengthened: If $f$ is a semifree $h$-cofibration, $g$ is a level equivalence, $X$, $Y$, and $X'$ are semifreely $h$-cofibrant, and $Y'$ is level $h$-cofibrant, then $f \square g$ is a level equivalence. However the proof of this strengthening is more complicated.
\end{rmk}


\section{Stable equivalences}\label{sec:stable}

In this section, we recall the definition of stable equivalences from \cite{ms}, giving the homotopy category $\ho\Osp(B)$ of spectra over $B$ and the homotopy cateory $\ho\Osp$ of spectra over all base spaces. We prove the existence of the ``$q$-model structure'' in \autoref{thm:stable_model_structure}. We then use it to derive the operations $\barsmash$, $f^*$, and $f_!$ with respect to the stable equivalences, making the homotopy category $\ho\Osp$ into a symmetric monoidal bifibration (SMBF). As explained in the introduction, this is the key structure that will give us the bicategory and all of the applications to fixed-point theory in later sections.

The results of this section use the convenient class of cofibrant and fibrant objects from \autoref{thm:intro_cof_fib} in a critical way. The replacement $PQX$ is needed to prove left properness in \autoref{cor:spectra_left_proper}, which supplies the missing ingredient for the proof of the $q$-model structure (\autoref{j_equiv}). After that, we construct the SMBF structure on $\ho\Osp$ in \autoref{thm:spectra_SMBF} by restricting to the convenient cofibrant and fibrant objects, and using the $q$-model structure to show that equivalences between them are preserved.

\beforesubsection
\subsection{Stable homotopy groups}\aftersubsection

If $X$ is a parametrized spectrum over $B$, its (stable) homotopy groups\index{homotopy!groups} are defined to be the stable homotopy groups of the fiber spectra $X_b$, for $b \in B$ and $n \in \Z$:
\[ \pi_{n,b}(X) := \pi_n(X_b) = \underset{k}\colim \pi_{k+n}((X_b)_n). \]

These homotopy groups do not preserve level equivalences. For instance the suspension spectrum of $\{0\}_{+I} \to I_{+I}$ over $I$ is a level equivalence but not an isomorphism on $\pi_{0,1/2}$. However, the homotopy groups are right-deformable, using the level $q$-fibrant spectra and the level fibrant replacement functor $R^{lv}$ from \autoref{prop:level_model_structure}. We therefore get right-derived homotopy groups, cf. \cite[12.3.4]{ms}:
\[ \R\pi_{n,b}(X) \cong \pi_{n,b}(R^{lv}X) \cong \pi_{n,b}(PQX). \]
\begin{df}\label{stable_equivalence}
A map $X \to Y$ is a \textbf{stable equivalence}\index{stable equivalence} if it induces isomorphisms on the derived homotopy groups for all $b \in B$ and $n \in \Z$. Informally, it is a map that induces isomorphisms on $\pi_*$ after we make the levels of our spectra into fibrations.
\end{df}

The class of stable equivalences is generated under 2-out-of-3 by the level equivalences and the maps of level $h$-fibrant spectra (or more generally level $q$-fibrant spectra) inducing isomorphisms on the homotopy groups of each fiber. We can see this directly from the fact that every map $X \to Y$ fits into the following diagrams.
\begin{equation}\label{eq:stable_equivalence_decomposition}
\xymatrix @R=1.7em{
	X \ar[d] \ar[r]^-\sim & R^{lv}X \ar[d] \\
	Y \ar[r]^-\sim & R^{lv}Y
}
\qquad
\xymatrix @R=1.7em{
	X \ar[d] & \ar[l]_-\sim QX \ar[d] \ar[r]^-\sim & PQX \ar[d] \\
	Y & \ar[l]_-\sim QY \ar[r]^-\sim & PQY.
}
\end{equation}

\begin{ex}\label{ex:equivalent_thom_spectra} \hfill
	\vspace{-1em}
	
	\begin{itemize}
		\item If $X \to B$ is a weak Hausdorff fibration then the map from the fiberwise suspension spectrum of $X$ to its classical fibrant replacement
		\[ \Sigma^\infty_{+B} X \to Q_B(\Sigma^\infty_{+B} X) = \underset{k}\colim \Omega^k_B \Sigma^{k}_B \Sigma^\infty_{+B} X \]
		is a stable equivalence. This is because both spectra are fibrations at every level, and on each fiber this statement is a classical fact in stable homotopy theory.
		\item For $X \in \mc R(B)$ there is a standard map
		\[ F_{n+1} \Sigma_B X \to F_n X. \]
		As long as $X$ is an $h$-cofibrant space, this standard map is a stable equivalence. The proof is simple: apply $P$ to make the levels fibrations. Since $P$ commutes with $\Sigma_B$ and $F_n$, this now follows from the same statement for non-parametrized spectra.
		\item By the previous example and \autoref{ex:thom_smash_product}, we have for any vector bundle $W$ over $B$ a stable equivalence
		\[ F_{m+n} \Th_B(\R^m \times W) \overset\sim\to F_n \Th_B(W). \]
		Therefore different models for the Thom spectrum of a virtual bundle $\xi$ (\autoref{ex:thom_spectra}) are stably equivalent.
	\end{itemize}
\end{ex}

\begin{lem}\label{ex:smash_with_cell_complex_is_derived}
	The following operations preserve stable equivalences:
	\begin{itemize}
		\item External smash products $K \barsmash X$, if $K$ is an $h$-cofibrant retractive space and $X$ is a level $h$-cofibrant spectrum.
		\item External mapping spectra $\barF_*(K,X)$, if $K$ is a finite cell complex over $*$ and $X$ is a level $q$-fibrant spectrum.
	\end{itemize}
\end{lem}

\begin{proof}
	The general method here and for the following results is to first check these operations preserve level equivalences. Using \eqref{eq:stable_equivalence_decomposition}, this reduces to the case where the spectra are fibrant. But once the spectra are fibrant, it suffices to show the map is a $\pi_*$ isomorphism on each fiber, and this is the non-parametrized case already handled in \cite[7.4(i,vi)]{mmss}.
\end{proof}

\begin{prop}(cf. \cite[12.4.3]{ms}, \cite[7.4]{mmss})\label{prop:coproduct_colimit_stable_equivalences}
	The following operations commute with derived homotopy groups, and therefore preserve stable equivalences:
	\begin{itemize}
		\item finite products $\prod_{i=1}^n X^i$ of level $q$-fibrant spectra $X^i$, 
		\item arbitrary coproducts $\bigcup_\alpha X^\alpha$ of level $h$-cofibrant spectra $X^\alpha$, and
		\item sequential colimits $\colim_i X^i$ along level $h$-cofibrations $X^i \to X^{i+1}$.\footnote{In (CGWH) these assumptions can be weakened. The maps of the colimit system only have to be level closed inclusions. In the coproduct, the $h$-cofibrant assumption can be dropped when $B = *$ but probably not in general.}
	\end{itemize}
\end{prop}

\begin{proof}
	The method is the same as in the previous lemma, only for the last two points we also have to observe that the monoidal fibrant replacement functor $P$ commutes with the coproduct or colimit.
\end{proof}

\begin{rmk}\label{level_stable_same_derived}
	\autoref{ex:smash_with_cell_complex_is_derived} implies that the left-derived functor of $K \barsmash -$ is the same whether we use the level equivalences or the stable equivalences to derive the functor. This is because the same cofibrant replacement ($Q$ from \autoref{prop:level_model_structure}) can be used in both cases. This happens more generally whenever a level cofibrant or fibrant replacement is enough to preserve stable equivalences, so it also applies to the other left- and right-derived functors in the previous two results.
	
	The only cases in this paper where this fails are $f_*$ from \autoref{prop:sheafy_pushforward} when the fibers are non-compact, $\barF_*(K,-)$ when $K$ is infinite, and homotopy limits when the indexing category does not have a compact nerve (this includes infinite products). In these cases one has to take a stable fibrant replacement from \autoref{thm:stable_model_structure} in order to derive with respect to the stable equivalences.
\end{rmk}

\begin{prop}\label{parametrized_stability}
	The following natural transformations induce (via \autoref{prop:passing_natural_trans_to_derived_functors}) stable equivalences between the composites of derived functors. (See also \autoref{fiber_shift_of_cofiber}.)
	\[ X \to \Omega_B\Sigma_B X, \qquad \Sigma_B\Omega_B X \to X, \quad \textup{ and } \quad X \cup_B Y \to X \times_B Y. \]
\end{prop}

\begin{proof}
	Since this is a statement about the derived functors, without loss of generality $X$ and $Y$ are level $f$-cofibrant, level $h$-fibrant, and weak Hausdorff. Then the spectra above are level $h$-fibrant, and the functors $\Sigma_B$, $\Omega_B$, $\cup_B$ and $\times_B$ are equivalent to their derived versions, by repeated application of \autoref{prop:h_cofibrations_pushout_product} and \autoref{h_fibrations_pullback_hom}. So it suffices to see in this case that each of the maps is a $\pi_*$-isomorphism on each fiber separately. This is done in \cite[7.4(i,i',ii)]{mmss}.
\end{proof}

\beforesubsection
\subsection{Cofibers, fibers, pushouts and pullbacks}\aftersubsection

Let $f\colon X \to Y$ be a map of spectra over $B$. Define the \textbf{mapping cone} (or uncorrected homotopy cofiber)\index{uncorrected homotopy cofiber}\index{mapping cone} of $f$ by the formula
\[ C_Bf = (X \barsmash I) \cup_{X \barsmash S^0} Y. \]
As in \autoref{space_cofiber}, on each fiber this gives the usual mapping cone $Cf_b = (X_b \sma I) \cup_{X_b} Y_b$. As in \autoref{commute_cone_with_P}, we have an isomorphism of spectra $PC_B f \cong C_B Pf$. By the discussion in \autoref{space_cofiber}, the mapping cone preserves level equivalences when $f$ is a level $h$-cofibration or both $X$ and $Y$ are level $h$-cofibrant.

The \textbf{homotopy cofiber}\index{homotopy!cofiber} the left-derived functor of the mapping cone, as a functor from maps of spectra to spectra, using the level equivalences:
\[ \L C_B f = C_B(Qf). \]
Note that the map to the strict cofiber $C_B f \to Y \cup_X B$ induces an equivalence of left-derived functors, so we could also think of the homotopy cofiber as left-derived from the cofiber.

\begin{lem}\label{lem:LES}
	There is a natural long exact sequence
	\[ \xymatrix @C=3em{
		\ldots \ar[r] & \R\pi_{n,b}(X) \ar[r]^-{f_*} & \R\pi_{n,b}(Y) \ar[r] & \R\pi_{n,b}(\L C_B f) \ar[r] & \R\pi_{n-1,b}(X) \ar[r] & \ldots
	}. \]
	Therefore $\L C_B f = C_BQf$ preserves stable equivalences.
\end{lem}

\begin{proof}	
	We take the usual the non-parametrized long exact sequence for the homotopy groups of the fibers of $PQX \to PQY \to C_B PQf$, then identify the homotopy groups of $C_B PQf$ with those of $PQC_B Qf$ along the functorial string of maps
	\[ \xymatrix{
		C_B PQf \ar@{<->}[r]^-\cong & PC_B Qf & \ar[l]_-\sim PQC_B Qf
	} \]
	inducing isomorphisms on the (underived) homotopy groups of every fiber spectrum. As usual, it is not necessary to worry whether all possible reasonable choices of isomorphism give the same map. We only need to know that some functorial long exact sequence exists. 
\end{proof}

As a result $C_B$ preserves \emph{stable} equivalences when $f$ is a level $h$-cofibration or both $X$ and $Y$ are level $h$-cofibrant. As in \autoref{level_stable_same_derived}, we can also conclude that $\L C_B f = C_BQf$ agrees with the left-derived functor of $C_B f$ using the \emph{stable} equivalences.

\begin{cor}[Left-properness]\label{cor:spectra_left_proper}
	In a strict pushout square of spectra over $B$
	\[ \xymatrix @R=1.7em{
		X \ar[r]^-i \ar[d]^-f & Y \ar[d] \\
		Z \ar[r] & Y \cup_X Z
	} \]
	if $f\colon X \to Z$ is a stable equivalence and either $i$ or $f$ is a level $h$-cofibration then $Y \to Y \cup_X Z$ is also a stable equivalence.
\end{cor}
\begin{cor}[Gluing lemma]\label{cor:spectra_gluing}
	Any diagram of spectra over $B$
	\[ \xymatrix @R=1.7em{
		Y \ar[d]^-\sim & X \ar[l]_-i \ar[r]^-f \ar[d]^-\sim & Z \ar[d]^-\sim \\
		Y' & X' \ar[l]^{i'} \ar[r]_-{f'} & Z'
	} \]
	in which the vertical maps are stable equivalences and both $i$ and $i'$ are level $h$-cofibrations, induces a stable equivalence of pushout spectra
	\[ Y \cup_X Z \to Y' \cup_{X'} Z' \]
\end{cor}\index{gluing lemma}

\begin{proof}
	For left properness, if $i$ is the cofibration, we factor $X \to Z$ into a level $h$-cofibration and a level equivalence, and from this reduce to the case where $f$ is the cofibration. Then take mapping cones in the vertical direction. For $X \to Z$ this cone is weakly contractible by \autoref{lem:LES}, but the cones are homeomorphic because the above square is a pushout. Therefore the cone of $Y \to Y \cup_X Z$ is also weakly contractible. By \autoref{lem:LES} again, $Y \to Y \cup_X Z$ is a stable equivalence.
	
	The gluing lemma is known to follow from left-properness by a long diagram-chase. Alternatively, we compare the mapping cones of $i\colon X \to Y$ and $\bar i\colon Z \to Y \cup_X Z$. In the square of mapping cones
	\[ \xymatrix @R=1.7em{
		\L C_B i \ar[d] \ar[r] & \L C_B \bar i \ar[d] \\
		\L C_B i' \ar[r] & \L C_B \bar i',
	} \]
	the horizontal maps are equivalences because they are isomorphisms before $\L$ and the maps $i, \bar i, i', \bar i'$ are all level $h$-cofibrations, hence $\L C_B \simeq C_B$. The left-vertical map is a stable equivalence by one application of \autoref{lem:LES} to $X \to Y$, therefore the right-vertical is a stable equivalence as well. By one more application of \autoref{lem:LES} to $\bar i$ and $\bar i'$, the map of pushouts is a stable equivalence.
\end{proof}

\begin{rmk}\label{fixed_may_sigurdsson_2}
	In \cite{ms}, the corresponding results assume that $i$ is a level $f$-cofibration, which is stricter than a level $h$-cofibration. Furthermore, the arguments employed in \cite[Ch. 5-6]{ms} and the counterexample \cite[6.1.5]{ms} suggest that no further improvement is possible. We would like to expand on \autoref{fixed_may_sigurdsson} by explaining why this approach avoids the difficulty and proves a stronger theorem. One difference is that we are using a symmetric monoidal fibrant replacement functor $P$, allowing us to commute fibrant replacement with the mapping cone.
	
	However, the most significant change is that we are thinking in terms of $\barsmash$ instead of $\sma_B$. If we think in terms of $\sma_B$, we are led to define the uncorrected homotopy cofiber as
	\[ C_B f = (X \sma_B I_B) \cup_{X \sma_B S^0_B} Y \]
	instead of $(X \barsmash I) \cup_{X \barsmash S^0} Y$. Though these two models are homeomorphic, the one with internal smash products suggests that $C_B f$ only preserves equivalences if $X$ and $Y$ are level $f$-cofibrant, even though it also preserves equivalences when they are level $h$-cofibrant. This small difference in assumptions leads to a large difference once we construct the model structure. If the cofibrant objects need to be level $f$-cofibrant, then the usual cells $S^{n-1}_{+B} \to D^n_{+B}$ won't work, we need to further restrict to cells where the map $S^{n-1} \to D^n$ is an $f$-cofibration, in other words the ``$qf$-cells.''\footnote{The language of ``well-grounded model categories'' is used in \cite{ms} to explain what assumptions are needed to get the gluing lemma to work. In that language, the $q$-model structure is not well-grounded \cite[6.1.3]{ms}. However, it actually is well-grounded, using a \emph{different ground structure}, that we obtain from the one in \cite[5.3.6]{ms} by replacing the $f$-cofibrations with $h$-cofibrations. If we examine the second half of the counterexample \cite[6.1.5]{ms} carefully, we see that it does not obstruct the existence of this ground structure, because $X \to Y'$ is not an $h$-cofibration to begin with.}
\end{rmk}

Define the mapping cocone or \textbf{uncorrected homotopy fiber}\index{uncorrected homotopy fiber} of $f\colon X \to Y$ by
\[ F_B f = X \times_{\barF_*(S^0,Y)} \barF_*(I,Y). \]
On each fiber over $B$, this gives the usual homotopy fiber $F_B f_b = X_b \times_{Y_b} F(I,Y_b)$ of the map $f_b$. Note however that it may not preserve level equivalences because the fibers $X_b$ and $Y_b$ could themselves change.

The \textbf{homotopy fiber}\index{homotopy!fiber} is the right-derived functor of $F_B$ using the level equivalences. It is equivalent to $FR^{lv}f$, or $F_B f$ whenever $f$ is a level $q$-fibration or a map between level $q$-fibrant spectra. On the subcategory of spectra whose levels are $h$-cofibrant, it is also equivalent to $F_B Pf$. This is all by the discussion in \autoref{space_fiber}.
\begin{lem}\label{lem:LES2}
	There is a natural long exact sequence
	\[ \xymatrix{
		\ldots \ar[r] & \R\pi_{n,b}(X) \ar[r]^-{f_*} & \R\pi_{n,b}(Y) \ar[r] & \R\pi_{n-1,b}(F_B R^{lv}f) \ar[r] & \R\pi_{n-1,b}(X) \ar[r] & \ldots
	}. \]
	Therefore $\R F f = F_B R^{lv}f$ preserves stable equivalences.
\end{lem}

\begin{proof}
	The proof is similar to that of \autoref{lem:LES}, but easier. We take the usual non-parametrized long exact sequence for the fibers of $F_B R^{lv}f \to R^{lv}X \to R^{lv}Y$. We check that $F_B R^{lv}X$ is level $q$-fibrant and therefore $\R\pi_{n,b}(F_B R^{lv}f) \cong \pi_{n,b}(F_B R^{lv}f)$.
\end{proof}

As a result $F_B$ preserves \emph{stable} equivalences when $f$ is a level $q$-fibration or both $X$ and $Y$ are level $q$-fibrant. As in \autoref{level_stable_same_derived}, we can also conclude that $\R F f = F_B R^{lv}f$ agrees with the right-derived functor of $F_B f$ using the \emph{stable} equivalences.

\begin{cor}[Right-properness]\label{cor:spectra_right_proper}
	In a strict pullback square of spectra over $B$
	\[ \xymatrix @R=1.7em{
		Y \times_W Z \ar[r] \ar[d] & Y \ar[d]^-p \\
		Z \ar[r]_-f & W
	} \]
	if $f\colon Z \to W$ is a stable equivalence and either $p$ or $f$ is a level $q$-fibration then $Y \times_W Z \to Y$ is also a stable equivalence.
\end{cor}
\begin{cor}[Dual gluing lemma]
	Any diagram of spectra over $B$
	\[ \xymatrix @R=1.7em{
		Y \ar[d]^-\sim & W \ar@{<-}[l]_-p \ar@{<-}[r]^-f \ar[d]^-\sim & Z \ar[d]^-\sim \\
		Y' & W' \ar@{<-}[l]^{p'} \ar@{<-}[r]_-{f'} & Z'
	} \]
	in which the vertical maps are stable equivalences and both $p$ and $p'$ are level $q$-fibrations, induces a stable equivalence of pullback spectra
	\[ Y \times_W Z \to Y' \times_{W'} Z' \]
\end{cor}

\begin{ex}\label{fiber_shift_of_cofiber}
	The natural transformation $F_B f \to \Omega_B C_B f$, defined as in the non-parametrized case (\cite{mmss}), induces an equivalence on the derived functors $(\R F_B)f \simeq (\R\Omega_B)(\L C_B)f$, by the same argument as in \autoref{parametrized_stability}.
\end{ex}

\begin{ex}\label{pushout_equals_pullback}
	One can define a homotopy pushout square of parametrized spectra to be a commuting square such that the induced map from the left-derived pushout to the final vertex is a stable equivalence. It does not matter whether we use level or stable equivalences to derive the pushout, and a square is homotopy pushout iff it is homotopy pullback.
\end{ex}

\beforesubsection
\subsection{The stable model structure}\aftersubsection

We are now ready to prove the existence of the ``$q$-model structure'', \autoref{thm:intro_q}. We often refer to it as the stable model structure, since it is the only one we use that has stable equivalences.

\begin{thm}[Stable model structure]\label{thm:stable_model_structure}
	There is a proper model structure on $\Osp(B)$ and on $\Psp(B)$ whose weak equivalences are the stable equivalences. Each of these model structures is cofibrantly generated by the sets of maps
	\[ \begin{array}{ccrll}
		I &=& \{ \ F_k\left[ S^{n-1}_{+B} \to D^n_{+B} \right] & : n,k \geq 0, & D^n \to B \ \} \\
		J &=& \{ \ F_k\left[ D^n_{+B} \to (D^n \times I)_{+B} \right] &: n,k \geq 0, & (D^n \times I) \to B \ \} \\
		& \cup & \{ \ k_{i,j} \ \square \ \left[ S^{n-1}_{+B} \to D^n_{+B} \right] & : i,j,n \geq 0, & D^n \to B \ \}.
	\end{array} \]
\end{thm}\index{stable model structure}

\begin{rmk}
	This model structure was not known to exist at the time of \cite{ms}. The ``$qf$-model structure'' of \cite{ms} is similar except that one restricts further to those cells for which the maps $S^{n-1} \to D^n$ are also $f$-cofibrations.
\end{rmk}

In the statement of \autoref{thm:stable_model_structure}, the pushout-product is carried out using the operation of smashing a spectrum over $*$ with a space over $B$, see \autoref{sec:smash_with_space}. The map of non-parametrized spectra $k_{i,j}$ includes the front end of the mapping cylinder $\Cyl_{i,j}$ of the map
\[ \lambda_{i,j}\colon F_{i+j} S^j \ra F_i S^0. \]
As in \autoref{ex:equivalent_thom_spectra}, in sequential spectra $\lambda_{i,j}$ is just a truncation, while in orthogonal spectra it arises from the map of spaces $S^j \to \mathscr J(\R^i,\R^{i+j})$ that identifies $S^j$ with the fiber over the standard embedding $\R^i \to \R^{i+j}$. So $\Cyl_{i,j}$ is defined as the pushout
\begin{equation}\label{eq:expanded_cylinder}
\xymatrix @R=1.7em{
	F_{i+j} S^j \ar@{<->}[r]^-\cong \ar[d]_-{\lambda_{i,j}} & \{1\}_+ \sma F_{i+j} S^j \ar[d]_-{\lambda_{i,j}} \ar[r] \ar@{}[rd]|(.85)*\txt{\huge $\ulcorner$} & I_+  \sma F_{i+j} S^j \ar[d] & \{0\}_+ \sma F_{i+j} S^j \ar[l] \ar[ld]^-{k_{i,j}} \\
	F_i S^0 \ar@{<->}[r]^-\cong & \{1\}_+ \sma F_i S^0 \ar[r] & \Cyl_{i,j}
}
\end{equation}
where all the horizontal maps are all free $q$-cofibrations by \autoref{ex:smashing_spaces_is_left_Quillen}. This implies all the horizontal maps in the diagram below are free $q$-cofibrations, hence so is $k_{i,j}$.
\[ \xymatrix @R=1.7em{
	& \{0,1\}_+ \sma F_{i+j} S^j \ar[d] \ar[r] \ar@{}[rd]|(.85)*\txt{\huge $\ulcorner$} & I_+ \sma F_{i+j} S^j \ar[d] \\
	\{0\}_+ \sma F_{i+j} S^j \ar[r] \ar@/_2em/[rr]_-{k_{i,j}} & \{0\}_+ \sma F_{i+j} S^j \vee \{1\}_+ \sma F_i S^0  \ar[r] & \Cyl_{i,j}
} \]

Before proving \autoref{thm:stable_model_structure} we establish some preliminaries. For any set of maps $I$ in a category, an \textbf{$I$-cell complex} is a sequential composition of pushouts of coproducts of maps in $I$. An \textbf{$I$-injective map} is any map with the right lifting property with respect to $I$. We call the set of such maps $I$-inj. An \textbf{$I$-cofibration} is any map with the left lifting property with respect to $I$-injective maps. 

\begin{prop}\label{j_equiv}
	Each map in $J$ is a stable equivalence.\footnote{The earlier proof of the $q$-model structure (see e.g. \cite{hu_duality}) was nearly complete, but had a gap at this step. As explained in \cite{ms}, the critical missing piece was the left-properness statement in \autoref{cor:spectra_left_proper}.}
\end{prop}

\begin{proof}
	For the first set of maps in $J$ this is clear because they are level equivalences. For the second set, we draw the following diagram in which the square is a pushout.
		\[ \xymatrix @R=1.7em{
			F_{i+j} S^j \barsmash S^{n-1}_{+B} \ar[r] \ar[d]  \ar@{}[rd]|(.85)*\txt{\huge $\ulcorner$} & Cyl_{i,j} \barsmash S^{n-1}_{+B} \ar[d] \ar[rd] \\
			F_{i+j} S^j \barsmash D^n_{+B} \ar[r] & Y \ar[r] & Cyl_{i,j} \barsmash D^n_{+B}
		} \]
		Since the left vertical map is a $q$-cofibration, to prove that the map from $Y$ into the final term is a stable equivalence, by \autoref{cor:spectra_left_proper} it suffices to argue that both horizontal maps of the form
		\[ F_{i+j} S^j \barsmash Z_{+B} \to Cyl_{i,j} \barsmash Z_{+B} \]
		are stable equivalences. In each case, the source and target are level $h$-cofibrant so we may apply $P$ to get
		\[ \xymatrix @R=1.7em{
			F_{i+j} [ S^j \barsmash (Z \times_B B^I)_{+B} ] \ar[r]^-\sim & F_{i+j} [ S^j \sma I_+ \barsmash (Z \times_B B^I)_{+B} ] \ar[r]  & Cyl_{i,j} \barsmash (Z \times_B B^I)_{+B} \\
			& F_{i+j} [ S^j \sma \{1\}_+ \barsmash (Z \times_B B^I)_{+B} ] \ar[r]^-\sim \ar[u] & F_{i} [ S^0 \barsmash (Z \times_B B^I)_{+B} ] \ar[u] \\
		} \]
		The first marked equivalence is a level equivalence. The second is a stable equivalence, because on each fiber it is the $\pi_*$-isomorphism $\lambda_{i,j}$ smashed with the well-based space $[(Z \times_B B^I)_b]_+$, see \autoref{ex:smash_with_cell_complex_is_derived}. Since the vertical maps are level $h$-cofibrations, all horizontal maps are therefore stable equivalences by \autoref{cor:spectra_left_proper}, finishing the argument.
\end{proof}

Let $p\colon X \to Y$ be a map of spectra over $B$.

\begin{lem}\label{lem:stable_fibrations}
$p$ is $J$-injective if and only if
		\begin{itemize}
			\item[(1)] each $p_n\colon X_n \to Y_n$ is a $q$-fibration, and either
			\item[(2)] the square \eqref{eq:fibration_of_spectra_means_this_is_a_pullback} is a homotopy pullback, cf. \cite[12.5.6]{ms}, or
			\item[($2'$)] the same square with right-derived loopspaces $\R\Omega_B^j$ is a homotopy pullback.
		\end{itemize}
		\begin{equation}\label{eq:fibration_of_spectra_means_this_is_a_pullback}
		\xymatrix @R=1.7em{
			X_i \ar[r] \ar[d]^-{p_i} & \Omega^j_B X_{i+j} \ar[d]^-{\Omega_B^j p_{i+j}} \\
			Y_i \ar[r] & \Omega^j_B Y_{i+j}
		}
		\end{equation}
\end{lem}

\begin{proof}
 The adjunction between free spectra and the forgetful functor to spaces implies that $p$ is $I$-injective iff each level $p_n\colon X_n \to Y_n$ is an acyclic $q$-fibration. Moreover, $p$ is $J$-injective iff
		\begin{itemize}
			\item each $p_n$ is just a $q$-fibration (the first set of maps in $J$), and
			\item $\Hom_\square(k_{i,j},p)$ (a map in $\mc R(B)$) is an acyclic $q$-fibration (by \autoref{lem:pullback_hom_adjunction}).
		\end{itemize}
		Since $k_{i,j}$ is always a $q$-cofibration, when $p$ is a level $q$-fibration the map $\Hom_\square(k_{i,j},p)$ is also a $q$-fibration. So we can rearrange the above necessary and sufficient condition for $p$ to be $J$-injective:
		\begin{itemize}
			\item each $p_n$ is a $q$-fibration, and
			\item $\Hom_\square(k_{i,j},p)$ is a weak equivalence.
		\end{itemize}
		It will be convenient to rearrange this one more time. Writing out the definition of $\Hom_\square(k_{i,j},p)$ as
		\begin{equation}\label{eq:horrible_map}
		\barmap_*(\Cyl_{i,j},X) \to \barmap_*(F_{i+j} S^j,X) \times_{\barmap_*(F_{i+j} S^j,Y)} \barmap_*(\Cyl_{i,j},Y),
		\end{equation}
		we see using the level model structure that when $p$ is a level $q$-fibration, the map
		\[ \barmap_*(F_{i+j} S^j,X) \to  \barmap_*(F_{i+j} S^j,Y) \]
		is a $q$-fibration of spaces, hence the pullback is a homotopy pullback. Therefore we can replace the terms of the form $\barmap_*(\Cyl_{i,j},Y)$ with the equivalent terms $\barmap_*(F_i S^0,Y)$\footnote{One might think we need to assume $Y$ is $q$-fibrant for this, but that is not necessary. The inclusion $F_i S^0 \to \Cyl_{i,j}$ is a homotopy equivalence of spectra, in the sense that it has an inverse up to homotopy. The functor $\barmap_*(-,Y)$ preserves this homotopy equivalence with no assumptions on $Y$.}, giving the weakly equivalent map
		\[ \barmap_*(F_{i} S^0,X) \to \barmap_*(F_{i+j} S^j ,X) \times_{\barmap_*(F_{i+j} S^j,Y)} \barmap_*(F_{i} S^0,Y). \]
		Using the universal property of free spectra, this map is homeomorphic to
		\[ X_i \to \Omega_B^j X_{i+j} \times_{\Omega_B^j Y_{i+j}} Y_i. \]
		We also verify that $X_i \to \Omega_B^j X_{i+j}$ is the map we expect by tracing through the composite $F_{i+j} S^j \to F_i S^0 \to X$ at spectrum level $i+j$.
		
		Finally we show that (2) is equivalent to ($2'$) in the presence of (1). We model $\R\Omega_B$ by picking any level equivalence from $p\colon X \to Y$ to another level $q$-fibration $p'\colon X' \to Y'$ in which $Y'$ is level $q$-fibrant. This gives the square below on the left, which is a homotopy pullback. Since its vertical maps are fibrations, their fibers are fibrant over $B$, hence when we take $\Omega^j_B$ they are still equivalent. Therefore the square on the right induces an equivalence on the strict fibers of the vertical maps. Since we already know the vertical maps are fibrations, this implies it is a homotopy pullback square.
		\[ \xymatrix @R=1.7em{
			X_{i+j} \ar@{->>}[d]^-{p_{i+j}} \ar[r]^-\sim & X_{i+j}' \ar@{->>}[d]^-{p_{i+j}'} && \Omega^j_B X_{i+j} \ar@{->>}[d]^-{\Omega_B^j p_{i+j}} \ar[r] & \Omega^j_B X_{i+j}' \ar@{->>}[d]^-{\Omega_B^j p_{i+j}'} \\
			Y_{i+j} \ar[r]^-\sim & Y_{i+j}' && \Omega^j_B Y_{i+j} \ar[r] & \Omega^j_B Y_{i+j}'.
		} \]
		Therefore square \eqref{eq:fibration_of_spectra_means_this_is_a_pullback} is a homotopy pullback with strict $\Omega^j_B$ iff it is a homotopy pullback with derived $\Omega^j_B$.
\end{proof}

\begin{proof}[Proof of \autoref{thm:stable_model_structure}]
	It suffices to check six conditions spelled out in the list below, cf. \cite{hovey_model_cats}. Note that our definition of $I$-cell complex is not the general transfinite one. As a result, when we discuss smallness, we require that for each domain $K$ of a map of $I$, and each $I$-cell complex
	\[ X_0 \ra X_1 \ra \ldots X_{\infty} = \underset{n}\colim\, X_n, \]
	every map $K \to X_\infty$ factors through some $X_n$.\footnote{This simplification is also used in the notion of a ``compactly generated model category'' from \cite{mmss}. But we won't be able to use that notion directly, because the model structures for parametrized spectra that we are interested in only satisfy the ``cofibration hypothesis'' with respect to the category of sequences of unbased spaces (not retractive spaces), and this forgetful functor does not preserve colimits and limits.}
We refer to this as the \textbf{countable smallness condition}.
	\begin{enumerate}
		\item \textbf{$W$ is closed under 2-out-of-3 and retracts.} The stable equivalences are defined by a collection of functors $\R \pi_{n,b}(-)$, so this is automatic.
		
		\item \textbf{$I$ satisfies the countable smallness condition.} A map of parametrized spectra
		\[ F_k S^{n-1}_{+B} \to X \]
		is determined by a map $S^{n-1} \to X_k$ of unbased spaces over $B$. If $X$ is an $I$-cell complex with skeleta $X^{(m)}$, then each map $X^{(m)} \to X^{(m+1)}$ is a free $q$-cofibration, therefore a level $h$-cofibration by \autoref{prop:free_implies_level}, therefore a level closed inclusion. Since the sphere is compact and the quotients $X^{(m)}_k/X^{(m-1)}_k$ are weak Hausdorff, $S^{n-1}$ factors (uniquely in fact) through some skeleton $X^{(m)}_k$.
		
		\item \textbf{$J$ satisfies the countable smallness condition.} We know each map in $J$ is a free $q$-cofibration, so it is also a level closed inclusion. By the point above, any map from $F_k D^n_{+B}$ into a $J$-cell complex therefore factors through some finite stage. If instead we take a map from the pushout
		\[ \xymatrix @R=1.7em{
			F_{i+j} S^j \barsmash F_k S^{n-1}_{+B} \ar[r] \ar[d] & Cyl_{i,j} \barsmash F_k S^{n-1}_{+B} \\
			F_{i+j} S^j \barsmash F_k D^n_{+B}
		} \]
		into a $J$-cell complex, the map on $Cyl_{i,j} \barsmash F_k S^{n-1}_{+B}$ is determined by its restriction to each of the terms in \eqref{eq:expanded_cylinder}, smashed with $F_k S^{n-1}_{+B}$. Using \autoref{barsmash_free}, we rewrite each such term as a free spectrum on a compact retractive space. Therefore each term separately factors through some level $X^{(m)}$. Taking the maximum of the resulting values of $m$, these factorizations are unique, and therefore agree along the maps in \eqref{eq:expanded_cylinder}, so they give a factorization of all of $Cyl_{i,j} \barsmash F_k S^{n-1}_{+B}$ through some finite level of the colimit system. We then repeat the same argument with the rest of the pushout.
		
		\item \textbf{$J$-cell complexes are in $W$ $\cap$ $I$-cof.} We already know that $J$-cell complexes are in $I$-cof. By \autoref{j_equiv}, each map in $J$ is a stable equivalence, and also a level $h$-cofibration (using \autoref{prop:spectra_pushout_product}). Any coproduct of such maps is an equivalence by \autoref{prop:coproduct_colimit_stable_equivalences}, and any pushout is by \autoref{cor:spectra_left_proper}. Then a sequential colimit of such is an equivalence by \autoref{prop:coproduct_colimit_stable_equivalences} again.
		
		\item \textbf{$I$-inj $\subseteq$ $W$ $\cap$ $J$-inj.} Suppose $p$ is $I$-injective. Then it is an acyclic $q$-fibration on every level, hence a level equivalence, hence a stable equivalence. It is certainly at least a $q$-fibration on every level. Furthermore the version of \eqref{eq:fibration_of_spectra_means_this_is_a_pullback} with derived $\Omega^j_B$ has both verticals weak equivalences, hence it is a homotopy pullback square, so by the equivalence of \autoref{lem:stable_fibrations} we have $p \in W \cap J$-inj.
		
		\item \textbf{$W$ $\cap$ $J$-inj $\subseteq$ $I$-inj.} If $p \in W \cap J$-inj, then by \autoref{lem:stable_fibrations}, each $p_i$ is a $q$-fibration and the squares \eqref{eq:fibration_of_spectra_means_this_is_a_pullback} are homotopy pullback squares. We just need to prove that each $p_i$ is also a weak equivalence, using the fact that $p$ is a stable equivalence.
		
		Examine the strict fiber spectrum $F$ of the map $p$, which at each spectrum level is the pullback $B \times_{Y_n} X_n$. Because $p$ is a level $q$-fibration, this is equivalent to the homotopy fiber spectrum. By \autoref{lem:LES2}, $F$ is contractible in the sense that its derived homotopy groups vanish. Since $p$ is a level $q$-fibration, $F$ is level $q$-fibrant, so the homotopy groups of its fiber spectra $F_b$ are also zero. Pulling back the vertical maps of \eqref{eq:fibration_of_spectra_means_this_is_a_pullback} to a point $b \in B$ gives the homotopy pullback square
		\[ \xymatrix{
			(F_b)_i \ar[d] \ar[r]^-\sim & \Omega^j_B (F_b)_{i+j} \ar[d] \\
			\{b\} \ar@{=}[r] & \{b\},
		} \]
		hence each fiber spectrum $F_b$ is a weak $\Omega$-spectrum with vanishing homotopy groups. Therefore each of the levels $(F_b)_i$ is weakly contractible.
		
		Since $p_i\colon X_i \to Y_i$ is a $q$-fibration and its fiber $(F_b)_i$ over $i(b) \in Y_i$ is contractible, it is an equivalence on every component of $Y_i$ containing $i(B)$. Therefore $\R\Omega_B p_i$ is an equivalence on every component. By the homotopy pullback square \eqref{eq:fibration_of_spectra_means_this_is_a_pullback}, this implies $p_{i-1}$ is a weak equivalence, for every value of $i$.
	\end{enumerate}
	Left and right properness follow immediately from \autoref{cor:spectra_left_proper} and \autoref{cor:spectra_right_proper}.
\end{proof}

We say $X$ is \textbf{stably fibrant}\index{stably fibrant} when it is fibrant in the stable model structure. By \autoref{lem:stable_fibrations}, this means $X$ is level $q$-fibrant and $X_n \to \Omega_B X_{n+1}$ is an equivalence for every $n \geq 0$. The stable model structure gives us a stable fibrant replacement functor $R^{st}$.\index{fibrant replacement!$R^{st}$ for spectra}

\begin{lem}
	The suspension spectrum functor
	\[ \Sigma^\infty\colon \mc R(B) \to \Psp(B) \textup{ or } \Osp(B) \]
	is left Quillen. More generally the same is true for the free spectrum functor $F_n$.
\end{lem}

\begin{proof}
	We already know it's a left adjoint, so we just check it preserves generating cofibrations and acyclic cofibrations.
\end{proof}

\begin{prop}\label{prop:pre_equiv_to_orth}
	The forgetful functor $\mathbb U\colon \Osp(B) \to \Psp(B)$ is a right Quillen equivalence.
\end{prop}

\begin{proof}
	The left adjoint of $\mathbb U$ is left Kan extension from $\mathscr N$ to $\mathscr J$, which we denote $\mathbb P$. Concretely, $\mathbb P$ sends free sequential spectra to free orthogonal spectra, and preserves colimits, which determines the rest. See \cite{mmss}.
	
	Since \autoref{lem:stable_fibrations} applies equally well to $\Psp(B)$ and $\Osp(B)$, the forgetful functor $\mathbb U$ preserves fibrations. It also preserves level equivalences and the homotopy groups of the fiber spectra, so it preserves stable equivalences. Therefore $\mathbb U$ is a Quillen right adjoint.
	
	It remains to show this gives an equivalence on the homotopy category. This will follow if the unit map $X \to \mathbb U\mathbb PX$ is stable equivalence when $X$ is an $I$-cell complex. Since $\mathbb P$ is left Quillen, it preserves preserves stable equivalences of $q$-cofibrant spectra. Inducting up the cell complex structure, we reduce to the case where $X$ is of the form $F_k Z_{+B}$. Then we reduce this to the non-parametrized version of the same statement \cite[10.3]{mmss} by applying $P$, commuting it with the free spectrum construction, and then restricting to a fiber over $B$.
\end{proof}

As a first application of the stable model structure, we verify that homotopy cofiber sequences of spectra over $B$ give long exact sequences of maps into a third spectrum $Z$.
\begin{prop}
	If $f\colon X \to Y$ is a map of spectra over $B$ and $Z$ is a spectrum over $B$ then there is a long exact sequence
	\[ \xymatrix{
		\ldots \ar[r] & [\Sigma_B X,Z] \ar[r] & [\L C_B f,Z] \ar[r] & [Y,Z] \ar[r] & [X,Z] \ar[r] & \ldots
	}. \]
	where $[-,-]$ denotes maps in the homotopy category (stable equivalences inverted). It is natural in both $f$ and $Z$.
\end{prop}
Similar statements apply to double mapping cylinders, homotopy coequalizers, mapping telescopes, together with homotopy fiber sequences and the duals of these constructions. In other words, this $q$-model structure has all of the good properties desired in \cite{ms}.

\begin{proof}
	Applying cofibrant replacement to $X \to Y$ and stable fibrant replacement to $Z$, the above maps in the homotopy category can be described as point-set maps up to homotopy, and $\L C_B f$ can be described as just $C_B f$. Then the usual argument applies.
\end{proof}

\beforesubsection
\subsection{Derived base change and smash product}\label{sec:derived_bc_and_smash}\aftersubsection

In this section we use the stable model structure to prove that pullback, pushforward, and smash product are deformable with respect to the stable equivalences. We then combine this with the results of Section \ref{sec:spectra} to extend the range of deformability, proving \autoref{thm:intro_cof_fib}.

\begin{lem}\label{lem:pullback_stable_Quillen}
	$(f_! \adj f^*)$ is a Quillen pair for the stable model structure.
	If $f\colon A \to B$ is a weak equivalence then it is a Quillen equivalence.
\end{lem}

\begin{proof}
	From \autoref{lem:stable_fibrations}, we see that $f^*$ preserves fibrations and acyclic fibrations, hence it is right Quillen. In fact, it preserves all stable equivalences between level $q$-fibrant spectra, because such equivalences are characterized as the maps that are stable equivalences on each fiber spectrum, and $f^*$ preserves fibers. Therefore $\R f^*$ can be computed by taking a level fibrant replacement, instead of a stable fibrant replacement.
	
	With this change to the definition of $\R f^*$, the derived unit and counit maps for $(f_! \adj f^*)$ become level equivalences when $f$ is a weak equivalence, by \autoref{spaces_quillen_adj}. Therefore $(f_! \adj f^*)$ is a Quillen equivalence in this case.
\end{proof}

\begin{lem}\label{lem:bc_proper}
	If $f$ is an acyclic $q$-cofibration of spaces, then $f_!$ preserves all stable equivalences. Dually, if $f$ is an acyclic $q$-fibration then $f^*$ preserves all stable equivalences.
\end{lem}

\begin{proof}
	Suppose that $f\colon A \to B$ is a weak equivalence and a $q$-cofibration. By \autoref{prop:technical_cofibrations}, the functor $f_!$ preserves all level equivalences. By \autoref{lem:pullback_stable_Quillen}, $f_!$ preserves stable equivalences between freely $q$-cofibrant spectra. But every spectrum is level equivalent to a freely $q$-cofibrant spectrum by \autoref{prop:level_model_structure}, so $f_!$ preserves all stable equivalences. (See \autoref{prop:stably_derived} below for a more detailed version of this argument.)
	
	When $f\colon A \to B$ is a weak equivalence and a $q$-fibration, we use the dual of \autoref{prop:technical_cofibrations} to show that $f^*$ preserves all level equivalences. On the other hand, it clearly preserves fibers, so it preserves stable equivalences between level $q$-fibrant spectra. By the factorizations in \eqref{eq:stable_equivalence_decomposition}, we conclude that $f^*$ preserves all stable equivalences.
\end{proof}

\begin{rmk}
	There is an \textbf{integral model structure}\index{integral model structure!for spectra} on $\Osp$, in which the weak equivalences are the maps $X \to Y$ inducing a stable equivalence $X \simar \R f^*Y$ over a weak equivalence of base spaces $f\colon A \simar B$. This follows from \cite[Thm 3.0.12]{harpaz_prasma_grothendieck}; the previous two lemmas check the hypotheses needed for that theorem to apply. See also \cite{cagne2017bifibrations} for a more general framework, and \cite{hebestreit_sagave_schlichtkrull} for a similar model structure for parametrized symmetric spectra.
	
	As in \autoref{integral_spaces}, there is a second integral model structure on $\Osp$, in which the map of base spaces is required to be an isomorphism, not just a weak equivalence. The homotopy category of this stricter model structure is given by taking homotopy categories fiberwise, by \autoref{prop:deform_bifibration}, so it is more closely related to the individual homotopy categories $\ho\Osp(B)$.
\end{rmk}

\begin{lem}\label{lem:other_stable_adjunction}
	If $f$ is a fiber bundle with base and fiber both cell complexes, $f^*$ is a Quillen left adjoint. If in addition $f$ is a weak equivalence, $f^*$ is a left Quillen equivalence.
\end{lem}

\begin{proof}
	We construct the right adjoint $f_*$ in spectra by applying the space-level $f_*$ on each level, and using the canonical commutation with $\Omega_B(-)$. Since the fiber of $f$ is a cell complex, $f^*$ sends generating cofibrations to cofibrations, and similarly for generating acyclic cofibrations, therefore $f^*$ is left Quillen. As in \autoref{ex:other_adjunction_spaces}, $f^*$ is homotopical in this case, so its left- and right-derived functors are equivalent, and so the Quillen equivalence part follows from \autoref{lem:pullback_stable_Quillen}.
\end{proof}

\begin{lem}\label{lem:stable_smash_is_left_quillen}
	Each of the external smash product functors
	\begin{align*}
	\mc R(A) \times \Psp(B) &\to \Psp(A \times B), \\
	\qquad \mc R(A) \times \Osp(B) &\to \Osp(A \times B), \\
	\qquad \Osp(A) \times \Osp(B) &\to \Osp(A \times B)
	\end{align*}
	is a left Quillen bifunctor, with respect to the stable model structure on spectra, and the Quillen model structure on spaces.
\end{lem}

\begin{proof}
	This is a straightforward exercise using \autoref{barsmash_free}.
\end{proof}

\begin{rmk}
	Taking $A = *$, this implies that $\Psp(B)$ and $\Osp(B)$ are simplicial model categories.
\end{rmk}

We are now ready to finish the proof of \autoref{thm:intro_cof_fib} from the introduction.
\begin{thm}\label{prop:stably_derived}\hfill
	\vspace{-1em}
	
	\begin{itemize}
		\item $f^*$ preserves all stable equivalences between level $q$-fibrant spectra.
		\item $f_!$ preserves all stable equivalences between level $h$-cofibrant spectra.
		\item $\barsmash$ preserves all stable equivalences between freely $h$-cofibrant spectra.
	\end{itemize}
\end{thm}

\begin{proof}
	This is a simple combination of the stable model structure with the technical work in \autoref{sec:spectra}. The first point was already proven in \autoref{lem:pullback_stable_Quillen}.
	
	For the second point, suppose $X \to Y$ is a stable equivalence of level $h$-cofibrant spectra. Using \autoref{prop:level_model_structure}, we can replace $X$ and $Y$ by level-equivalent, freely $q$-cofibrant spectra $QX$ and $QY$, giving a commuting square
	\[ \xymatrix @R=1.5em{
		f_!(QX) \ar[d]^-\sim \ar[r]^-\sim & f_!X \ar[d] \\
		f_!(QY) \ar[r]^-\sim & f_!Y.
	} \]
	The horizontal maps are level equivalences since $f_!$ preserves level equivalences of level $h$-cofibrant spectra. The left-hand map is a stable equivalence since $f_!$ is left Quillen (\autoref{lem:pullback_stable_Quillen}) and $QX$ and $QY$ are freely $q$-cofibrant. It follows that $f_!X \to f_!Y$ is a stable equivalence.
	
	The third point is argued the same way, by fitting a smash product of two stable equivalences $X \to Y$ and $W \to Z$ of freely $h$-cofibrant spectra into a commuting square
	\[ \xymatrix @R=1.5em{
		QX \barsmash QW \ar[d]^-\sim \ar[r]^-\sim & X \barsmash W \ar[d] \\
		QY \barsmash QZ \ar[r]^-\sim & Y \barsmash Z.
	} \]
	The map on the left is a stable equivalence by \autoref{lem:stable_smash_is_left_quillen}, and the horizontal maps are level equivalences by \autoref{prop:spectra_pushout_product}. It follows that the map on the right is a stable equivalence.
\end{proof}

\begin{rmk}\label{cor:level_derived_equals_stable_derived}
	Each of the functors in \autoref{prop:stably_derived} preserves both stable and level equivalences, after deforming the input by a level equivalence. It follows that the derived functors $\L f_!$, $\R f^*$, and $\barsmash^{\L}$ are the same whether we use level equivalences or stable equivalences to derive them. The same is not true of $f_*$ however, at least if the fiber of $f$ is not a finite complex.
\end{rmk}

\beforesubsection
\subsection{Parametrized homology and cohomology}\aftersubsection

We conclude this section with a brief discussion of homology and cohomology with local coefficients in a spectrum.

If $B$ is a space, $\mc E$ is a parametrized sequential spectrum or orthogonal spectrum over $B$, and $X \to B$ is an unbased space over $B$, we define the twisted homology and cohomology spectra of $X$ with coefficients in $\mc E$ as follows.
\begin{align*}
	H_\bullet(X;\mc E) := (\L r_!)(\R p^*) \mc E, \\
	H^\bullet(X;\mc E) := (\R r_*)(\R p^*) \mc E.
\end{align*}
Here $r\colon X \to *$ and $p\colon X \to B$ are the projections. If $\mc E$ is freely $f$-cofibrant and level $h$-fibrant then the base-change operations are already derived, so we get the simpler formula
\begin{align*}
	H_\bullet(X;\mc E) \simeq r_!p^* \mc E \\
	H^\bullet(X;\mc E) \simeq r_*p^* \mc E.
\end{align*}
In other words, we pull back $\mc E$ to $X$, then either quotient out the basepoint (homology) or take sections (cohomology). Using \eqref{half_smash_internal} and \eqref{half_map_sections}, we can also describe these spectra as
\begin{align*}
	H_\bullet(X;\mc E) &\simeq (\L r_!)(\R\Delta_B^*) (X_{+B} \barsmash^\L \mc E), \\
	H^\bullet(X;\mc E) &\simeq \R \barF_B(X_{+B},\mc E),
\end{align*}
where $r\colon B \to {*}$ and $\Delta_B\colon B \to B \times B$ are 0-fold and 2-fold diagonal maps.

When $\mc E = r^*E = B \times E$ is a trivial bundle of spectra, we get the classical definitions of homology and cohomology spectra:
		\[ H_\bullet(X;E) \simeq X_+ \sma^{\L} E, \qquad H^\bullet(X;E) \simeq \R F(X_+,E). \]

Another common case is $X = B$, where we get
		\[ H_\bullet(B;\mc E) \simeq (\L r_!)\mc E, \qquad H^\bullet(B;\mc E) \simeq (\R r_*)(\mc E). \]

If $\mc A$ is a bundle of abelian groups and $\mc E = H\mc A$ is the associated bundle of Eilenberg-Maclane spectra, then the above definition of $H_*(X;H\mc A)$ is isomorphic to the usual notion of twisted homology $H_*(X;\mc A)$, and similarly for cohomology.

It is possible to characterize these theories by a variant of the Eilenberg-Steenrod axioms for parametrized spaces over $B$. See \cite[20.1]{ms} for more details. In particular, the twisted $K$-theory of \cite{atiyah_segal} is represented by a parametrized spectrum, as discussed in detail in \cite{hebestreit_sagave_schlichtkrull}.

\section{Structures on parametrized spectra}\label{sec:structures}

In this section we give new, more direct constructions of several structures that exist on the homotopy category of parametrized spectra. We first show that the external smash product $\barsmash$ makes the homotopy category $\ho\Osp$ into a symmetric monoidal bifibration (SMBF). A related result is that the smash product over $B$, $\sma_B$, makes the fiber category $\ho\Osp(B)$ symmetric monoidal and the pullback functors $f^*$ strong symmetric monoidal. We then turn to the bicategory $\Ex$ and its fiberwise version $\Ex_{B}$, as discussed in the introduction.

\beforesubsection
\subsection{The SMBF $\ho\Osp$}\aftersubsection

Recall from \autoref{all_spectra} that $\Osp$ is the category of all orthogonal spectra over all base spaces. We form $\ho\Osp$ by inverting the stable equivalences in each fiber category, see \autoref{sec:deriving_SMBFs}.

\begin{thm}\label{thm:spectra_SMBF}
	The homotopy category $\ho\Osp$ forms a symmetric monoidal bifibration (SMBF) over $\cat{Top}$, with Beck-Chevalley squares the homotopy pullback squares of spaces.
\end{thm}\index{SMBF!of parametrized spectra $\ho\Osp$}

\begin{proof}
	We work in (CG). The point-set category $\Osp$ is a symmetric monoidal bifibration with Beck-Chevalley for the strict pullback squares, using \autoref{prop:beck_chevalley_spaces} and \autoref{prop:spectra_external_smash_and_base_change}. To pass to the homotopy category $\ho\Osp$, we check the seven conditions from \autoref{prop:deform_SMBF}.
	\begin{enumerate}
		\item $\barsmash$ preserves freely $h$-cofibrant spectra, and stable equivalences between them, by \autoref{prop:spectra_pushout_product} and \autoref{prop:stably_derived}.
		\item $F_0 S^0$ is freely $h$-cofibrant.
		\item The pushforwards $f_!$ are coherently left-deformable, using for instance the level $h$-cofibrant spectra, by \autoref{prop:stably_derived}.
		\item The pullbacks $f^*$ are coherently right-deformable, using level $q$-fibrant spectra.
		\item For a pullback square along an $h$-fibration, the Beck-Chevalley maps are coherently deformable using \autoref{lem:spectrum_f_preserves} and \autoref{prop:stably_derived}, and the freely $f$-cofibrant, level $h$-fibrant spectra as inputs.
		\item $\barsmash$ and all of the pushforwards $f_!$ are coherently deformable, using for instance the freely $f$-cofibrant spectra, by \autoref{prop:stably_derived}.
		\item $\barsmash$ and all of the pullbacks $f^*$ are coherently deformable using the freely $f$-cofibrant, level $h$-fibrant spectra. This uses \autoref{lem:spectrum_f_preserves}, \autoref{prop:stably_derived}, and especially \autoref{prop:spectra_pushout_product}(3).
	\end{enumerate}
	As in \autoref{expand_beck_chevalley}, since $\L f_!$ is an equivalence when $f$ is a weak equivalence, and every homotopy pullback square is equivalent to a strict pullback along an $h$-fibration, we therefore get Beck-Chevalley for every homotopy pullback square of spaces.
\end{proof}

\begin{rmk}
	This result has been known for some time, see e.g. \cite{shulman_framed_monoidal}, but as pointed out in the proof of \cite[Thm 9.9]{mp1}, it is difficult to piece together the argument without somehow using the monoidal fibrant replacement functor $P$ from \autoref{prop:px_properties}.
\end{rmk}

\autoref{thm:spectra_SMBF} describes a compatibility between the operations $\barsmash^\L$, $\L f_!$, and $\R f^*$ on the homotopy category. They be interchanged by the three canonical commutation isomorphisms
\begin{align*}
\L f_!X \barsmash^{\L} \L g_!Y &\simeq \L(f \times g)_!(X \barsmash^{\L} Y), \\
\R f^*X \barsmash^{\L} \R g^*Y &\simeq \R(f \times g)^*(X \barsmash^{\L} Y), \\
\L g_! \circ \R f^* &\simeq \R p^* \circ \L q_!,
\end{align*}
as in \cite[\S 13.7]{ms}. Note however that \autoref{thm:spectra_SMBF} is somewhat more useful than earlier constructions because it tells us \emph{which map} is the isomorphism. In particular, each of these isomorphisms arises by a universal property in the category $\ho\Osp$. Alternatively, by the proofs of \autoref{prop:deform_bifibration} and \autoref{prop:deform_SMBF}, the same isomorphisms arise by deforming the unique point-set isomorphism, using \autoref{prop:passing_natural_trans_to_derived_functors}.

The same argument shows that the category of sequential spectra $\ho \Psp$ is a bifibration with the same Beck-Chevalley squares. We also get isomorphisms as above that commute pullback and pushforward with the operation that smashes a spectrum with a space.

\begin{ex}
	As in \cite[13.7.4]{ms}, we deduce that suspension spectrum commutes with pullback and pushforward,
	\[ f_!\Sigma^\infty_A X \simeq \Sigma^\infty_B f_!X, \qquad \Sigma^\infty_A f^*Y \simeq f^*\Sigma^\infty_B Y. \]
	It is straightforward to check we have such isomorphisms on the point-set level, and they are unique. Since the functors are coherently deformable (\autoref{prop:passing_natural_trans_to_derived_functors}), this passes to a canonical equivalence of derived functors.
\end{ex}

Let $\Osp^{c} \subseteq \Osp$ be the subcategory of freely $f$-cofibrant spectra. The induced map $\ho \Osp^{c} \to \ho \Osp$ is a fiberwise equivalence of categories over $\cat{Top}$. The proof of \autoref{thm:spectra_SMBF}, in (CG) or in (CGWH), shows that $\ho\Osp^{c}$ is an SMBF as well.

\begin{rmk}\label{cofibrant_SMBF}
The category $\Osp^{c}$ is perhaps the most convenient category of parametrized spectra. In this setting, the smash product and pushforward do not need to be derived, since everything is already cofibrant, and the pullback can be derived using $P$:
\[ \barsmash^{\L} = \barsmash^{\R} = \barsmash, \qquad \L f_! = \R f_! = f_!, \qquad \R f^* = f^* P. \]
This brings the homotopy category suprisingly close to the point-set category. Almost every operation preserves weak equivalences, and the only one that doesn't (pullback) can be derived in a simple way, by a functor $P$ that commutes with smash products, mapping cones, and so on.

Another nice feature of this setting is that we can think of $\barsmash$ and $f_!$ as \emph{right}-deformable, since they are homotopical, and homotopical functors are both left- and right-deformable. By \autoref{lem:spectrum_f_preserves} and \autoref{prop:spectra_pushout_product}, every composite of $\barsmash$, $f_!$, and $f^*$ is coherently right-deformable in the sense \autoref{coherently_left_deformable}, so long as we only take pushforwards $f_!$ along $h$-fibrations.
\end{rmk}

\beforesubsection
\subsection{The symmetric monoidal category $\ho \Osp(B)$}\aftersubsection

Our next task is to treat the internal smash product $\sma_B$ on orthogonal spectra over $B$. This is conceptually more complicated because $\sma_B$ is not left- or right-deformable. It is only a composite of coherently deformable functors. So, its derived functor
\[ \sma_B^{\M} := (\R\Delta_B^*)(\barsmash^{\L}) \]
is only unique if we stipulate that it must arise as a composite of the derived functors for $\barsmash$ and $\Delta_B^*$.

We extend this to a symmetric monoidal structure on the homotopy category, with unit $\Sph_B := \Sph \times B \simeq \R r_B^* \Sph$, in any of the following four ways.
\begin{itemize}
	\item[(1)] The functors obtained by iterating $\sma_B = \Delta_B^* \circ \barsmash$ and $r_B^*$ are coherently deformable, using \autoref{prop:stably_derived}. Therefore the recipe of \autoref{prop:passing_natural_trans_to_derived_functors} passes the symmetric monoidal structure of $\Osp(B)$ up to the homotopy category.
	\item[(2)] Restrict to the equivalent subcategory $\ho\Osp^{c}(B)$ of freely $f$-cofibrant spectra. By \autoref{prop:stably_derived}, $\sma_B$ preserves this subcategory and is right-deformable using $P$. Therefore the recipe of \autoref{prop:passing_natural_trans_to_derived_functors} passes the symmetric monoidal structure of $\Osp(B)^c$ up to the homotopy category.
	\item[(3)] Restrict to the equivalent subcategory $\ho\Osp^{cf}(B)$ of freely $f$-cofibrant, level $h$-fibrant spectra. By \autoref{prop:stably_derived}, $\sma_B$ is homotopical here and sends this subcategory into itself. Therefore the symmetric monoidal structure on $\Osp^{cf}(B)$ that arose from the rigidity theorem, passes immediately to the homotopy category.
	\item[(4)] Formally pull back the symmetric monoidal structure on $\ho \Osp$ along the inclusion $\ho \Osp(B) \subseteq \ho \Osp$, using \autoref{pullback_SMBF}.
\end{itemize}

\begin{prop}\label{lem:two_smashes_over_b_isomorphic}
	These four recipes give canonically isomorphic symmetric monoidal structures on $\ho\Osp(B)$.
\end{prop}

\begin{proof}
	An isomorphism of symmetric monoidal structures is determined by what it does on an equivalent subcategory, so we will restrict to $\ho\Osp^{cf}(B)$, the freely $f$-cofibrant level $h$-fibrant spectra. The first, second, and fourth recipes use $\sma_B^\M$ and $\R r_B^* \Sph$ here, while the third recipe uses $\sma_B$ and $r_B^* \Sph$. These admit canonical isomorphisms because of the uniqueness of right-derived and left-derived functors. The work is to check that the associator, unitor, and symmetry isomorphism in the three recipes commute along these isomorphisms. For the first, second, and third recipes, this is true just by examination of \autoref{prop:passing_natural_trans_to_derived_functors}. For the fourth recipe, on this subcategory we may remove the $\L$ and $\R$ decorations from the functors, and then we are left comparing a point-set associator map pulled back from $\ho\Osp^{cf}$, and the point-set associator constructed internally to $\ho\Osp^{cf}(B)$. These must coincide by the rigidity theorem, \autoref{thm:spectra_rigidity}. The argument for the unitor and symmetry isomorphism proceeds the same way.
\end{proof}

\begin{ex}
	Suppose $\xi = V - n$ and $\omega = W - m$ are virtual bundles over $B$. Using \autoref{barsmash_free}, we get a the point-set isomorphism
\[ \Th_B(\xi) \sma_B \Th_B(\omega) := F_n \Th_B(V) \sma_B F_m \Th_B(W) \cong F_{n+m} \Th_B(V \times_B W) = \Th_B(\xi \oplus \omega). \]
Because everything is in $\Osp^{cf}(B)$, by recipe (3) above, we get the same isomorphism on the homotopy category. This is one example of how an obvious point-set isomorphism can be lifted quickly to the homotopy category, when we use convenient notions of cofibrant and fibrant spectra.
\end{ex}

For $f\colon A \to B$, the strict pullback functor $f^*\colon \Osp(B)\to\Osp(A)$ has a unique symmetric monoidal structure, by \autoref{thm:spectra_rigidity}. Each of the above four recipes suggests a way of passing this to a symmetric monoidal structure on the derived pullback $\R f^*$: coherent deformability, restricting to the freely $f$-cofibrant or the freely $f$-cofibrant and level $h$-fibrant spectra, or using the universal property of cartesian arrows in the fibration $\ho\Osp$.
\begin{prop}\label{prop:unique_sm_structure_on_pullback}
	These give identical symmetric monoidal structures on $$\R f^*\colon \ho\Osp(B) \to \ho\Osp(A).$$
\end{prop}

\begin{proof}
	We just have to show they give the same maps in the homotopy category
	\[ \R f^*(X) \sma_B^{\M} \R f^*(Y) \sim \R f^*(X \sma_B^{\M} Y), \qquad \R f^*(r_B^* \Sph) \sim r_A^* \Sph. \]
	Technically, the third recipe produces such maps without the $\M$ and $\R$ decorations, so we use the canonical isomorphisms to compare those maps to the maps furnished by the other three recipes. As before, it suffices to restrict to freely $f$-cofibrant, level $h$-fibrant spectra, and then the first and second recipes give the same map as the third, essentially by their definition. The proof that the fourth recipe agrees is also similar to before: we first remove the $\L$ from $\barsmash$, then remove the $\R$ from the pullbacks, then use rigidity to show that the remaining point-set isomorphisms agree.
\end{proof}

\subsection{The SMBF $\ho \Osp_{(B)}$}\label{sec:Osp_B}

Next, we extend the symmetric monoidal structure on $\ho \Osp(B)$ to the larger homotopy category $\ho \Osp_{(B)}$ of all spectra over all spaces over $B$.

\begin{df}\label{ospb}
We define the point-set category $\Osp_{(B)}$ as the pullback of categories
\[ \xymatrix @R=1.7em{
	\Osp_{(B)} \ar[d] \ar[r] & \Osp \ar[d] \\
	\cat{Fib}_B \ar[r] & \cat{Top}
} \]
where $\cat{Fib}_B$ is the category of $h$-fibrations $p\colon E \to B$, mapping to $\cat{Top}$ by retaining the total space and forgetting the fibration. So an object of $\Osp_{(B)}$ is a pair $(p\colon E \to B, X \in \Osp(E))$. The morphisms are computed by forgetting $p$ and taking morphisms in $\Osp$.
\end{df}

By \autoref{pullback_SMBF}, the point-set category $\Osp_{(B)}$ has a symmetric monoidal structure that makes it into an SMBF. We call this product the \textbf{external smash product rel $B$}\index{external!smash rel $B$ $X \barsma{B} Y$}. For $X \in \Osp(D)$ and $Y \in \Osp(E)$, it is defined as
\[ X \barsma{B} Y := \Delta_{D,E}^*(X \barsmash Y), \qquad \Delta_{D,E}\colon D \times_B E \to D \times E. \]
The original symmetric monoidal category $\Osp(B)$ sits inside $\Osp_{(B)}$ as the fiber category over the object $\id_B\colon B \to B$. So this structure extends the one discussed in the previous subsection.

We pass to the homotopy category $\ho\Osp_{(B)}$ by inverting the stable equivalences. The result is also the pullback
\[ \xymatrix @R=1.7em{
	\ho \Osp_{(B)} \ar[d] \ar[r] & \ho \Osp \ar[d] \\
	\cat{Fib}_B \ar[r] & \cat{Top}
} \]
by comparing fiber categories and using \autoref{prop:deform_bifibration}. This homotopy category receives four symmetric monoidal structures, by the four recipes from \autoref{lem:two_smashes_over_b_isomorphic}: coherent deformability of $\barsma{B}$ and its iterates on the whole category, or on the subcategory of freely $f$-cofibrant spectra, passing to freely $f$-cofibrant level $h$-fibrant spectra where $\barsma{B}$ is homotopical, or formally pulling back the symmetric monoidal structure with product $\barsmash^{\L}$ on $\ho\Osp$.

\begin{prop}
	These four symmetric monoidal structures on $\ho \Osp_{(B)}$ are canonically isomorphic, and make $\ho \Osp_{(B)}$ into a symmetric monoidal bifibration (SMBF) over $\cat{Fib}_B$.
\end{prop}

\begin{proof}
	For the first claim, the proof of \autoref{lem:two_smashes_over_b_isomorphic} applies verbatim. The second claim follows by applying \autoref{pullback_SMBF} to the fourth recipe (formally pulling back the symmetric monoidal structure from $\ho\Osp$).
\end{proof}

\begin{rmk}
	If we enlarge the base category of $\Osp_{(B)}$ to be all spaces with maps to $B$, not just fibrations, then by the same argument we would get a symmetric monoidal structure on $\Osp_{(B)}$ and the homotopy category $\ho\Osp_{(B)}$. However it fails to be an SMBF, because the smash product does not preserve all co-cartesian arrows in the homotopy category.
\end{rmk}

\beforesubsection
\subsection{The bicategory $\Ex$}\aftersubsection

Recall that a \textbf{bicategory}\index{bicategory} is a collection of objects or 0-cells, together with morphism categories between them, whose objects are called 1-cells and whose morphisms are called 2-cells. These compose in a way that is associative and unital up to coherent natural isomorphism.
\begin{ex}
	The bimodule bicategory $\mathscr Bimod$ has
\begin{itemize}
	\item 0-cells are rings $A$, $B$, $C$, ...
	\item 1-cells from $A$ to $B$ are $(A,B)$-bimodules $M$,
	\item 2-cells $M \to N$ are homomorphisms of $(A,B)$-bimodules, and
	\item horizontal composition of the $(A,B)$-bimodule $M$ with the $(B,C)$-bimodule $N$ given by the tensor product over $B$, $M \otimes_B N$, which is an $(A,C)$-bimodule.
\end{itemize}
\end{ex}

A \textbf{bicategory with shadow}\index{shadowed bicategory}\index{bicategory with shadow}, or shadowed bicategory \cite{ponto_asterisque}, is a bicategory that also allows circular compositions of 1-cells. These are also sometimes called trace theories, after \cite{kaledin_traces}. To define a bicategory with shadow $\mathscr B$, we need the following data.
\begin{itemize}
	\item A collection of 0-cells $\ob \mathscr B$. We typically label these $A$, $B$, $C$, ...
	\item For each pair of 0-cells $A$, $B$, a category $\mathscr B(A,B)$ whose objects are called 1-cells $X$, $Y$, ... and whose morphisms are called 2-cells $f$, $g$, ...
	\item A category $\mathscr B_{\shad{}}$ called the ``shadow category.''
	\item For every triple of 0-cells $A$, $B$, $C$ a horizontal composition functor $\odot\colon \mathscr B(A,B) \times \mathscr B(B,C) \to \mathscr B(A,C).$
	\item For every 0-cell $A$ a functor $\shad{}\colon \mathscr B(A,A) \to \mathscr B_{\shad{}}$ and a 1-cell $U_A \in \ob\mathscr B(A,A)$.
	\item Natural associator, left unitor, and right unitor isomorphisms
	\[
	\alpha\colon (X \odot Y) \odot Z \cong X \odot (Y \odot Z),
	\qquad
	\lambda\colon U_A \odot X \cong X,
	\qquad
	\rho\colon X \odot U_B \cong X
	\]
	and a natural ``rotator'' isomorphism
	\[ \theta\colon \shad{X \odot Y} \cong \shad{Y \odot X}. \]
\end{itemize}
This data must also satisfy a coherence condition. If we start with any expression for an $n$-fold product of distinct 1-cells, either arranged along a line or a circle, then if we re-arrange this expression using the isomorphisms $\alpha$, $\lambda$, $\rho$, and $\theta$, and eventually come back to the same expression, the composite isomorphism must be the identity.

As in the case of symmetric monoidal categories, it suffices to check this condition for four particular strings of the above isomorphisms. See \cite{benabou} and \cite[4.4.1]{ponto_asterisque} for the four conditions, and \cite[Thm 4.20]{mp3} for the proof that these four conditions imply all of the remaining coherences.

\begin{rmk}\label{rmk:left_deformable_bicategory}
	The bicategories we encounter often come in pairs, a ``point-set'' bicategory $\mathscr B$ and a ``homotopy'' bicategory $\ho\mathscr B$ obtained by inverting certain 2-cells and left-deriving all the bicategory operations. Using the theory from \autoref{sec:composing_comparing} (or simply citing \cite[22.11]{shulman2006homotopy}), we can form $\ho\mathscr B$ from $\mathscr B$ any time there exists a left retraction $Q$ on each category $\mathscr B(A,B)$, landing in a designated subcategory of cofibrant 1-cells, such that
	\begin{itemize}
		\item $\odot$ preserves cofibrant objects,
		\item $\odot$ and $\shad{}$ preserve weak equivalences between cofibrant objects,
		\item $QU_A \odot QX \to U_A \odot QX$ is always a weak equivalence, and
		\item $QX \odot QU_B \to QX \odot U_B$ is always a weak equivalence.
	\end{itemize}
	We say that $\mathscr B$ is a \textbf{left-deformable bicategory}\index{left-deformable!bicategory} when this happens. We can also dualize this to get the notion of a \textbf{right-deformable bicategory}.
\end{rmk}

We may now construct the bicategory of parametrized spectra. We start by defining a bicategory $\mc R/\cat{Top}$ whose 0-cells are topological spaces. The 1-cells from $A$ to $B$ are retractive spaces $X$ over $A \times B$, and 2-cells are maps of retractive spaces. The product of $X \in \mc R(A \times B)$ and $Y \in \mc R(B \times C)$ is defined by the formula
		\[ X \odot Y := (r_B)_!(\Delta_B)^*(X \barsmash Y), \]
as shown in the following diagram.
		\[ \xymatrix @R=1.7em @C=4em{
			X \odot Y \ar[d] & \ar[l] \Delta_B^*(X \barsmash Y) \ar[d] \ar[r] & X \barsmash Y \ar[d] \\
			A \times C & \ar[l]_-{r_B} A \times B \times C \ar[r]^-{\Delta_B} & A \times B \times B \times C
		} \]
We sometimes call this $X \odot^B Y$ if the choice of $B$ is not clear. 
Intuitively, the product is a space over $A \times C$ whose fiber over $(a,c)$ is the sum of $X_{a,b} \sma Y_{b,c}$ over all $b \in B$. Another intuition is that away from the basepoint sections, this is just the fiber product $X \times_B Y$.\index{bicategory!of retractive spaces $\mc R/\cat{Top}$}
The shadow and unit are
\[ \shad{X} := (r_A)_!(\Delta_A)^*X, \qquad U_A := (\Delta_A)_!(r_A)^* S^0 \cong A_{+(A \times A)}. \]

We have the rigidity theorem \autoref{prop:spaces_rigidity} waiting on standby, so as soon as we show that isomorphisms $\alpha$, $\lambda$, $\rho$, and $\theta$ exist, we will know for free that they are unique and coherent. To show these isomorphisms exist, we decompose each of the operations $\odot$, $\shad{}$, and $U_A$ into three pieces. Then we compose several smaller isomorphisms from the SMBF structure on $\mc R$, one for each small square or triangular region in the diagrams below.\footnote{These diagrams were typeset by Kate Ponto.} In fact, this procedure works for any symmetric monoidal bifibration (SMBF), not just $\mc R \to \cat{Top}$.

\begin{thm}\label{thm:SMBF_to_bicategory}\cite{ponto_shulman_indexed}
	For any SMBF $(\mc C,\boxtimes,I)$ over $\bS$, there is a shadowed bicategory $\mc C/\bS$ whose 0-cells are the objects of $\bS$, 1- and 2-cells the objects and morphisms in the fiber category $\mc C^{A \times B}$, shadow category $\mc C^*$, three operations $\odot$, $\shad{}$, $U_A$ defined by
	\[ \begin{array}{rclcrcl}
	\mc C^{A \times B} \times \mc C^{B \times C} & \overset\odot\to & \mc C^{A \times C}
	&& X \odot Y &=& (r_B)_!(\Delta_B)^*(X \boxtimes Y) \\
	\mc C^{A \times A} & \overset{\shad{}}\to & \mc C^{*}
	&& \shad{X} &=& (r_A)_!(\Delta_A)^*X \\
	{*} & \overset{U_A}\to & \mc C^{A \times A}
	&& U_A &=& (\Delta_A)_!(r_A)^* I
	\end{array} \]
	and isomorphisms $\alpha$, $\lambda$, $\rho$, $\theta$ given by \autoref{fig:assoc_indexed}, \autoref{fig:unit_indexed}, and \autoref{fig:shadow_indexed}.
\end{thm}

\begin{figure}[h]
	{
		\begin{tikzcd}[column sep=30pt,row sep=15pt]
			{\begin{pmatrix} \mathscr{C}^{A\times B}
					\\
					\mathscr{C}^{B\times C}
					\\
					\mathscr{C}^{C\times D}
				\end{pmatrix}
			}
			\ar[d,"{\boxtimes \times 1}"]
			\ar[r,"{1 \times \boxtimes}"]
			&
			{\begin{pmatrix} \mathscr{C}^{A\times B}
					\\
					\mathscr{C}^{B\times C\times C\times D}
				\end{pmatrix}
			}\ar[r,"{1 \times \Delta_C^*}"]\ar[d,"\boxtimes"]
			&
			{\begin{pmatrix} \mathscr{C}^{A\times B}
					\\
					\mathscr{C}^{B\times C\times D}
				\end{pmatrix}
			}\ar[r,"{1 \times (r_C)_!}"]\ar[d,"\boxtimes"]
			&
			{\begin{pmatrix} \mathscr{C}^{A\times B}
					\\
					\mathscr{C}^{B\times D}
				\end{pmatrix}
			}\ar[d,"\boxtimes"]
			\\
			{\begin{pmatrix} \mathscr{C}^{A\times B\times B\times C}
					\\
					\mathscr{C}^{C\times D}
				\end{pmatrix}
			}\ar[r,"\boxtimes"]\ar[d,"{\Delta_B^* \times 1}"]
			&\mathscr{C}^{A\times B\times B\times C\times C\times D}\ar[r,"\Delta_C^*"]\ar[d,"\Delta_B^*"]
			&\mathscr{C}^{A\times B\times B\times C\times D}\ar[r,"(r_C)_!"]\ar[d,"\Delta_B^*"]
			&\mathscr{C}^{A\times B\times B\times D}\ar[d,"\Delta_B^*"]
			\\
			{\begin{pmatrix} \mathscr{C}^{A\times B\times C}
					\\
					\mathscr{C}^{C\times D}
				\end{pmatrix}
			}\ar[r,"\boxtimes"]\ar[d,"{(r_B)_! \times 1}"]
			&\mathscr{C}^{A\times B\times C\times C\times D}\ar[r,"\Delta_C^*"]\ar[d,"(r_B)_!"]
			&\mathscr{C}^{A\times B\times  C\times D}\ar[r,"(r_C)_!"]\ar[d,"(r_B)_!"]
			&\mathscr{C}^{A\times B\times D}\ar[d,"(r_B)_!"]
			\\
			{
				\begin{pmatrix} \mathscr{C}^{A\times C}
					\\
					\mathscr{C}^{C\times D}
				\end{pmatrix}
			}\ar[r,"\boxtimes"]
			&
			\mathscr{C}^{A\times C\times C\times D}\ar[r,"\Delta_C^*"]
			&
			\mathscr{C}^{A\times C\times D}\ar[r,"(r_C)_!"]
			&
			\mathscr{C}^{A\times D}
		\end{tikzcd}
	}
	\caption{Associator (products of categories are written vertically in parentheses)}\label{fig:assoc_indexed}
\end{figure}
\begin{figure}[h]
	\begin{tikzcd}[column sep=50pt,row sep=30pt]
		\sC^{A\times B}\times \ast \ar[d,"1 \times I"]\ar[dr,"\cong"] \\
		\sC^{A\times B}\times \sC^\ast \ar[d,"1 \times r_B^*"]\ar[r,"\boxtimes"]
		&\sC^{A\times B}\ar[d,"r_B^*"]\ar[dr,"\cong"]
		\\
		\sC^{A\times B}\times \sC^B\ar[d,"1 \times (\Delta_B)_!"]\ar[r,"\boxtimes"]
		&\sC^{A\times B\times B}\ar[d,"(\Delta_B)_!"]\ar[r,"\Delta_B^*"]
		&\sC^{A\times B}\ar[d,"(\Delta_B)_!"]\ar[dr,"\cong"]
		\\ 
		\sC^{A\times B}\times \sC^{B\times B}\ar[r,"\boxtimes"]
		&
		\sC^{A\times B\times B\times B}\ar[r,"\Delta_B^*"]
		&
		\sC^{A\times B\times B}\ar[r,"(r_B)_!"]
		&\sC^{A\times B}
	\end{tikzcd}
	\caption{Unitor}\label{fig:unit_indexed}
\end{figure}
\begin{figure}[h]
	\begin{tikzcd}[column sep=50pt,row sep=30pt]
		{\mathscr{C}^{A\times B} \times \mathscr{C}^{B\times A}}
		\ar[r,"{\boxtimes}"]
		\ar[d,"{\boxtimes}"]
		&\mathscr{C}^{B\times A\times A\times B}
		\ar[ld,"\simeq"]
		\ar[r,"\Delta_A^*"]
		\ar[d,"\simeq"]
		&\mathscr{C}^{B\times A\times B}
		\ar[r,"(r_A)_!"]
		\ar[d,"\simeq"]
		&\mathscr{C}^{B\times B}\ar[d,"="]
		\\
		\mathscr{C}^{A\times B \times B\times A}
		\ar[d,"\Delta_B^*"]
		\ar[r,"\simeq"]
		&\mathscr{C}^{A\times A \times B\times B}
		\ar[r,"\Delta_A^*"]
		\ar[d,"\Delta_B^*"]
		&\mathscr{C}^{A\times B\times B}
		\ar[r,"(r_A)_!"]
		\ar[d,"\Delta_B^*"]
		&\mathscr{C}^{B\times B}
		\ar[d,"\Delta_B^*"]
		\\
		\mathscr{C}^{A\times B\times A}
		\ar[d,"(r_B)_!"]
		\ar[r,"\simeq"]
		&\mathscr{C}^{A\times A\times B}
		\ar[r,"\Delta_A^*"]
		\ar[d,"(r_B)_!"]
		&\mathscr{C}^{A\times B}
		\ar[r,"(r_A)_!"]
		\ar[d,"(r_B)_!"]
		&\mathscr{C}^{B}
		\ar[d,"(r_B)_!"]
		\\
		\mathscr{C}^{A\times A}
		\ar[r,"="]
		&
		\mathscr{C}^{A\times A}
		\ar[r,"\Delta_A^*"]
		&
		\mathscr{C}^{A}\ar[r,"(r_A)_!"]
		&
		\mathscr{C}^{\ast}
	\end{tikzcd}
	\caption{Rotator}\label{fig:shadow_indexed}
\end{figure}

The content of this theorem is that the isomorphisms $\alpha$, $\lambda$, $\rho$, and $\theta$ are coherent. For $\mc R/\cat{Top}$, this is true by rigidity, but it turns out the isomorphisms are coherent even when we don't have such a rigidity theorem.

Applying \autoref{thm:SMBF_to_bicategory} to the point-set category $\Osp$ of parametrized orthogonal spectra over all base spaces, gives the \textbf{point-set bicategory of parametrized spectra} $\Osp/\cat{Top}$. The 0-cells are spaces, 1-cells from $A$ to $B$ are parametrized orthogonal spectra over $A \times B$, and 2-cells are maps in $\Osp(A \times B)$. The product, unit, and shadow are defined just as we did for $\mc R/\cat{Top}$. By the rigidity theorem, since we know that isomorphisms $\alpha$, $\lambda$, $\rho$, and $\theta$ exist, they are also unique.\footnote{If we use (CGWH) then not all of these maps are isomorphisms, so we have to restrict to the subcategory of freely $i$-cofibrant spectra, or perhaps the further to the subcategory of freely $f$-cofibrant spectra. This does not introduce any difficulty in the rest of this section, so we won't comment on it again.}

\begin{rmk}
	If $K \in \mc R(A \times B)$ is a retractive space and $X \in \Osp(B \times C)$ is a parametrized spectrum, we can also define a product $K \odot X \in \Osp(A \times C)$ by following the above recipe and using the smash product with a space operation $K \barsmash X$ for the external smash product. This is uniquely naturally isomorphic to $\Sigma^\infty K \odot X$. Under this description, the monoidal fibrant replacement functor $P$ on both spaces and spectra is given by the composition product $PX \cong B^I_{+(B \times B)} \odot X$.
\end{rmk}

Applying \autoref{thm:SMBF_to_bicategory} to the homotopy category $\ho\Osp$ gives the \textbf{homotopy bicategory of parametrized spectra} $\ho \Osp/\cat{Top}$, or $\Ex$ for short. In this bicategory, the 0-cells are again topological spaces and the 1-cells are again orthogonal spectra over $A \times B$, but the 2-cells are maps in the homotopy category $\ho \Osp(A \times B)$ with the stable equivalences inverted.\index{bicategory!of parametrized spectra $\Ex$}

The description just given of $\Ex$ is the standard one in the literature, see \cite{ms,shulman_framed_monoidal}. We can now give three alternative, simplified definitions and show they are canonically isomorphic to this standard one. As in \autoref{lem:two_smashes_over_b_isomorphic}, there are four ways to define $\odot$, $\shad{}$, and the rest of the bicategory structure on the homotopy category:
\begin{enumerate}
	\item Define $\odot$ and $\shad{}$, and $U_A$ as composites of left- and right-deformable functors
	\[ X \odot^\M Y := \L(r_B)_!\R(\Delta_B)^*(X \barsmash^\L Y), \qquad \M\shad{X} := \L(r_A)_!\R(\Delta_A)^*X, \]
	\[ U_A = \Sigma^\infty_{+(A \times A)} A = (\Delta_A)_!(r_A)^*\Sph \simeq \L(\Delta_A)_!\R(r_A)^*\Sph. \]
	Define the isomorphisms $\alpha$, $\lambda$, $\rho$, $\theta$ from the point-set level isomorphisms in the bicategory $\Osp/\cat{Top}$, using \autoref{prop:passing_natural_trans_to_derived_functors}. The relevant composites of functors are coherently deformable, using the subcategory of freely $f$-cofibrant, level $h$-fibrant spectra in the source category, and all of its images.\footnote{Note that the coherent deformability condition is slightly complicated to check for the unit isomorphism $U_A \odot X \cong X$. $X$ is freely $f$-cofibrant and level $h$-fibrant, but $U_A \barsmash X$ is only freely $f$-cofibrant. However the map
	\[ (\Delta_A)^*(U_A \barsmash X) \to \R(\Delta_A)^*(U_A \barsmash X) := (\Delta_A)^*(PU_A \barsmash X) \]
	is isomorphic at each spectrum level to
	\[ A \times_A X_n \to A^I \times_A X_n \]
	and is therefore an equivalence.}
	\item Restrict $\Osp/\cat{Top}$ to the freely $f$-cofibrant spectra. Then the resulting bicategory is right-deformable, i.e. the compositions of $\odot$ and $\shad{}$ are coherently right-deformable (see \autoref{cofibrant_SMBF} and \autoref{rmk:left_deformable_bicategory}), for instance by applying $P$ to the inputs.\index{right-deformable!bicategory} Therefore the isomorphisms between their composites on the point-set level pass to the homotopy category.
	\item Restrict each of the categories $\Osp(A \times B)$ to the subcategory $\Osp(A\times B)^{cfu}$ containing the freely $f$-cofibrant level $h$-fibrant spectra, along with the unit $U_A$ if $A = B$. Using \autoref{prop:stably_derived},
	 the operation $\odot$ preserves these objects and weak equivalences between them. Therefore it passes directly to operations and coherent isomorphisms between them on the homotopy categories $\ho\Osp(A \times B)^{cfu}$, giving a bicategory.\footnote{While this recipe is the simplest, we actually don't think of it as giving a shadowed bicategory because $\shad{U_A}$ has the wrong homotopy type. If we really want to use shadows with this recipe, we can, at the cost of removing the strict units $U_A$.}
	\item As above, apply \autoref{thm:SMBF_to_bicategory} to the homotopy category $\ho\Osp$. In other words, pass $\barsmash$, $f^*$, and $f_!$ and the isomorphisms between them to the homotopy category before assembling them together into $\odot$, $U_A$, and $\shad{}$.
\end{enumerate}

\begin{thm}\label{thm:four_bicategories_of_spectra}
	These give canonically isomorphic shadowed bicategory structures on the homotopy categories $\ho\Osp(A \times B)$.
\end{thm}

In other words, we can pass from the SMBF to the bicategory before or after taking homotopy categories, and the results are canonically isomorphic.

\begin{rmk}
	Description (4) has been the most common definition of $\Ex$ in the literature so far, but (2) and (3) are often more useful on a technical level. For instance, using (3) it is straightforward to prove that $\ho\Osp/\cat{Top}$ has a ``shadowed $n$-Fuller structure'' as defined in \cite{mp1} -- such a structure exists on the point-set category, using the rigidity theorem to check all the coherence conditions. It then passes to the homotopy category because all of the composites of functors that appear are coherently right-deformable. A careful check of this will appear in \cite{mp4}.
	
	This technique should also allow us to prove that $\Ex$ it is a symmetric monoidal bicategory, even though the list of axioms for this structure is incredibly long \cite[4.4-4.8]{stay_compact_closed}. Unfortunately (3) has a drawback that it doesn't extend well to base-change objects, where (2) is better.
\end{rmk}

\begin{proof} of \autoref{thm:four_bicategories_of_spectra}.
	The canonical isomorphisms $\odot \cong \odot'$ and $\shad{} \cong \shad{}'$ arise because in every recipe $\odot$ and $\shad{}$ are defined as composites of deformations of the same functors. So the content of this theorem is that these canonical isomorphisms commute with the associator, unitor, and rotator isomorphisms. When comparing recipes (1), (2), and (3) this is true essentially by the definition of passing an isomorphism of functors to an isomorphism between its derived functors, \autoref{prop:passing_natural_trans_to_derived_functors}. So it remains to compare (2) to (4) on the subcategory of freely $f$-cofibrant spectra.
	
	Examine \autoref{fig:assoc_indexed}, \autoref{fig:unit_indexed}, and \autoref{fig:shadow_indexed}. We derive each of the functors along each little edge by right-deriving the pullbacks with $P$, and leaving the smash products and pushforwards alone. The isomorphism furnished by (2) says, restrict to fibrant inputs at the top-left, then remove the $P$s, and use the unique point-set isomorphism between the outside routes of the diagram on this input. This large isomorphism is the composite of the isomorphisms for the small square and triangular regions arising from the SMBF $\Osp^c$. So we could get the same isomorphism by leaving the $P$s in, removing and then re-inserting them every time we traverse one square or triangle using them. In other words, we compose the right-derived isomorphisms in each little region. By the proofs of \autoref{prop:deform_bifibration} and \autoref{prop:deform_SMBF}, these right-derived isomorphisms agree with the isomorphisms coming from the SMBF structure on $\ho\Osp^{c}$. This leaves us with the isomorphism from recipe (4).
	
	This argument runs beautifully for the associator and rotator but runs into a small hiccup for the unitor: at the bottom of the second column, in $\mathscr{C}^{A\times B\times B\times B}$, the image of our input is not fibrant, hence we are at risk of not being able to remove $P$. However, it comes from a fibrant object in the category $\mathscr{C}^{A \times B \times B}$ just above. Using the proof of \autoref{prop:deform_SMBF}, we get a commuting square in the homotopy category
	\[ \xymatrix @R=1.7em{
		(\Delta_B)_!(\Delta_B)^* \ar[d] \ar[r]^-\cong & (\Delta_B)^*(\Delta_B)_! \ar[d] \\
		(\Delta_B)_!\R(\Delta_B)^* \ar[r]^-\cong & \R(\Delta_B)^*(\Delta_B)_!
	} \]
	where the horizontal maps are Beck-Chevalley isomorphisms. On fibrant inputs, the left vertical is an isomorphism, hence the right vertical is as well. Therefore on $(\Delta_B)_!$ of a fibrant input, $(\Delta_B)^*$ is equivalent to its right-derived functor. (This is essentially a formalization of the observation we made when we defined recipe (1).) The proof can then proceed as in the other cases.
\end{proof}

\begin{rmk}
	The bicategory $\Ex$ is a framed bicategory in the sense of \cite{shulman_framed_monoidal}. Concretely, this means is a pseudofunctor from the category of spaces into $\Ex$, taking every map of spaces $f\colon A \to B$ to a 1-cell
	\[ \bcr{A}{}{B} := \Sigma^\infty_{+(A \times B)} A \]
	from $A$ to $B$. We call these ``base change'' 1-cells, because taking composition product $\odot$ with one has the effect of pulling back along $f$, or pushing forward, depending on which side we take the product. We also get canonical isomorphisms
\[ \bcr{A}{f}{B} \odot \bcr{B}{g}{C} \cong \bcr{A}{g \circ f}{C} \]
\[ \bcr{A}{\id}{A} \cong U_A \]
that satisfy the same coherences as for a monoidal functor. These arise from \autoref{fig:base_change_composition_indexed}, which is formally similar to the definition of the unitor isomorphism from \autoref{fig:unit_indexed}. In fact, we get such base-change 1-cells for any bicategory arising from an SMBF as in \autoref{thm:SMBF_to_bicategory}, but we do not give a proof of this fact here.
\end{rmk}

\begin{figure}[h]
	\begin{tikzcd}[column sep=50pt,row sep=30pt]
		{\ast}
		\ar[d,"I \times I"]
		\ar[rd,"I"]
		\\
		{\mathscr{C}^{\ast} \times \mathscr{C}^{\ast}}
		\ar[r,"{\boxtimes}"]
		\ar[d,"r_A^* \times r_B^*"]
		& \mathscr{C}^{\ast}
		\ar[d,"r_{A \times B}^*"]
		\ar[rd,"r_A^*"]
		\\
		{\mathscr{C}^{A} \times \mathscr{C}^{B}}
		\ar[r,"{\boxtimes}"]
		\ar[d,"{(1,f)_! \times (1,g)_!}"]
		&\mathscr{C}^{A\times B}
		\ar[r,"{(1,f)^*}"]
		\ar[d,"{((1,f)\times(1,g))_!}"]
		&\mathscr{C}^{A}
		\ar[rd,"{(1,g \circ f)_!}"]
		\ar[d,"{(1,f,g \circ f)_!}"]
		\\
		{\mathscr{C}^{A \times B} \times \mathscr{C}^{B \times C}}
		\ar[r,"\boxtimes"]
		&\mathscr{C}^{A\times B \times B\times C}
		\ar[r,"(1\Delta_B1)^*"]
		&\mathscr{C}^{A\times B\times C}
		\ar[r,"(1r_B1)_!"]
		&\mathscr{C}^{A\times C}
	\end{tikzcd}
	\caption{Base change composition 
	}\label{fig:base_change_composition_indexed}
\end{figure}

We do point out that the recipes (1), (2)\index{right-deformable!base-change objects} and (4) from \autoref{thm:four_bicategories_of_spectra} extend in a natural way to base-change objects and their composition isomorphisms, by the same construction that we used for the unit isomorphism.\footnote{Recipe (3) fails because the base-change objects aren't fibrant, and if you include them, tensoring on one side has the effect of pushing forward fibrant spectra, which creates additional non-fibrant spectra.} The same proof then establishes

\begin{prop}\label{prop:base_change_all_isomorphic}
	These three pseudofunctors from spaces into $\Ex$ are canonically isomorphic.
\end{prop}

Specifically, along the canonical isomorphisms between the different models for $\odot$ and $U_A$, the composition and unit isomorphisms for these pseudofunctors agree with each other.

\beforesubsection
\subsection{The bicategory $\Ex_B$}\aftersubsection

If we fix a base $B$, then we can run the entirety of the previous section with $\Osp$ replaced by the category $\Osp_{(B)}$ constructed in \autoref{sec:Osp_B}. We first get a point-set bicategory $\Osp_{(B)}/\cat{Fib}_B$. The 0-cells are fibrations $E \to B$, and the 1- and 2-cells from $D$ to $E$ are the objects and morphisms, respectively of the category $\Osp(D \times_B E)$. The product of $X \in \mc R(D \times_B E)$ and $Y \in \mc R(E \times_B F)$ is defined by the formula
\[ X \odot_B Y := (r_E)_!(\Delta_E)^*(X \barsma{B} Y), \]
\[ \xymatrix @R=1.7em @C=4em{
	X \odot_B Y \ar[d] & \ar[l] \Delta_E^*(X \barsma{B} Y) \ar[d] \ar[r] & X \barsma{B} Y \ar[d] \\
	D \times_B F & \ar[l]_-{r_E} D \times_B E \times_B F \ar[r]^-{\Delta_E} & D \times_B E \times_B E \times_B F.
} \]
We sometimes call this $X \odot_B^E Y$ if the choice of $E$ is not clear. Similarly, the shadow and unit are
\[ \shad{X} := (r_E)_!(\Delta_E)^*X, \qquad U_E := (\Delta_E)_!(r_E)^* \Sph_B \cong \Sigma^\infty_{+(E \times_B E)} E. \]
It is helpful to imagine that these are the same constructions as before, carried out over every point of $B$. In particular, the shadow lands in spectra over $B$, and the shadow of $U_E$ is the fiberwise free loop space $\Sigma^\infty_{+B} \Lambda_B E$, where $\Lambda_B E$ consists of those loops in $\Lambda E$ that are contained in a single fiber over $B$. The same figures as before define associator, unitor, and rotator isomorphisms, and these in turn are unique by the rigidity theorem.

We then build the homotopy bicategory $\ho\Osp_{(B)}/\cat{Fib}_B$, sometimes abbreviated as $\Ex_B$, whose 0- and 1-cells are the same but whose 2-cells are morphisms in the homotopy categories $\ho\Osp(D \times_B E)$. There are again four ways to do this, which are canonically isomorphic, by the same proof as above. We also get a system of base-change 1-cells that agree along these canonical isomorphisms.

We conclude with a technical lemma about $\odot$ and $\odot_B$ at the space level that is often used in conjunction with recipe (2) to tell when the composition product agrees with the right-derived composition product.
\begin{lem}\label{lem:derived_circle_product}
	If $A, C, D$ are $h$-fibrant spaces over $B$, then for $X \in \mc R(A \times_B C)$ and $Y \in \mc R(C \times_B D)$ both $f$-cofibrant, the relative composition product $X \odot_B Y$ is $f$-cofibrant and preserves equivalences so long as $X \to C$ or $Y \to C$ is at least a $q$-fibration.
\end{lem}

\begin{proof}
	Applying $\Delta_C^*$ to the square defining the external smash product gives the pushout square
	\[ \xymatrix @R=1.7em{
		(A \times_B C \times_C Y) \cup_{A \times_B C \times_B D} (X \times_C C \times_B D) \ar[d] \ar[r] & X \times_C Y \ar[d] \\
		A \times_B C \times_B D \ar[r] & \Delta_C^*(X \barsmash Y).
	} \]
	The horizontal maps are $f$-cofibrations over $A \times_B C \times_B D$, so pushing forward to $A \times_B D$ will preserve equivalences, so we only have to focus on $\Delta_C^*(X \barsmash Y)$. Since the square is a homotopy pushout square, it suffices to check that the top two terms preserve equivalences. In the top-left, this is because $A \to B$ and $D \to B$ are fibrations. In the top-right, this uses the assumption that $X \to C$ or $Y \to C$ is a $q$-fibration.
\end{proof}

\begin{cor}
	For $X \in \mc R(A \times C)$ and $Y \in \mc R(C \times D)$ both $f$-cofibrant, the composition product $X \odot Y$ is $f$-cofibrant and preserves equivalences so long as $X \to C$ or $Y \to C$ is at least a $q$-fibration. The same applies when $X$ and $Y$ are spectra that are either freely or level $f$-cofibrant, and $X_n \to C$ or $Y_n \to C$ is a $q$-fibration for each $n$.
\end{cor}

\begin{proof}
	The space-level statement is immediate. For the spectrum-level statement we argue as in \autoref{prop:stably_derived}: the product $X \odot Y$ preserves level equivalences under this assumption, but up to level equivalence we can make $X$ and $Y$ freely $f$-cofibrant and level $h$-fibrant, at which point $X \odot Y$ preserves stable equivalences by \autoref{prop:stably_derived} and \autoref{prop:spectra_pushout_product}.
\end{proof}

\beforesubsection
\subsection{The Reidemeister trace}\label{sec:trace}\aftersubsection

We briefly explain how the framework developed in this paper gives a cleaner and more robust connection between the bicategory $\Ex$ and fixed-point theory. In particular, we get a new proof that the point-set formula for the Reidemeister trace from e.g. \cite{crabb_james} agrees with the bicategorical-trace definition of Ponto \cite{ponto_asterisque,ponto_shulman_general}. We only give summaries here -- see \cite{ponto_shulman_general} and 
\cite{malkiewich2019parametrized} for more details.

\begin{df}\label{def:dual_in_bicategory}
In a bicategory $\mathscr B$, a 1-cell $X$ from $A$ to $B$ is \textbf{dualizable over $B$} if there is another 1-cell $Y$ from $B$ to $A$ and maps
	\[ c\colon U_A \to X \odot Y \]
	\[ e\colon Y \odot X \to U_B \]
	such that the following two composites are identity maps.
	\[ X \cong U_A \odot X \to X \odot Y \odot X \to X \odot U_B \cong X \]
	\[ Y \cong Y \odot U_A \to Y \odot X \odot Y \to U_B \odot Y \cong Y \]
\end{df}\index{dualizable!in a bicategory}
In this case, the \textbf{trace}\index{trace!in a bicategory} of a map $f\colon X \to X$ is defined as the composite
\[ \xymatrix @R=0.5em @C=3em{
	\shad{U_A} \ar[r]^-{\shad{c}} & \shad{X \odot Y} \ar[r]^-{\shad{f \odot \id}} & \shad{X \odot Y} \ar@{<->}[r]^-{\cong} & \shad{Y \odot X} \ar[r]^-{\shad{e}} & \shad{U_B}.
} \]
More generally, if $Q$ is a 1-cell from $A$ to $A$, and $P$ goes from $B$ to $B$, the \textbf{twisted trace} of a map $f\colon Q \odot M \to M \odot P$ is the composite
\begin{equation}\label{twisted_trace}
\xymatrix @R=0.5em @C=3em{
	\shad{Q} \ar@{<->}[r]^-\cong & \shad{Q \odot U_A} \ar[r]^-{\shad{\id \odot c}} & \shad{Q \odot X \odot Y} \ar[r]^-{\shad{f \odot \id}} & \shad{X \odot P \odot Y} \ar@{<->}[r]^-\cong & \shad{Y \odot X \odot P} \ar[r]^-{\shad{e \odot \id}} & \shad{P}.
}
\end{equation}
For any space $X$, the parametrized spectrum $\bcr{X}{}{*} = \Sigma^\infty_{+(X \times *)} X$ gives a 1-cell in $\Ex$ from $X$ to $*$. It turns out this is dualizable over $X$ when $X$ is a finite cell complex, or more generally when $X$ is finitely dominated.

For any self-map $f\colon X \to X$ of a finitely dominated space $X$, by \autoref{prop:base_change_all_isomorphic} there is a canonical isomorphism of base-change objects
\[ \xymatrix{ \bcr{X}{}{*} \ar[r]^-\cong & \bcr{X}{}{X} \odot \bcr{X}{}{*} } \]
arising from the fact that the composite $X \overset{f}\to X \to *$ agrees with the unique map $X \to *$. The trace of this isomorphism becomes a map in the homotopy category
\[ \xymatrix{ \shad{\Sph} \ar[r]^-{R(f)} & \shad{\Sigma^\infty_{+(X \times X)} X^f} \simeq \Sigma^\infty_+ \Lambda^f X, } \]
from the sphere spectrum $\Sph$ to the suspension spectrum of the twisted free loop space $\Lambda^f X$. We call this map the \textbf{Reidemeister trace} of $f$.

In the special case that $X$ is a Euclidean neighborhood retract (ENR), embedded into $\R^n$ with neighborhood retract $p\colon N \to X$, the coevaluation and evaluation maps of \autoref{def:dual_in_bicategory} have explicit descriptions in terms of Pontryagin-Thom collapse maps. Results of this form are well-known, see for instance \cite[D]{klein2001dualizing}, \cite[18.5]{ms}, and \cite[2.4]{ww1}, and a version for the current framework is spelled out explicitly in \cite{malkiewich2019parametrized}. We therefore get a concrete, explicit formula for the Reidemeister trace, as the $n$-fold desuspension of the map of topological spaces
\begin{equation}\label{point_set_rf}
\xymatrix @R=0em{
		S^n \ar[r] & S^n_\epsilon \sma (\Lambda^f X)_+ \\
		v \ar@{|->}[r] & (v - f(p(v))) \sma \gamma_{f(p(v)),v}.
	}
\end{equation}
Here $\gamma_{f(p(v)),v}$ is defined for $v \in \R^n$ sufficiently close to a fixed point of $f$ in $X \subseteq \R^n$, by taking a straight-line path in $\R^n$ from $v$ to $f(p(v))$ and projecting back to $X$. See \autoref{fig:reidemeister_formula}.

\begin{figure}[h]
	\centering
	\def\svgwidth{0.85\columnwidth}
	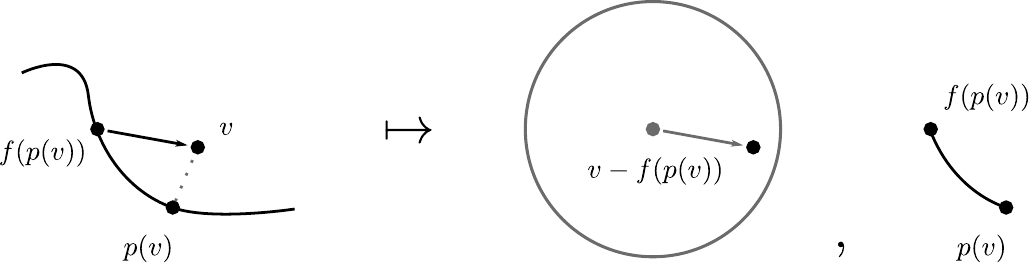
	\caption{A point-set formula for the Reidemeister trace.}\label{fig:reidemeister_formula}
\end{figure}

\begin{rmk}
	The agreement of classical indices like \eqref{point_set_rf} with the abstract trace of \eqref{twisted_trace} was previously carried out by comparing both to the Reidemeister trace as formed in chain complexes, see \cite{ponto_asterisque}. The foundations developed in this paper permit a more direct proof that does not pass through chain complexes: we simply observe that the products in \eqref{twisted_trace} can be formed in the homotopy category by forming them on the point-set level and making the inputs cofibrant and fibrant. Then we check that the resulting formula is \eqref{point_set_rf}. See \cite{malkiewich2019parametrized} for details.
\end{rmk}

\section{Genuine equivariance}\label{sec:G}

In this final section we describe how to insert a compact Lie group $G$ of equivariance everywhere. Almost every theorem remains true, with the same proof, but there are a few places where an extra argument is needed. We say very little about genuine fixed points, geometric fixed points, or multiplicative norm constructions, because these require additional arguments that go beyond the scope of this paper.

\beforesubsection
\subsection{Parametrized spaces}\aftersubsection

Let $G$ be a compact Lie group. A $G$-space is a space $B$ with a left action by $G$, and a map $f\colon A \to B$ of $G$-spaces is equivariant when it commutes with the action of each $g \in G$. We say the map is non-equivariant when the condition does not necessarily hold. A retractive $G$-space over $B$ is a $G$-space $X$ such that the inclusion and projection are both equivariant.

Let $G\mc R(B)$ be the category of retractive $G$-spaces over $B$. The morphisms are equivariant maps commuting with the inclusion and projection. We also have a larger category $G\mc R(B)^\non$, equivalent to $\mc R(B)$, with the same objects but where the morphisms are non-equivariant maps. The mapping spaces of $G\mc R(B)^\non$ have a $G$-action by conjugation,
\[ g(f) := g \circ f \circ g^{-1}, \]
and the $G$-fixed maps are precisely the equivariant ones. Unless otherwise noted, all results use the category $G\mc R(B)$ of equivariant maps.

By convention, we only consider closed subgroups of $G$. If $H \leq G$, the \textbf{$H$-fixed points} $X^H$ is the closed subspace of $X$ on which $hx = x$ for all $h \in H$. This is a retractive space over $B^H$, with an action by the Weyl group $WH = NH/H$. This defines a functor $G\mc R(B) \to WH\mc R(B^H)$.

When defining equivariant weak equivalences and $q$-fibrations, we ask that the equivariant map $f\colon X \to Y$ be a weak equivalence (resp. $q$-fibration) on the fixed points
\[ f^H: X^H \to Y^H \]
for all (closed) $H \leq G$.\index{$q$-fibration!equivariant} For equivariant $h$-fibrations, $h$-cofibrations, and $f$-cofibrations, we ask that the relevant retract or section be $G$-equivariant.\index{$h$-fibration!equivariant}\index{$h$-cofibration!equivariant}\index{$f$-cofibration!equivariant}

\begin{lem}
	The $H$-fixed point functor preserves every notion of equivalence, fibration, and cofibration.
\end{lem}

The remaining technical lemmas, \autoref{prop:technical_cofibrations} through \autoref{lem:closed_inclusion_equalizer}, hold for the equivariant versions of equivalence, fibration, and cofibration. In particular, equivariant $h$-fibrations are preserved by pushout along an equivariant $f$-cofibration (\autoref{prop:clapp}). The weak equivalences and $q$-fibrations form a Quillen model structure on $G\mc R(B)$, and the cofibrations are generated by maps of the form
	\[ \begin{array}{rlll}
	(G/H \times S^{n-1})_{+B} \to (G/H \times D^n)_{+B}, & n,k \geq 0, & H \leq G, & D^n \to B.
\end{array} \]

For an equivariant map $f\colon A \to B$ of base spaces, the pullback $f^*$, pushforward $f_!$, and sheafy pushforward $f_*$ are defined as in \autoref{sec:base_change}. They also define functors on the larger category of nonequivariant maps. On the category of equivariant maps, the Beck-Chevalley isomorphism is equivariant, the adjunction $(f_! \adj f^*)$ is Quillen, and this adjunction is a Quillen equivalence whenever $f\colon A \to B$ is an equivariant weak equivalence.

The base-change functors preserve equivariant cofibrations, fibrations, and weak equivalences under the same conditions as in \autoref{lem:f_star_preserves} and \autoref{lem:f_shriek_preserves}. But the proof of this requires the following additional lemma:
\begin{lem}\label{lem:commute_with_fixed_points}
	$f^*$ commutes with $H$-fixed points for $H \leq G$. $f_!$ commutes with $H$-fixed points (CGWH) always, (CG) on $i$-cofibrant spaces.
\end{lem}
The issue is that $H$-fixed points don't commute with all pushouts, only pushouts along closed inclusions. This is not something that switching between (CG) and (CGWH) solves. It happens even in simple examples like the pushout of
\[ \xymatrix{
	{*} & \ar[l] \Z/2 \ar[r] & {*}.
} \]
However, working in (CGWH) makes all of the basepoint sections $B \to X$ into closed inclusions, so that $f_!$ is always a pushout along a closed inclusion. For this reason and others, working in (CGWH) becomes more convenient in the equivariant theory than it was in the non-equivariant theory.

As a corollary, the monoidal fibrant replacement $PX$ has the same properties as in \autoref{prop:px_properties}. It commutes with $H$-fixed points under the same assumptions as in \autoref{lem:commute_with_fixed_points}.

We define $X \barsmash Y$ as in \autoref{def_external_smash}, and get an isomorphism of spaces over $A^H \times B^H$
\begin{equation}\label{eq:smash_fixed_points}
(X \barsmash Y)^H \cong X^H \barsmash Y^H
\end{equation}
for $i$-cofibrant $X$ and $Y$.

When defining $\barmap_B(Y,Z)$ we use the space of \emph{non-equivariant} maps, with $G$ acting by conjugation, so that $\barmap_B(Y,Z)$ is a $G$-equivariant retractive space over $A$ when $Y \in G\mc R(B)$ and $Z \in G\mc R(A \times B)$. When translating the proofs to the equivariant case, we remember that most of the maps that appear are equivariant, only the ones representing points in a mapping space $\Map(-,-)$ are not. In particular, equivariant maps of retractive spaces $Y' \to Y$ and $Z \to Z'$ induce an equivariant map
\[ \barmap_B(Y,Z) \to \barmap_B(Y',Z'), \]
and we get a bijection both on equivariant and on non-equivariant maps
\[ X \barsmash Y \to Z \quad \textup{ over }A \times B \quad \longleftrightarrow \quad X \to \barmap_B(Y,Z) \quad \textup{ over }A. \]
Letting $S^V$ denote the one-point compactification of an orthogonal $G$-representation $V$, this gives us fiberwise reduced suspension and based loop functors
\[ \Sigma^V_B X := S^V \barsmash X, \quad \Omega^V_B X := \barmap_*(S^V,X) \]
for $X \in G\mc R(B)$. Note that $\barmap_B(Y,Z)$ does not commute with fixed points, there is instead a restriction map
\[ \barmap_B(Y,Z)^H \to \barmap_{B^H}(Y^H,Z^H). \]

The pushout-product and pullback-hom interact with cofibrations, fibrations, and weak equivalences just as in \autoref{prop:h_cofibrations_pushout_product} and \autoref{h_fibrations_pullback_hom}. The relevant variant of \autoref{lem:pullback_hom_adjunction} asks for all the dotted maps to be non-equivariant and all the solid maps to be equivariant. Though, when proving that the pullback-hom preserves weak equivalences, without loss of generality we only have to consider its $G$-fixed points, and then all the maps become $G$-equivariant and the cells do not require an extra $G/H$.

The external smash product commutes with pullbacks and pushforwards as in \autoref{lem:external_smash_and_base_change}, so we get canonical isomorphisms $PX \barsmash PY \cong P(X \barsmash Y)$ as in \autoref{prop:P_strong_monoidal}.

The first significant change in the equivariant theory is that the functor $f_!g^*(X_1 \barsmash \ldots \barsmash X_n)$ is not rigid. A counterexample is given by any span of the form
\[ \xymatrix{ \Z/2 & \ar[l] (\Z/2) \amalg (\Z/2) \ar[r]^-\cong & \Z/2 \times \Z/2. } \]
However, we can organize the non-equivariant rigidity theorem, \autoref{prop:spaces_rigidity}, into the following statement.
\begin{prop}[Rigidity]\label{prop:Gspaces_rigidity}
Suppose $n \geq 0$ and we have maps of $G$-spaces
	\[ \xymatrix{ B & \ar[l]_-f A \ar[r]^-g & C_1 \times \ldots \times C_n } \]
	such that $(f,g)\colon A \to B \times C_1 \times \ldots \times C_n$ is injective.
	
	Then the functor $G\mc R(C_1) \times \ldots \times G\mc R(C_n) \to G\mc R(B)$ given by
	\[ \Phi\colon (X_1,\ldots,X_n) \leadsto f_!g^*(X_1 \barsmash \ldots \barsmash X_n) \]
	has only one automorphism that is natural with respect to the non-equivariant maps. I.e., only one automorphism of this functor extends to an automorphism of functors
	\[ G\mc R(C_1)^\non \times \ldots \times G\mc R(C_n)^\non \to G\mc R(B)^\non. \]\index{rigidity!for $G$-spaces}
\end{prop}

\begin{rmk}
	Suppose that we also take $G$-fixed points. Then the proof of \autoref{prop:spaces_rigidity} shows that the functor $f_!g^*(X_1 \barsmash \ldots \barsmash X_n)^G$ has a unique automorphism. In (CG), we restrict to $i$-cofibrant spaces so that we don't have to worry about whether we should take $(-)^G$ or $f_!$ first. This result and \autoref{prop:Gspaces_rigidity} can be generalized: for any $H \leq G$, the functor $f_!g^*(X_1 \barsmash \ldots \barsmash X_n)^H$ has only one automorphism that is natural with respect to the $H$-equivariant maps.
\end{rmk} 

\autoref{prop:Gspaces_rigidity} makes $G\mc R$ into a symmetric monoidal category in a canonical way.\footnote{Namely, we choose the associator to be the unique isomorphism that is natural with respect to the non-equivariant maps, and similarly for the other isomorphisms.} This makes $G \mc R$ into an SMBF over the category $G\cat{Top}$ of unbased $G$-spaces, with Beck-Chevalley for the strict pullback squares of $G$-spaces.

We define the internal smash product $X \sma_B Y$ as in \autoref{sec:internal}, making $G\mc R(B)$ into a symmetric monoidal category. In the proof of the third point of \autoref{cor:internal_smash_properties}, we reduce to the non-equivariant case, because all of the constructions commute up to homeomorphism with the fixed points $(-)^H$.

\autoref{sec:derived} is entirely formal, so it applies equally well to the equivariant case. The pushforward $f_!$ and external smash product $\barsmash$ are still left-deformable, while the pullback $f^*$ and external mapping space $\barmap_B(-,-)$ are right-deformable, and we get the same composites of coherently deformable functors as in the nonequivariant case.

\beforesubsection
\subsection{Parametrized spectra}\aftersubsection

In the interest of brevity we only discuss orthogonal spectra. A parametrized orthogonal $G$-spectrum is an object with left $G$-action, in the category of parametrized orthogonal spectra $\Osp$. Equivalently, it consists of a $G$-space $B$, regarded as a $G \times O(n)$-space for every $n$ by having $O(n)$ act trivially, a sequence of equivariant retractive spaces $X_n \in (G \times O(n))\mc R(B)$, and $G$-equivariant bonding maps
\[ \sigma\colon \Sigma_B X_n \to X_{1+n}, \]
such that each composite
\[ \sigma^p\colon S^p \barsmash X_q \to \ldots \to X_{p + q} \]
is $O(p) \times O(q)$-equivariant. Note that the $G$ and $O(n)$ actions on $X_n$ commute. For each $b \in B$ the fiber spectrum $X_b$ is usually not a $G$-spectrum, but rather a spectrum with an action of the stabilizer subgroup $G_b = \{ g \in G : gb = b \}$.\index{orthogonal spectra!equivariant category $G\Osp(B)$}

As in the non-parametrized case \cite[V.1.5]{mandell2002equivariant}, this is equivalent to giving a diagram over the category $\mathscr J_G$, with one object for every finite-dimensional $G$-representation $V \subset \mc U$ in a fixed complete $G$-universe $\mc U$. The morphism $G$-spaces $\mathscr J_G(V,W)$ are defined just as $\mathscr J(V,W)$, but using the \emph{non-equivariant} linear isometric maps $V \to W$. This space inherits a $G$-action by conjugation, making $\mathscr J_G$ into a category enriched in $G$-spaces.\index{orthogonal spectra!equivariant indexing category $\mathscr J_G$}

As in \cite[V.1.5]{mandell2002equivariant}, the category of $\mathscr J_G$-diagrams in $G\mc R(B)^\non$ is equivalent to the category of orthogonal $G$-spectra, by restricting $\mathscr J_G$ to the trivial representations. More concretely, $X(V)$ is isomorphic to $X_n$, with $G$ acting through a homomorphism $G \leq G \times O(n)$ that is the graph of some smooth homomorphism $\rho\colon G \to O(n)$. So $X(V)^H \cong X_n^\Gamma$ for some (closed) subgroup $\Gamma$ of such a graph. We call such a $\Gamma$ a \textbf{graph subgroup} of $G \times O(n)$. In other words, it is a subgroup whose elements are of the form $(g,\rho(g))$ for some fixed homomorphism $\rho\colon G \to O(n)$.\footnote{If we consider instead the larger class of subgroups that come from homomorphisms $H \to O(n)$, we are keeping track of representations that have an $H$-action that might not extend to a $G$-action. This eventually leads to the ``complete'' model structure on $G$-spectra from \cite[App B]{hhr}.}

As in \autoref{sec:spectra}, we can define the free spectrum $F_V A$ on a retractive $G$-space $A$ and a nontrivial $G$-representation $V$. At level $n$ this gives the retractive space $\mathscr J_G(V,\R^n) \barsmash A$. The space $\mathscr J_G(V,\R^n)$ is a $G \times O(n)$-CW complex by an application of Illman's triangulation theorem \cite{illman}.

A \textbf{level equivalence}\index{level equivalence!equivariant} is a map of spectra $X \to Y$ inducing an equivalence $X(V)^H \to Y(V)^H$ for all $G$-representations $V$ and all (closed) subgroups $H \leq G$. Similarly a \textbf{level $q$-fibration}\index{$q$-fibration!equivariant level} is a map for which $X(V)^H \to Y(V)^H$ is a Serre fibration for all $V$ and $H$. Note that $X \to Y$ is a level equivalence or $q$-fibration precisely when $X_n^\Gamma \to Y_n^\Gamma$ is an equivalence or $q$-fibration for every $n \geq 0$ and every graph subgroup $\Gamma \leq G \times O(n)$.

A \textbf{level $h$-fibration}\index{$h$-fibration!equivariant level} is a map for which $X_n \to Y_n$ is a $G \times O(n)$-equivariant $h$-fibration for all $n$. This is strictly stronger than asking for $X(V) \to Y(V)$ to be a $G$-equivariant $h$-fibration for all $V$. In a \textbf{level $f$-cofibration}\index{$f$-cofibration!equivariant level} or \textbf{level $h$-cofibration}\index{$h$-cofibration!equivariant level} we ask that $X_n \to Y_n$ is a $G \times O(n)$-equivariant $f$-cofibration or $h$-cofibration, respectively. Again, this implies that $X(V) \to Y(V)$ is a $G$-equivariant cofibration for all $V$.

Consider all maps of spectra over $B$ of the form $F_V(K \to X)$, where $K \to X$ is an equivariant $f$-cofibration of equivariantly $f$-cofibrant retractive $G$-spaces over $B$. The class of \textbf{free $f$-cofibrations}\index{$f$-cofibration!equivariant free} is the smallest class of maps of spectra containing the above, and closed under pushout, transfinite composition, and retracts. The \textbf{free $h$-cofibrations}\index{$h$-cofibration!equivariant free} are defined similarly, and for \textbf{free $q$-cofibrations}\index{$q$-cofibration!equivariant free} we ask for $K \to X$ to be a relative $G$-cell complex, so that the free $q$-cofibrations are generated by the maps of the form
\[ F_V((G/H \times S^{n-1})_{+B}) \to F_V((G/H \times D^n)_{+B}). \]

\begin{lem}
	Every free cofibration is a level cofibration. The free spectrum functor $F_V$ sends cofibrant spaces to freely cofibrant spectra. It sends $h$-cofibrant, $h$-fibrant spaces to level $h$-fibrant spectra.
\end{lem}
	
\begin{prop}[Level model structure]
	The level equivalences, free $q$-cofibrations, and level $q$-fibrations define a proper model structure on $G\Osp(B)$. It is cofibrantly generated by
\[ \resizebox{\textwidth}{!}{
$
\begin{array}{ccrllll}
I &=& \{ \ F_V\left[ (G/H \times S^{n-1})_{+B} \to (G/H \times D^n)_{+B} \right] & : n \geq 0, & V \subset \mc U, & H \leq G, & D^n \to B^H \ \} \\
J &=& \{ \ F_V\left[ (G/H \times D^n)_{+B} \to (G/H \times D^n \times I)_{+B} \right] &: n \geq 0, & V \subset \mc U, & H \leq G, & (D^n \times I) \to B^H \ \}.
\end{array}
$
}
\]
\end{prop}
	
The pullback $f^*$, pushforward $f_!$, and monoidal fibrant replacement functor $P$ lift to spectra as before by applying them to each level separately. They have the same properties as in \autoref{prop:spectrum_px} and \autoref{lem:spectrum_f_preserves}. We let $G\Osp$ denote the category of orthogonal $G$-spectra over all $G$-spaces, as in \autoref{all_spectra}.

The external smash product functors
\[ \xymatrix @R=0.3em{
	\barsmash\colon G\mc R(A) \times G\Osp(B) \to G\Osp(A \times B) \\
	\barsmash\colon G\Osp(A) \times G\Osp(B) \to G\Osp(A \times B)
} \]
and their right adjoints $\barF_A(K,-)$, $\barmap_B(X,-)$, and $\barF_A(K,-)$ have the same definitions as in \autoref{sec:smash_with_space}, and are unique up to canonical (not unique) isomorphism. When smashing two spectra, we may either left Kan extend along the direct sum map $\mathscr J_G \sma \mathscr J_G \to \mathscr J_G$, or take the non-equivariant smash product that extends along the map $\mathscr J \sma \mathscr J \to \mathscr J$, and give the result the diagonal $G$-action. Both approaches give canonically isomorphic results. The description using $\mathscr J_G$ is useful however for verifying that
\begin{equation}\label{eq:smash_free}
	F_V X \barsmash F_W Y \cong F_{V \oplus W} (X \barsmash Y)
\end{equation}
for nontrivial $G$-representations $V$ and $W$, as in \autoref{barsmash_free}.\footnote{We insist that $V \oplus W$ is a model of the direct sum of $V$ and $W$ that is a subspace of $\mc U$. The choice of subspace does not change the answer, up to canonical isomorphism.} As in \autoref{prop:spectra_external_smash_and_base_change}, the external smash product always commutes with pushforward, and commutes with pullback either in (CG) or in (CGWH) when the spectra are freely $i$-cofibrant.

As in \autoref{prop:Gspaces_rigidity}, the functor $f_!g^*(X_1 \barsmash \ldots \barsmash X_n)$ is not rigid, but it is rigid with respect to non-equivariant maps. The version for spectra is as follows, and has the same proof as \autoref{thm:spectra_rigidity}.
\begin{thm}[Rigidity]\label{thm:Gspectra_rigidity}
	Suppose $n \geq 0$ and we have maps of $G$-spaces $(f,g)$ as in \autoref{prop:Gspaces_rigidity}. Then any functor $G\Osp(C_1) \times \ldots \times G\Osp(C_n) \to G\Osp(B)$ isomorphic to
	\[ \Phi\colon (X_1,\ldots,X_n) \leadsto f_!g^*(X_1 \barsmash \ldots \barsmash X_n) \]
	has a canonical isomorphism to $\Phi$, the one that is natural on the larger category of $G$-spectra and non-equivariant maps $G\Osp(C_1)^\non \times \ldots \times G\Osp(C_n)^\non$.
	
	The same is true on any full subcategory of $\Osp(C_1) \times \ldots \times \Osp(C_n)$ containing all $n$-tuples of the form $(F_{V_1} (*_{+C_1}),\ldots, F_{V_n} (*_{+C_n}))$.\index{rigidity!for $G$-spectra}
\end{thm}
	
\beforesubsection
\subsection{Skeleta and semifree cofibrations}\label{sec:G_reedy}\aftersubsection

The equivariant version of \autoref{sec:reedy} requires some additional lemmas, primarily because the level equivalences are sensitive to the entire $G \times O(n)$-action at level $n$. As a result, we must make the equivariant version of a Reedy cofibration more sensitive to the $O(n)$-action.

Let $O$ be a compact Lie group, such as $O = O(n)$, and let $H$ be a (closed) subgroup. We give $O$ a left $H$-action with $h$ acting by $o \mapsto oh^{-1}$. For a left $H$-space $X$ we let $O \times_H X$ denote the quotient of $O \times X$ by the diagonal left $H$-action.
\begin{lem}\label{lem:induce_equivalences}
	The functor $O \times_H -$ sends $H$-equivalences of unbased spaces to $O$-equivalences.
\end{lem}
\begin{proof}
	It suffices to recall the formula for the $K$-fixed points of $O \times_H X$ for a closed subgroup $K \leq O$. By \cite[B.17]{schwede_global} with $G = H$ and $K = K$ applied to $O \times X$,
	\[ (O \times_H X)^K \cong \coprod_{(\alpha)} (O \times X)^\alpha/C(\alpha). \]
	Here the sum runs over conjugacy classes of homomorphisms $\alpha\colon K \to H$, $C(\alpha)$ is the centralizer of $\alpha(K)$ in $H$, and the $(-)^\alpha$ denotes points that are fixed by the subgroup of $K \times H$ that is the graph of $\alpha$. This subgroup is isomorphic to $K$, but acts on $O \times X$ by
	\[ k(o,x) = (ko\alpha(k)^{-1},\alpha(k)x). \]
	Therefore its fixed points are $Q_\alpha \times X^{\alpha(K)}$ where $Q_\alpha \leq O$ is the subgroup of $O$ fixed by the $K$-action $o \mapsto ko\alpha(k)^{-1}$. Since $C(\alpha) \leq H$ acted freely on $O$, its action on the subset $Q_\alpha$ is necessarily still free, so we get
	\[ (O \times_H X)^K \cong \coprod_{(\alpha)} Q_\alpha \times_{C(\alpha)} X^{\alpha(K)}. \]
	An $H$-equivalence in the $X$ variable induces an equivalence on all the subspaces $X^{\alpha(K)}$, which then become equivalences on the above terms by a simple induction on the free $C(\alpha)$-cells of $Q_\alpha$.
\end{proof}

\begin{lem}\label{lem:induce_fibrations}
	The functor $O \times_H -$ sends $H$-equivariant Hurewicz fibrations to $O$-equivariant Hurewicz fibrations.
\end{lem}
\begin{proof}
	Let $L$ be the tangent space at $1H \in O/H$ with the conjugation $H$-action, and $D(L)$ its unit disc. As in \cite[3.2.1]{schwede_global}, we can pick a smooth embedding $s\colon D(L) \to O$ satisfying $s(hl) = hs(l)h^{-1}$ for all $h \in H$, such that the map $D(L) \times H \to O$, $(l,h) \mapsto s(l)h$, is an embedding whose image is a tubular neighborhood of $H$ in $O$. Note that such a map respects the left and right $H$-multiplications in the following way.
	\[ \begin{array}{rcl}
	l,h &\mapsto & s(l)h \\
	h_1l,h_1h &\mapsto & h_1s(l)h \\
	l,hh_2 &\mapsto & s(l)hh_2
	\end{array} \]
	The map $(D(L) \times H) \times_H X \to O \times_H X$ is a closed inclusion by directly checking the topology. (It is also an open inclusion on the interior of $D(L)$.) This gives a subspace that is $H$-equivariantly homeomorphic to the product $D(L) \times X$ with the diagonal $H$-action.
	
	Since $D(L) \times -$ clearly preserves $H$-equivariant fibrations, this means that for an $H$-equivariant fibration $X \to Y$, $O \times_H Y$ can be covered by $H$-equivariant open sets on which $O \times_H X \to O \times_H Y$ is an $H$-equivariant fibration. Furthermore these sets are equivariantly numerable, meaning that the continuous function to $[0,1]$ that picks them out factors through the $H$-orbits. By \cite[3.2.4]{waner_fibrations}, or the $H$-equivariant version of the usual argument from e.g. \cite[7.4]{concise}, this means that the entire map $O \times_H X \to O \times_H Y$ is an $H$-equivariant fibration .
	
	To prove that it is an $O$-equivariant fibration, consider the path space $(O \times_H Y)^I$, and let $(O \times_H Y,Y)^{(I,0)}$ denote the $H$-invariant subspace of those paths starting in $H \times_H Y$. The $O$-action gives a continuous bijection
	\begin{equation}\label{eq:induce_paths}
		O \times_H (X \times_Y (O \times_H Y,Y)^{(I,0)}) \to (O \times_H X) \times_{(O \times_H Y)} (O \times_H Y)^I.
	\end{equation}
	If we restrict to $D(L) \times H$, the resulting map is isomorphic to
	\[ D(L) \times (X \times_Y (O \times_H Y,Y)^{(I,0)}) \to (D(L) \times X) \times_{(D(L) \times Y)} (O \times_H Y,D(L) \times Y)^{(I,0)} \]
	and given by the formula
	\[ (l,x,\gamma) \mapsto (l,x),(s(l)\gamma), \]
	so it has a continuous inverse $(l,x),\gamma \mapsto (l,x,s(l)^{-1}\gamma)$. Therefore the map \eqref{eq:induce_paths} is open in the interior of $D(L) \times H$. Translating by elements of $O$, we get an open cover on which the map is open, therefore \eqref{eq:induce_paths} is a homeomorphism.
	
	Now take an $H$-equivariant path-lifting function for $O \times_H X \to O \times_H Y$, restrict it to $X \times_Y (O \times_H Y,Y)^{(I,0)}$, and extend to an $O$-equivariant map
	\[ O \times_H (X \times_Y (O \times_H Y,Y)^{(I,0)}) \to (O \times_H X)^I. \]
	Along the homeomorphism \eqref{eq:induce_paths}, this gives an $O$-equivariant path-lifting function for $O \times_H X \to O \times_H Y$.
\end{proof}

The retractive space version of these lemmas follows almost immediately.
\begin{cor}\label{lem:induce_all}
	The functor $O_+ \barsmash_H (-)\colon H\mc R(B) \to O\mc R(B)$ preserves equivariant $h$- and $f$-cofibrations, and on equivariantly $h$-cofibrant spaces it preserves $h$-fibrations and equivalences.
\end{cor}
\begin{proof}
	Recall the pushout square
	\[ \xymatrix @R=1.7em{
		O/H \times B \ar[d] \ar[r] & O \times_H X \ar[d] \\
		B \ar[r] & O_+ \barsmash_H X.
	} \]
	This construction preserves cofibrations because the orbits commute with $- \times I$, so they preserve the retract that defines an $h$- or $f$-cofibration. On an $h$-cofibrant space, the horizontal maps of the above square are equivariant $h$-cofibrations. Since the term $O \times_H X$ preserves equivariant equivalences and fibrations by the previous two results, for $h$-cofibrant $X$ the pushout $O_+ \barsmash_H X$ does as well.
\end{proof}

From here we can follow the outline of \autoref{sec:reedy}. Only \autoref{smashing_with_free_complex} has a substantially different proof. Let $O$ and $O'$ be compact Lie groups, such as $O = O(n)$ and $O' = O(m)$ with $m \geq n$. They will always act trivially on the base spaces.

\begin{prop}\label{eq_smashing_with_free_complex}\hfill
	\vspace{-1em}
	
	\begin{itemize}
		\item If $f: K \to L$ is an $O$-free $O \times O'$-cell complex of based spaces and $g: X \to Y$ is a $G \times O$-equivariant $h$-cofibration of spaces over $B$ then $f \square g$ constructed by $(- \barsmash -)_{O}$ is a $G \times O'$-equivariant $h$-cofibration.
		\item The same is true for $f$-cofibrations or closed inclusions.
		\item If $L$ is any finite $O$-free based $O \times O'$-cell complex, and $X$ is any $G \times O$-equivariantly $h$-cofibrant and $h$-fibrant retractive space, then $(L \barsmash X)_O$ is $G \times O'$-equivariantly $h$-fibrant.
		\item If $L$ is any finite $O$-free based $O \times O'$-cell complex, and $g: X \to Y$ is any $G \times O$-equivariant equivalence of $G \times O$-$h$-cofibrant spaces over $B$, then $(\id_L \barsmash g)_{O}$ is a $G \times O'$-equivariant weak equivalence.
	\end{itemize}
\end{prop}

\begin{proof}
	An $O$-free $O \times O'$-cell complex is built out of cells of the form $(O \times O')/\rho_H \times D^n$, where $H \leq O'$ is closed, $\rho\colon H \to O$ is a homomorphism, and $\rho_H \leq O \times O'$ is the image of $H$ under the homomorphism $h \mapsto (\rho(h),h)$. So if $D$ is an unbased space with no group action, and $Y$ is a $G \times O$-equivariant $h$-cofibrant retractive space, we form the pushout square
	\[ \xymatrix @R=1.7em{
		O \times O' \times D \times B \ar[d] \ar[r] & O \times O' \times D \times Y \ar[d] \\
		B \ar[r] & (O \times O' \times D)_+ \barsmash Y.
	} \]
	in which the top horizontal is a $G \times O \times O'$-equivariant $h$-cofibration. Quotienting by the diagonal left $O$-action on the $O$ and $Y$, and the diagonal right $H$-action on $O$ (through $\rho$) and $O'$, gives a new pushout square along a $G \times O'$-equivariant $h$-cofibration
	\[ \xymatrix @R=1.7em{
		O'/H \times D \times B \ar[d] \ar[r] & D \times (O' \times_H Y) \ar[d] \\
		B \ar[r] & ((O \times O')/\rho_H \times D)_+ \barsmash_O Y
	} \]
	where the $H$-action on $Y$ in the upper-right term is through the homomorphism $\rho$. Comparing to the pushout squares
	\[ \xymatrix @R=1.7em{
		D \times O'/H \times B \ar[d] \ar[r] & D \times (O' \times_H Y) \ar[d] \\
		D \times B \ar[d] \ar[r] & D \times (O' \barsmash_H Y) \ar[d] \\
		B \ar[r] & D_+ \barsmash (O' \barsmash_H Y)
	} \]
	gives a natural isomorphism of $G \times O'$-equivariantly $h$-cofibrant spaces
	\[ ((O \times O')/\rho_H \times D)_+ \barsmash_O Y \cong D_+ \barsmash (O' \barsmash_H Y). \]
	We note that any $G \times O$-equivariant property of $Y$ becomes a $G \times H$-equivariant property along $\rho\colon H \to O$, and that
	\[ (G \times O') \barsmash_{G \times H} Y \cong O' \barsmash_H Y \]
	as $G \times O'$-spaces, hence this becomes a $G \times O'$-equivariant property of $O' \barsmash_H Y$ by \autoref{lem:induce_all}.
	
	The first two claims follow immediately from this analysis because as in \autoref{smashing_with_free_complex} if suffices to assume that $f$ is a single cell
	\[ [(O \times O')/\rho_H \times S^{n-1}]_+ \to [(O \times O')/\rho_H \times D^n]_+. \]
	For the fourth statement we induct on the skeleta of $L$. If $L$ is obtained from $K$ by pushout along a single cell, then $g$ gives a map of two pushout squares of the form
	\[ \xymatrix @R=1.7em{
		S^{d-1}_+ \barsmash O' \barsmash_H (-) \ar[r] \ar[d] & D^d_+ \barsmash O' \barsmash_H (-) \ar[d] \\
		(K \barsmash (-))_O \ar[r] & (L \barsmash (-))_O.
	} \]
	The horizontal maps are $G \times O'$-cofibrations by the first part, so these squares are $G \times O'$-equivariantly homotopy pushout squares. The map $g$ induces $G \times O'$-equivalences on the top-left and top-right by the above analysis, so we can do the induction and conclude $(\id_{L'} \barsmash f)_O$ is an equivalence for all skeleta $L'$ of $L$, and therefore for $L$ itself. The third statement is proven by the same induction, with just one pushout square of the above form for each cell of $L$.
\end{proof}

We define semifree spectra $\mc G_n A$ as in \autoref{sec:reedy}, where $A$ has a $G \times O(n)$-action. It turns out that a semifree spectrum on a nontrivial representation $V$ is always isomorphic to a semifree spectrum on a trivial representation, so it is unnecessary to define $\mc G_V A$ for general $V$. We then define skeleta and latching maps as in \autoref{sec:reedy}. The $n$th latching map $L_n X \to X_n$ is now a $G \times O(n)$-equivariant map.

\begin{df}
	A map of spectra $X \to Y$ over $B$ is a \textbf{semifree} or \textbf{Reedy $h$-cofibration} if each relative latching map $L_n Y \cup_{L_n X} X_n \to Y_n$ is a $G \times O(n)$-equivariant $h$-cofibration. A semifree $f$-cofibration is defined similarly.\index{$f$-cofibration!equivariant semifree}\index{$h$-cofibration!equivariant semifree}
\end{df}

As in \autoref{reedy_preserved}, the semifree cofibrations are closed under pushouts, compositions, sequential colimits, and retracts. Using the formula
	\[ F_V A \cong \mc G_n (\mathscr J_G(V,\R^n) \barsmash A) \]
for $n = \dim V$, we check that every generating free cofibration is a semifree cofibration. Therefore every free cofibration is also a semifree cofibration.

The remaining arguments in \autoref{sec:reedy} are unchanged, keeping in mind that the notion of cofibration, fibration, or weak equivalence at spectrum level $n$ is always the $G \times O(n)$-equivariant one. We conclude the equivariant version of \autoref{prop:reedy_pushout_product}. We will only record its specialization back to free spectra, the equivariant analogue of \autoref{prop:spectra_pushout_product}.

\begin{thm}\label{eq_prop:spectra_pushout_product}\hfill
	\vspace{-1em}
	
	\begin{enumerate}
		\item Let $f: K \to X$ and $g: L \to Y$ be free $h$-cofibrations of orthogonal $G$-spectra over $A$ and $B$, respectively. Then $f \square g$, constructed using $\barsmash$, is a free $h$-cofibration. The same is true for free $f$-cofibrations, free $q$-cofibrations, and free closed inclusions.
		\item If $X$ and $Y$ are freely $h$-cofibrant and level $h$-fibrant then $X \barsmash Y$ is level $h$-fibrant.
		\item If $X$ is freely $h$-cofibrant and $g: Y \to Y'$ is a level equivalence of freely $h$-cofibrant spectra then $\id_X \barsmash g$ is a level equivalence.
	\end{enumerate}
\end{thm}
	
\beforesubsection
\subsection{Stable equivalences}\aftersubsection

For each parametrized orthogonal $G$-spectrum $X \in G\Osp(B)$, the \textbf{homotopy groups}\index{homotopy!groups} are the genuine stable homotopy groups of the $G_b$-equivariant fiber spectrum $X_b$, for $b \in B$ and $n \in \Z$. For each (closed) subgroup $H \leq G_b$ these are defined as
\begin{equation}\label{eq:genuine_homotopy_groups}
\pi_{k,b}^H(X) := \begin{cases}\quad
\underset{V \subset \mathcal U}\colim\, \pi_k([\Omega^V X_b(V)]^H)  &\text{if } k \geq 0 \\
\underset{V \subset \mathcal U,\ \R^{|k|} \subset V}\colim\, \pi_0([\Omega^{V-\R^{|k|}} X_b(V)]^H)  &\text{if } k < 0.\end{cases}
\end{equation}
where $V$ runs over all $G_b$-representations in $\mc U$ and their inclusions in $\mc U$. On level $q$-fibrant spectra, every level equivalence induces an equivariant equivalence on the spaces $X_b(V)$, hence an isomorphism on these groups. So the right-derived homotopy groups $\R \pi_{n,b}(X)$ can be defined from these by replacing the spectrum $X$ by one that is level $q$-fibrant.

The stable equivalences are the maps inducing isomorphisms on these homotopy groups. These are generated under 2-out-of-3 by the level equivalences and the maps of level $q$-fibrant spectra that are stable equivalences on each fiber.\index{stable equivalence} The examples (\autoref{ex:equivalent_thom_spectra} through \autoref{parametrized_stability}) have the same properties in the equivariant case, using \cite[III.3]{mandell2002equivariant} in the place of \cite[\S 7]{mmss}.

Most of the proofs in this section use the formal behavior of the various kinds of cofibrations and fibrations, and so the equivariant versions have the same proofs, verbatim. One might worry about the $i$-cofibrancy hypothesis on commuting $f_!$, $\barsmash$, or $P$ with $H$-fixed points, but that was already used to establish the earlier lemmas concerning weak equivalences and cofibrant spaces, so it does not need to be invoked again.

\begin{thm}[Stable model structure]\label{eq_thm:stable_model_structure}
	There is a proper model structure on $G\Osp(B)$ whose weak equivalences are the stable equivalences, and is cofibrantly generated by the sets of maps
	\[ \resizebox{\textwidth}{!}{$
		\begin{array}{ccrllll}
	I &=& \{ \ F_V\left[ (G/H \times S^{n-1})_{+B} \to (G/H \times D^n)_{+B} \right] & : n \geq 0, & V \subset \mc U, & H \leq G, & D^n \to B^H \ \} \\
	J &=& \{ \ F_V\left[ (G/H \times D^n)_{+B} \to (G/H \times D^n \times I)_{+B} \right] &: n \geq 0, & V \subset \mc U,  & H \leq G, & (D^n \times I) \to B^H \ \} \\
	& \cup & \{ \ k_{V,W} \ \square \ \left[ (G/H \times S^{n-1})_{+B} \to (G/H \times D^n)_{+B} \right] & : n \geq 0, & V,W \subset \mc U, & H \leq G, & D^n \to B^H \ \}.
	\end{array}
	$} \]
\end{thm}\index{stable model structure}
The map $k_{V,W}$ is the inclusion of the front end of the mapping cylinder $\Cyl_{V,W}$ for the map
\[ \lambda_{V,W}\colon F_{V \oplus W} S^W \ra F_V S^0 \]
adjoint to the equivariant map of spaces $S^W \to \mathscr J_G(V,V \oplus W)$ that identifies $S^W$ with the fiber over the standard embedding $V \to V \oplus W$. As in \autoref{thm:stable_model_structure}, this is a free $q$-cofibration, and it is a stable equivalence by \cite[4.5]{mandell2002equivariant}. The analogue of \autoref{lem:stable_fibrations} is that the fibrations ($J$-injective maps) are the maps that give equivariant $q$-fibrations at every level $V$, and for which every one of the squares
\begin{equation}\label{eq_eq:fibration_of_spectra_means_this_is_a_pullback}\index{stably fibrant}
\xymatrix @R=1.7em{
	X_V \ar[r] \ar[d]^-{p_V} & \Omega^W_B X_{V\oplus W} \ar[d]^-{\Omega_B^W p_{V\oplus W}} \\
	Y_V \ar[r] & \Omega^W_B Y_{V\oplus W}
}
\end{equation}
\noindent
is an equivariant homotopy pullback square (i.e. a homotopy pullback on the $H$-fixed points for all $H \leq G$).

The proof of \autoref{eq_thm:stable_model_structure} is the same as \autoref{thm:stable_model_structure}, with the following modifications. When proving the smallness condition, we use the fact that equivariant maps out of $G/H \times D$ correspond to maps from $D$ into the $H$-fixed points, and that $H$-fixed points preserve pushouts and sequential colimits along levelwise $h$-cofibrations. When proving condition (6), we show that $F_b$ is a weak $\Omega$-spectrum in the sense that its adjoint structure maps are equivariant weak equivalences, before taking $H$-fixed points and concluding that the maps in the colimit system in \eqref{eq:genuine_homotopy_groups} are isomorphisms. Since the colimit is zero, the individual groups are also zero, hence $F_b$ is levelwise weakly contractible.

The remaining arguments from \autoref{sec:derived_bc_and_smash} and \autoref{sec:structures} are formal enough that no modifications are needed to the proofs. For simplicity we often restrict the category $G\Osp_{(B)}$ to base spaces $B$ with trivial $G$-action, but the proofs still work even if the $G$-action is nontrivial. In particular:

\begin{thm}\label{thm:G_spectra_SMBF}
	The homotopy category $\ho G\Osp$ of all parametrized orthogonal $G$-spectra forms a symmetric monoidal bifibration (SMBF) over the category $G\cat{Top}$ of unbased $G$-spaces, with Beck-Chevalley for each equivariant homotopy pullback square of $G$-spaces.
\end{thm}

\begin{proof}
	The proof is the same as in \autoref{thm:spectra_SMBF}. Note that this method is much cleaner than the one in \cite[Thm 9.9]{mp1}.
\end{proof}

\begin{thm}\label{thm:GEx}
	The homotopy categories $\ho G\Osp(A \times B)$ together form a bicategory with shadow $G\Ex$ in a canonical way.
\end{thm}

\begin{proof}
	The proof is the same as in \autoref{thm:four_bicategories_of_spectra}.
\end{proof}

We rely on this formulation of the foundations in our work in equivariant fixed point theory, see e.g. \cite{mp1}.


\bibliographystyle{amsalpha}
\bibliography{streamlined}{}

\end{document}